\documentclass[letterpaper, 10 pt, conference]{ieeeconf}  % Comment this line out if you need a4paper

\IEEEoverridecommandlockouts                              % This command is only needed if 
\usepackage{amsmath}
\usepackage{amsfonts}
\usepackage{amssymb}

\usepackage{amsthm}
\usepackage{xcolor}
\usepackage{url}

\usepackage{color}
\definecolor{darkred}{RGB}{150,0,0}
\definecolor{darkgreen}{RGB}{0,150,0}
\definecolor{darkblue}{RGB}{0,0,200}

\usepackage{upgreek}
\usepackage{graphicx}
\graphicspath{ {images/} }
\usepackage{pifont}
\usepackage{hyperref}
\hypersetup{colorlinks,linkcolor={blue},citecolor={blue},urlcolor={red}}
\usepackage{xstring}
\usepackage{multirow}
\usepackage{calc}

\usepackage{bbm}

\newtheorem{theorem}{Theorem}
\newtheorem{lemma}[theorem]{Lemma}

\newtheorem{fact}{Fact}
\newtheorem{definition}[theorem]{Definition}
\newtheorem{assumption}[theorem]{Assumption}

\newtheorem*{definition*}{Definition}
\newtheorem*{assumption*}{Assumption}
\newtheorem*{problem*}{Problem}
\newtheorem*{theorem*}{Theorem}
\newtheorem*{lemma*}{Lemma}
\newtheorem*{corollary*}{Corollary}
\newtheorem*{proposition*}{Proposition}
\newtheorem*{remark*}{Remark}

\let\labelindent\relax
\usepackage{enumitem}
\allowdisplaybreaks
\newcommand{\bv}[1]{\mathbf{#1}}		% Bold variable for English letters
\newcommand{\bvgrk}[1]{{\boldsymbol{#1}}}	% Bold variable for Greek letters

\newcommand{\varphihat}{\hat{\varphi}}
\newcommand{\varphis}{\varphi^\star}
\newcommand{\vu}{\bv{u}}

\newcommand{\vw}{\bv{w}}
\newcommand{\vx}{\bv{x}}
\newcommand{\vy}{\bv{y}}

\newcommand{\ub}{\bv{u}}
\newcommand{\wb}{\bv{w}}
\newcommand{\xb}{\bv{x}}

\newcommand{\xbb}{\bar{\bv{x}}}

\newcommand{\Jhat}{\hat{J}}

\newcommand{\Gcal}{\mathcal{G}}
\newcommand{\Hcal}{\mathcal{H}}

\newcommand{\Kcal}{\mathcal{K}}

\newcommand{\Mcal}{\mathcal{M}}

\newcommand{\Scal}{\mathcal{S}}
\newcommand{\Tcal}{\mathcal{T}}

\newcommand{\tn}[1]{\|{#1}\|_{\ell_2}}
\newcommand{\vA}{\bv{A}}
\newcommand{\vB}{\bv{B}}
\newcommand{\vC}{\bv{C}}

\newcommand{\vI}{\bv{I}}

\newcommand{\vK}{\bv{K}}
\newcommand{\vL}{\bv{L}}
\newcommand{\vM}{\bv{M}}
\newcommand{\vN}{\bv{N}}

\newcommand{\vP}{\bv{P}}
\newcommand{\vQ}{\bv{Q}}
\newcommand{\vR}{\bv{R}}
\newcommand{\vS}{\bv{S}}
\newcommand{\vT}{\bv{T}}
\newcommand{\vU}{\bv{U}}
\newcommand{\vV}{\bv{V}}

\newcommand{\vX}{\bv{X}}
\newcommand{\vY}{\bv{Y}}

\newcommand{\Bb}{\bv{B}}
\newcommand{\Cb}{\bv{C}}

\newcommand{\Ib}{\bv{I}}

\newcommand{\Kb}{\bv{K}}
\newcommand{\Lb}{\bv{L}}
\newcommand{\Mb}{\bv{M}}
\newcommand{\Nb}{\bv{N}}

\newcommand{\Pb}{\bv{P}}
\newcommand{\Qb}{\bv{Q}}
\newcommand{\Rb}{\bv{R}}
\newcommand{\Sb}{\bv{S}}

\newcommand{\Vb}{\bv{V}}

\newcommand{\Xb}{\bv{X}}

\newcommand{\vAhat}{\hat{\bv{A}}}
\newcommand{\vBhat}{\hat{\bv{B}}}
\newcommand{\vChat}{\hat{\bv{C}}}

\newcommand{\vKhat}{\hat{\bv{K}}}
\newcommand{\vLhat}{\hat{\bv{L}}}
\newcommand{\vMhat}{\hat{\bv{M}}}

\newcommand{\vPhat}{\hat{\bv{P}}}
\newcommand{\vQhat}{\hat{\bv{Q}}}

\newcommand{\vShat}{\hat{\bv{S}}}
\newcommand{\vThat}{\hat{\bv{T}}}

\newcommand{\vLtil}{\tilde{\bv{L}}}

\newcommand{\nec}[1]{\textcolor{red}{NO: #1}}

\newcommand{\order}[1]{{\cal{O}}(#1)}
\newcommand{\vGamma}{\bvgrk{\Gamma}}

\newcommand{\vDelta}{\bvgrk{\Delta}}

\newcommand{\vLambda}{\bvgrk{\Lambda}}

\newcommand{\vpi}{\bvgrk{\uppi}}

\newcommand{\vSigma}{\bvgrk{\Sigma}}

\newcommand{\vPhi}{\bvgrk{\Phi}}

\newcommand{\nn}{\nonumber}
\newcommand{\bm}[1]{\mathbf{#1}}

\newcommand{\mtx}[1]{\bm{#1}}
\newcommand{\bSi}{{\boldsymbol{{\Sigma}}}}
\newcommand{\E}{\operatorname{\mathbb{E}}}

\newcommand{\Tc}{\mathcal{T}}

\newcommand{\tf}[1]{\|{#1}\|_{F}}

\newcommand{\Nc}{\mathcal{N}}
\newcommand{\Iden}{{\mtx{I}}}
\newcommand{\Jbar}{\bar{J}}

\renewcommand{\P}{\operatorname{\mathbb{P}}}
\newcommand{\leqsym}[1]{\stackrel{\text{(#1)}}{\leq}}

\newcommand{\Abs}{\bv{A}^\star}
\newcommand{\Bbs}{\bv{B}^\star}

\newcommand{\Kbs}{\bv{K}^\star}
\newcommand{\Lbs}{\bv{L}^\star}

\newcommand{\Pbs}{\bv{P}^\star}

\newcommand{\Rbs}{\bv{R}^\star}

\newcommand{\Tbs}{\bv{T}^\star}

\newcommand{\Abst}{\bv{A}^{\star \T}}
\newcommand{\Bbst}{\bv{B}^{\star \T}}

\newcommand{\norm}[1]{\|{#1} \|}

\newcommand{\normbig}[1]{\left\|{#1} \right\|}

\newcommand{\curlybracketsbig}[1]{\left\{ #1 \right\}}
\newcommand{\curlybrackets}[1]{\{ #1 \}}
\newcommand{\squarebracketsbig}[1]{\left[ #1 \right]}
\newcommand{\squarebrackets}[1]{[ #1 ]}
\newcommand{\parenthesesbig}[1]{\left( #1 \right)}
\newcommand{\parentheses}[1]{( #1 )}

\newcommand{\indicator}{\mathbf{1}}

\newcommand{\R}{\mathbb{R}}				% Real line R
\newcommand{\dm}[2]
{
	\IfStrEq{#2}{1}{\R^{#1}}{\R^{#1 \x #2}}
}

\newcommand{\tr}{\textup{\textbf{tr}}} 			% Trace
\newcommand{\diag}{\textup{\textbf{diag}}} 		% Diag
\newcommand{\vek}{\textup{\textbf{vec}}}			% Vectorization
\newcommand{\expctn}{\mathbb{E}} 	 	% Expectation
\newcommand{\inv}{{-1}} 				% Superscript inverse
\newcommand{\distas}{\overset{\text{i.i.d.}}{\sim}}

\makeatletter
\newcommand*{\T}{{\mathpalette\@transpose{}}}
\newcommand*{\@transpose}[2]{\raisebox{\depth}{$\m@th#1\intercal$}}
\makeatother

\newcommand*{\x}{\mathsf{x}\mskip1mu} 	

\newcommand{\splitatcommas}[1]{%
	\begingroup
	\begingroup\lccode`~=`, \lowercase{\endgroup
		\edef~{\mathchar\the\mathcode`, \penalty0 \noexpand\hspace{0pt plus .1em}}%
	}\mathcode`,="8000 #1%
	\endgroup
}

\newcommand{\Item}[1]{%
	\ifx\relax#1\relax  \item \else \item[#1] \fi
	\abovedisplayskip=0pt\abovedisplayshortskip=0pt~\vspace*{-\baselineskip}
} 

\usepackage{todonotes}
\newcommand{\numSys}{s}
\newcommand{\dimSt}{n}

\newcommand{\NOMsys}{approximate }       % ? + system/MJS
\newcommand{\NOMparam}{approximate }     % ? + parameters
\newcommand{\NOMcDARE}{perturbed }       % ? + cDARE
\newcommand{\NOMcntrlr}{CE }             % ? + controller, CE stands for certainty equivalent control
\newcommand{\NOMsoltn}{perturbed }       % ? + solution

\title{\LARGE \bf
Certainty Equivalent Quadratic Control for Markov Jump Systems
}

\author{
Zhe Du$^*$$^{1}$ \quad\quad\quad\quad Yahya Sattar$^*$$^{2}$ \quad\quad Davoud Ataee Tarzanagh$^{1}$\\
Laura Balzano$^{1}$ 
\quad\quad Samet Oymak$^{2}$ 
\quad\quad Necmiye Ozay$^{1}$
\thanks{* Equal contribution.}
\thanks{$^{1}$ Department of Electrical Engineering and Computer Science, University of Michigan. Email: \{zhedu,tarzanaq,girasole,necmiye\}@umich.edu.}%
\thanks{$^{2}$ Department of Electrical and Computer Engineering, University of California, Riverside. Email: \{ysatt001,soymak\}@ucr.edu}%
}

\begin{document}

\maketitle
\thispagestyle{empty}
\pagestyle{empty}

%%%%%%%%%%%%%%%%%%%%%%%%%%%%%%%%%%%%%%%%%%%%%%%%%%%%%%%%%%%%%%%%%%%%%%%%%%%%%%%%
%{\color{brown} \textbf{Brown color: Copied directly from ICML draft, further changes needed}}

\begin{abstract}
Real-world control applications often involve complex dynamics subject to abrupt changes or variations. Markov jump linear systems (MJS) provide a rich framework for modeling such dynamics. Despite an extensive history, theoretical understanding of parameter sensitivities of MJS control is somewhat lacking. Motivated by this, we investigate robustness aspects of certainty equivalent model-based optimal control for MJS with quadratic cost function. Given the uncertainty in the system matrices and in the Markov transition matrix is bounded by $\epsilon$ and $\eta$ respectively, robustness results are established for (i) the solution to coupled Riccati equations and (ii) the optimal cost, by providing explicit perturbation bounds which decay as $\mathcal{O}(\epsilon + \eta)$ and $\mathcal{O}((\epsilon + \eta)^2)$ respectively.
\end{abstract}

% % The following four lines removes the spacing above and below equations to save space
% \setlength{\belowdisplayskip}{1pt} \setlength{\belowdisplayshortskip}{1pt}
% \setlength{\abovedisplayskip}{1pt} \setlength{\abovedisplayshortskip}{1pt}

\section{Introduction}\label{sec:intro}

The Linear Quadratic Regulator (LQR) is both theoretically well understood and commonly used in practice when the system dynamics are known. It also provides an interesting benchmark, when system dynamics are unknown, for reinforcement learning with continuous state and action spaces and for adaptive control \cite{campi1998adaptive, abbasi2011regret, dean2019sample, mania2019certainty, simchowitz2020naive, abeille2020efficient,lale2020explore}.

A natural generalization of linear dynamical systems is Markov jump linear systems (MJS) that allow the dynamics of the underlying system to switch between multiple linear systems according to an underlying finite Markov chain. Similarly, a natural generalization of LQR problem to MJS is to use mode-dependent cost matrices, which allows to have different control goals under different modes. While the optimal control for MJS-LQR is well understood when one has perfect knowledge of the system dynamics \cite{chizeck1986discrete,costa2006discrete}, in practice it may not be optimal due to the imperfect knowledge of the system dynamics and the transition matrix. {For instance, one might use system identification techniques to learn an approximate model for the system.} Designing optimal controllers for MJS-LQR with this \NOMsys system dynamics and transition matrix in place of the true ones leads to so-called certainty equivalent (CE) control which is used extensively in practice. However, a theoretical understanding of the suboptimality of the CE control for MJS-LQR is lacking. The main challenge here is the hybrid nature of the problem that requires consideration of both the system dynamics uncertainty $\epsilon$, and the underlying Markov transition matrix uncertainty $\eta$.

The solution of infinite horizon MJS-LQR involves coupled algebraic Riccati equations. Our goal is to understand how sensitive the solution of these equations and the corresponding optimal cost are to the perturbations in system model. To this aim, we first develop explicit $\order{\epsilon+\eta}$ perturbation bound for the solution to coupled algebraic Riccati equations that arise in the context of MJS-LQR. This in turn is used to establish explicit $\order{(\epsilon+\eta)^2}$ suboptimality bound. Finally, numerical experiments are provided to support our theoretical claims. Our proof strategy requires nontrivial advances over those of \cite{mania2019certainty,konstantinov1993perturbation}. Specifically, the coupled nature of Riccati equations requires novel perturbation arguments as these coupled equations lack some of the nice properties of the standard Riccati equations, like uniqueness of solutions under certain conditions or being amenable to matrix factorization based approaches.

\section{Related Work}
The performance analysis of CE control for the classical LQR problem for linear time invariant (LTI) systems relies on the perturbation/sensitivity analysis of the underlying algebraic Riccati equations (ARE), i.e. how much the ARE solution changes when the parameters in the equation are perturbed. This problem is studied in many works \cite{konstantinov2003perturbation}. Early results on ARE solution perturbation bound are presented in \cite{kenney1990sensitivity} (continuous-time) and \cite{konstantinov1993perturbation} (discrete-time). Most literature, however, only discusses perturbed solutions within the vicinity of the ground-truth solution. The uniqueness of such a perturbed solution is not discussed until \cite{sun1998perturbation}, which is further refined in \cite{sun2002condition} to provide explicit perturbation bounds and generalization to complex equations. Tighter bounds are obtained \cite{zhou2009perturbation} when the parameters have a special structure like sparsity. Channelled by ARE perturbation results, the end-to-end CE LQR control suboptimality bound in terms of dynamics perturbation is established in \cite{mania2019certainty}. The field of CE MJS-LQR control and corresponding coupled ARE (cARE) perturbation analysis, however, is very barren. To the best of our knowledge, there is no work on performance guarantee for CE MJS-LQR control, and the only two works \cite{konstantinov2002perturbation, konstantinov2003perturbation2} for cARE perturbation analysis only consider continuous-time cARE that arises in robust control applications, which is not applicable in MJS-LQR setting. Our work is also related to robust control for MJS (see, e.g., \cite{shi1999control,costa2006discrete}), where the focus is to numerically compute a controller to achieve a guaranteed cost under a given uncertainty bound. Whereas, we aim to characterize how the degradation in performance depends on perturbations in different parameters when CE control is used. Therefore, our work contributes to the body of work in robust control and CE control of MJS from a different perspective, and also paves the way to use these ideas in the context of learning-based control with performance guarantees.

\section{Preliminaries and Problem Setup}\label{sec:setup}
%\noindent\textbf{Notation:}
%\zhe{I did some notation changes: $\T$ for transpose, $\numSys$ for the number of modes, $\Tb$ for the Markov transition matrix, $\dimInput$ for the input dimension. }
% \subsection{Notations}

%\zhe{Remove unnecessary notations at last to save space.} 
We use boldface uppercase (lowercase) letters to denote matrices (vectors). For a matrix $\vV$, $\rho(\vV)$, $\underline{\sigma}(\vV)$, and $\norm{\vV}$ denote its spectral radius, smallest singular value, and spectral norm, respectively. We let $\norm{\vV}_+:= \norm{\vV}+1$. The Kronecker product of two matrices $\vM$ and $\vN$ is denoted as $\vM \otimes \vN$. $\vV_{1:\numSys}$ denotes a set of $s$ matrices $\curlybrackets{\vV_i}_{i=1}^\numSys$ of same dimensions. We use $\diag(\vV_{1:\numSys})$ to denote a block diagonal matrix whose $i$-th diagonal block is given by $\vV_i$. We define $[\numSys]:= \curlybrackets{1,2, \dots, \numSys}$, $\underline{\sigma}(\vV_{1:\numSys}):= \min_{i \in [\numSys]} \underline{\sigma}(\vV_i)$, $\norm{\vV_{1:\numSys}} := \max_{i \in [\numSys]} \norm{\vV_i}$, and $\norm{\vV_{1:\numSys}}_+ := \max_{i \in [\numSys]} \norm{\vV_i}_+$. We use $\alpha \vU_{1:\numSys} + \beta \vV_{1:\numSys}$ to denote $\curlybrackets{\alpha \vU_i + \beta \vV_i}_{i=1}^\numSys$.
We use $\lesssim$ and $\gtrsim$ for inequalities that hold up to a constant factor. 
%We use $\vV \succ (\succeq) 0$ to denote a symmetric and positive (semi)definite matrix. 
% Lastly, $\indicator_{x=i}$ outputs $1$ when $x=i$ and $0$ otherwise. 

\subsection{Markov Jump Systems}
We consider the problem of optimally controlling MJS, which are governed by the state equation,
%We consider MJS with the following dynamics (denoted by MJS($\vA^\star_{1:\numSys}, \vB^\star_{1:\numSys}, \vT^\star$)):
\begin{align}
	\vx_{t+1} = \Abs_{\omega(t)}\vx_t + \Bbs_{\omega(t)}\ub_t + \vw_t, \label{switched LDS}
\end{align}
where $\vx_t \in \R^n$, $\vu_t \in \R^p$ and $\vw_t\in \R^n$ denote the state, input~(or action) and noise at time $t$ respectively. Throughout, we assume $\vx_0 \sim \Nc(0,\Iden_n)$ and $\{\vw_t\}_{t=0}^{\infty} \distas \Nc(0,\sigma_w^2\Iden_n)$. There are $\numSys$ modes in total, and the dynamics of the $i$-th mode is given by $(\vA^\star_i, \vB^\star_i)$. The active mode at time $t$ is indexed by $\omega(t) \in [s]$. In MJS the mode sequence $\{\omega(t)\}_{t = 0}^{\infty}$ follows an ergodic Markov chain with transition matrix $\Tbs \in \R^{\numSys \times \numSys}$ such that for all $t \geq 0$, the $ij$-th element of $\Tbs$ denotes the conditional probability
\begin{equation}\label{eq_MarkovMtx}
	[\Tbs]_{ij} := \P\big(\omega(t+1) = j \mid \omega(t) = i \big), \forall i,j \in [\numSys].
\end{equation}
Due to ergodicity, for any initial distribution $\vpi_0 {\ \in \ } \dm{\numSys}{1}$ of $\omega(0)$, there exists a unique stationary distribution $\vpi_{\vT^\star} \in \dm{\numSys}{1}$ such that $(\vT^{\star \T})^t \vpi_0 \to \vpi_{\vT^\star}$ as $t \to \infty$.
Throughout, we assume the initial state $\vx_0$, Markov chain $\curlybrackets{\omega(t)}_{t=0}^\infty$, and noise $\curlybrackets{\vw_t}_{t=0}^\infty$ are mutually independent. 
%$\prob(\omega(t)=i)=\vpi_{\vT^\star}(i)$ as $t\rightarrow \infty$. 
%We let MJS($\vA^\star_{1:\numSys}, \vB^\star_{1:\numSys}, \vT^\star$) denote MJS with dynamics given by \eqref{switched LDS}.

For mode-dependent controller $\vK_{1:\numSys}$ that yields inputs $\vu_t {=} \vK_{\omega(t)} \vx_t$, we use $\vL_i{:=} \vA^\star_i + \vB^\star_i \vK_i$ to denote the closed-loop state matrix for mode $i$. We use $\vx_{t+1} {=} \vL_{\omega(t)} \vx_t$ to denote the noise-free autonomous MJS, either open-loop ($\vL_i {=} \Abs_i$) or closed-loop ($\vL_i {=} \vA^\star_i + \vB^\star_i \vK_i$). Due to the randomness in $\{\omega(t)\}_{t=0}^{\infty}$, it is common to consider the stability of MJS in the mean-square sense which is defined as follows.
%the stability considered for MJS is usually in the mean square sense which is defined as follows.
\begin{definition}\cite[Definitions 3.8, 3.40]{costa2006discrete}\label{def_mss}
	(a) We say MJS in \eqref{switched LDS} with $\vu_t = 0$ is mean square stable (MSS) if there exists $\vx_\infty, \vSigma_\infty$ such that for any initial state/mode $\vx_0, ~\omega(0)$, as $t \rightarrow \infty$, we have
	\begin{align}
		\norm{\expctn[\vx_t] - \vx_\infty} \to 0, \quad	\norm{\expctn[\vx_t \vx_t^\T] - \vSigma_\infty} \to 0.
	\end{align}
	In the noise-free case ($\vw_t=0$), we have $\vx_\infty=0$, $\vSigma_\infty=0$. (b) We say MJS in \eqref{switched LDS} with $\vw_t{=}0$ is (mean square) stabilizable if there exists mode-dependent controller $\Kb_{1:s}$ such that the closed-loop MJS $\vx_{t+1} = (\vA^\star_{\omega(t)} + \vB^\star_{\omega(t)} \vK_{\omega(t)}) \vx_t$ is MSS. We call such $\Kb_{1:s}$ a stabilizing controller.
\end{definition}
One can check the stabilizability of an MJS via linear matrix inequalities \cite[Proposition 3.42]{costa2006discrete}. It is well-known that the stability of non-switching systems is related to the spectral radius of the state matrix. Similarly, the mean-square stability of an autonomous MJS $\vx_{t+1} = \vL_{\omega(t)} \vx_t$ is related to the spectral radius of the augmented state matrix: $\vLtil \in \dm{\numSys \dimSt^2}{\numSys \dimSt^2}$ with $ij$-th $\dimSt^2 {\times} \dimSt^2$ block given by
\begin{equation}\label{def_augmentedL}
	\squarebrackets{\vLtil}_{ij} := [\vT^\star]_{ij} \vL^\T_i \otimes \vL^\T_i, \quad \forall~ i,j \in [s].
\end{equation}
Before stating a lemma to relate the MSS with the spectral radius of $\vLtil$, we define the operator,
\begin{equation}\label{eq_phistar}
	\varphis_i(\vV_{1:\numSys}) := \sum \nolimits_{j=1}^\numSys [\Tbs]_{ij} \vV_j, \quad \forall ~ i \in [\numSys].
\end{equation}
\begin{lemma} \cite[Theorem 3.9]{costa2006discrete}\label{prop MSS}
	The following are equivalent: (a) MJS $\xb_{t+1} = \Lb_{\omega(t)}\xb_t$ is MSS; (b) $\rho(\tilde{\Lb}) <1$; (c) \label{item_mssLyapunov} there exists $\vV_{1:\numSys}$ with $\vV_i \succ 0$, such that $\vV_i - \vL^\T_i \varphi^\star_i(\vV_{1:\numSys}) \Lb_i \succ 0, \quad \forall~ i \in [\numSys]$.
\end{lemma}
These assertions reduce to the classical stability results regarding spectral radius and Lyapunov equation %when there is a single mode 
when $\numSys=1$. Moreover, it can be shown that the augmented matrix $\vLtil^\T$ maps $\curlybrackets{\expctn[\vx_t \vx_t^\T \mathbf{1}_{\omega(t)=i}]}_{i=1}^\numSys$ to $\curlybrackets{\expctn[\vx_{t+1} \vx_{t+1}^\T \mathbf{1}_{\omega(t+1)=i}]}_{i=1}^\numSys$ \cite[p.35]{costa2006discrete}, hence its spectral radius determines MSS.

\subsection{Linear Quadratic Regulator}
The optimal control problem we consider in this paper is the following Markov jump system infinite-horizon linear quadratic regulator (MJS-LQR) problem:
\begin{equation}\label{LQR optimization}
	\begin{aligned}
		\inf_{\{\ub_0,\ub_1,\dots\}} &J(\vu_0, \vu_1, \dots) \\
		\text{s.t.} \quad \vx_{t+1} = &\Abs_{\omega(t)}\vx_t + \Bbs_{\omega(t)}\ub_t + \vw_t.
	\end{aligned}
\end{equation}
Here, we consider the long-term average quadratic cost
\begin{equation}\label{eq_lqrCostFunction}\hspace{-4pt}
	J(\vu_0{,} \vu_1{,} \dots \hspace{-1pt})
	{:=} \hspace{-1pt} \limsup_{T \to \infty} \hspace{-1pt} \expctn \hspace{-1pt} \bigg[\hspace{-2pt} \frac{1}{T} \hspace{-4pt} \sum_{t=0}^{T} \hspace{-2pt} \vx_t^\T\Qb_{\omega(t)}\vx_t {+} \ub_t^\T\Rb_{\omega(t)}\ub_t \hspace{-2pt} \bigg]
\end{equation}
where $\Qb_{\omega(t)}$ and $\Rb_{\omega(t)}$ are mode-dependent cost matrices chosen by users, and 
% $\vR_i$ being positive definite means all components of the control variable will be penalized. 
the expectation is over the randomness of initial state $\vx_0$, noise $\{\vw_t\}_{t = 0}^{\infty}$ and Markovian modes $\{\omega(t)\}_{t = 0}^{\infty}$. 
%The quadratic cost $\vx_t^\T\Qb_{\omega(t)}\vx_t$ usually represents the deviation of states, e.g. velocity, position, angle, from the desired value, whereas $\ub_t^\T\Rb_{\omega(t)}\ub_t$ represents the control effort, e.g. battery/fuel consumption.
Unlike classical LQR for LTI systems, where cost matrices are usually fixed throughout the time horizon, the mode-dependent cost matrices in MJS-LQR allows us to have different control goals under different modes. To guarantee MJS-LQR~\eqref{LQR optimization} is solvable, we assume the MJS in \eqref{switched LDS} and the cost matrices satisfy the following.
\begin{assumption} \label{asmp_1}
	(a) For all $i \in [s]$, the cost matrices $\Qb_{i}$ and $\Rb_{i}$ are positive definite; and (b) the MJS in~\eqref{switched LDS} with $\vw_t=0$ is stabilizable and each pair $(\vQ_i^{\frac{1}{2}}, \vA_i)$ is observable.
\end{assumption}
%\begin{assumption}\label{asmp_2}
%	The MJS in~\eqref{switched LDS} with $\vw_t=0$ is stabilizable and each pair $(\vQ_i^{\frac{1}{2}}, \vA_i)$ is observable.
%\end{assumption}
%Under this assumption, we know
%\ysatt{Lemma 4 is very long. Can we shorten it. e.g. I feel there is no need of stating the optimal cost. We are only concerned with suboptimality gap. Another thing $\vP^\star_{i}$ in the expression of optimal cost is the solution of cDARE or coupled Lyapunov equation?}

The following lemma characterizes some properties of the minimizer of \eqref{eq_lqrCostFunction}.
\begin{lemma} \cite[Theorem 4.6 and Corollary A.21]{costa2006discrete}\label{lemma_vanillaLQRResult}
	Under Assumption \ref{asmp_1} (a) and (b), for all $i \in [\numSys]$, the coupled discrete-time algebraic Riccati equations (cDARE)
	\begin{equation}\label{eq_CARE}
		\begin{aligned}
			\Xb_i &= \Abst_i \varphi^\star_i(\Xb_{1:\numSys}) \Abs_i + \Qb_i - \Abst_i \varphi^\star_i(\Xb_{1:\numSys}) \Bbs_i \\
			&\quad\big(\Rb_i + \Bbst_i \varphi^\star_i(\Xb_{1:\numSys}) \Bbs_i\big)^{-1} \Bbst_i \varphi^\star_i(\Xb_{1:\numSys}) \Abs_i
		\end{aligned}
	\end{equation}
	have a unique solution $\vP^\star_{1:\numSys}$ among $\curlybrackets{\vX_{1:\numSys}: \vX_i \succeq 0,~ \forall~ i}$, and $\vP^\star_i \succ 0$ for all $i \in [\numSys]$. Moreover, the mode-dependent state feedback controller
	\begin{equation}\label{eq_optfeedback}
		\Kbs_i = -\big(\Rb_i + \Bbst_i \varphi^\star_i(\Pbs_{1:\numSys}) \Bbs_i\big)^{-1} \Bbst_i \varphi^\star_i(\Pbs_{1:\numSys}) \Abs_i
	\end{equation}
	stabilizes the MJS in \eqref{switched LDS} and minimizes the cost \eqref{eq_lqrCostFunction} with input $\vu_t = \vu_t^\star := \Kbs_{\omega(t)} \vx_t$ and optimal cost $J^\star := J(\vu^\star_0, \vu^\star_1, \dots) = \sigma_w^2 \sum_{i=1}^\numSys  \tr ( \vpi_{\Tbs}(i)\vP^\star_{i} )$.
\end{lemma}
In Assumption \ref{asmp_1}, $\vR_i \succ 0$ means every component of control variable will be penalized. With $(\vQ_i^{\frac{1}{2}}, \vA_i)$ being observable, $\vQ_i \succ 0$ is not mandatory for Lemma \ref{lemma_vanillaLQRResult} to hold. This condition is needed mainly for our  subsequent theoretical developments.

% \begin{remark}
% \nec{@Zhe, can you put something that emphasizes solution P is not unique in general as in ARE but positive definite solution is unique}
% \end{remark}

In the remaining paper, we let cDARE$(\vA^\star_{1:\numSys}, \vB^\star_{1:\numSys}, \vT^\star)$ denote equations with structure given by \eqref{eq_CARE}, where $\vT^\star$ determines the operator $\varphi^\star$. In practice, cDARE can be solved efficiently either with LMI or recursively ~\cite{costa2006discrete}.

\subsection{Certainty Equivalent Controller}
In this work we seek to control MJS with unknown dynamics $(\Abs_{1:s},\Bbs_{1:s},\Tbs)$ based on \NOMparam parameters $(\vAhat_{1:s},\vBhat_{1:s},\vThat)$. The cost matrices $(\Qb_{1:s},\Rb_{1:s})$  are assumed known and the modes $\{\omega(t)\}_{t = 0}^{\infty}$ are observed at run-time. We analyze the CE approach, that is, using \NOMparam parameters $(\vAhat_{1:s},\vBhat_{1:s},\vThat)$, we solve the \NOMcDARE cDARE$(\vAhat_{1:s}, \vBhat_{1:s},\vThat)$,
\begin{equation}
	\begin{aligned}\label{eq_CARE_perturbed}
		\Xb_i &= \vAhat_i^\T \varphihat_i(\Xb_{1:\numSys}) \vAhat_i + \Qb_i - \vAhat_i^\T \varphihat_i(\Xb_{1:\numSys}) \vBhat_i \\
		&\quad\big(\Rb_i + \vBhat_i^\T \varphihat_i(\Xb_{1:\numSys}) \vBhat_i\big)^{-1} \vBhat_i^\T \varphihat_i(\Xb_{1:\numSys}) \vAhat_i, 
	\end{aligned}
\end{equation}
for all $i \in [s]$ and $\Xb_i \succeq 0$, where the operator $\varphihat$ is defined as
\begin{equation}
	\varphihat_i(\Vb_{1:\numSys}) := \sum \nolimits_{j = 1}^{\numSys} [\vThat]_{ij} \Vb_j.
\end{equation}
Let $\vPhat_{1:\numSys}$ be the positive definite solution of \eqref{eq_CARE_perturbed}, then the \NOMcntrlr controller is given by
\begin{equation}\label{eq_15}
	\vKhat_i = -\big(\Rb_i + \vBhat_i^\T \hat{\varphi}_i(\vPhat_{1:\numSys}) \vBhat_i\big)^{-1} \vBhat_i^\T \hat{\varphi}_i(\vPhat_{1:\numSys}) \vAhat_i,
\end{equation}
for all $i \in [s]$. Lastly, we apply the input $\hat{\ub}_t = \vKhat_{\omega(t)}\xb_t$ to control the true MJS. Let $\Jhat:= J(\hat{\vu}_0, \hat{\vu}_1, \dots)$ be the cost incurred by playing the CE controller on the true MJS. Throughout we assume that for all $i \in [s]$, the \NOMparam parameters have the following accuracy levels:
\begin{equation}
	\norm{\Abs_{i} -  \vAhat_{i}} \leq  \epsilon,~ \norm{\Bbs_{i} - \vBhat_{i}} \leq \epsilon,~ \norm{\Tbs - \vThat}_\infty  \leq  \eta. \label{eqn:uncertainty bound}
\end{equation}
In the next section, we address the following questions: (a) When can the perturbed cDARE in~\eqref{eq_CARE_perturbed} be guaranteed to have a unique positive semi-definite solution $\vPhat_{1:\numSys}$? (b) What is a tight upper bound on $\norm{\vPhat_{1:\numSys} - \Pbs_{1:\numSys}}$? (c) When does $\vKhat_{1:\numSys}$ stabilize the true MJS? (d) How large is the suboptimality gap $\Jhat - J^\star$? 
%\begin{enumerate}
%	\item When is the solution $\vPhat_{1:\numSys}$ to the nominal cDARE$(\vAhat_{1:\numSys},\vBhat_{1:\numSys},\vThat)$ guaranteed to uniquely exist?
%	\item How much is $\norm{\vPhat_{1:\numSys} - \vP_{1:\numSys}}$?
%	\item Can the nominal controllers $\vKhat_{1:\numSys}$ stabilize the true MJS?
%	\item How much is the suboptimality $\Jhat - J^\star$?
%\end{enumerate}
%It is possible that the solution $\vPhat_{1:\numSys}$ whose deviation we bound in the second question is not the solution we obtain practice. This is why we also require solution uniqueness in addition to existence. The first two questions will be answered in Section \ref{subsec_cDAREPerturb}, and the second two in Section \ref{subsec_costPerturb}.

\section{Perturbation Analysis for MJS-LQR}\label{sec:main}
Before we formally state our results, we introduce a few more concepts and assumptions. We use $\vL^\star_i: = \vA^\star_i + \vB^\star_i \vK^\star_i$ to denote the closed-loop state matrix under the optimal MJS-LQR controller \eqref{eq_optfeedback}, and define the augmented state matrix $\vLtil^\star$ similar to \eqref{def_augmentedL} such that its $ij$-th block is given by
\begin{equation}\label{def_augmentedLstar}
	\squarebrackets{\vLtil^\star}_{ij} := [\vT^\star]_{ij} {\vL^\star_i}^\T \otimes {\vL^\star_i}^\T.
\end{equation}
From Lemma \ref{lemma_vanillaLQRResult}, we know the closed-loop MJS $\vx_{t+1}=\vL^\star_{\omega(t)} \vx_t$ is MSS, thus $\rho(\vLtil^\star)<1$ by Lemma \ref{prop MSS}. Furthermore, we define the following to quantify the decay of $\vLtil^\star$.
\begin{definition} \label{def_tau}
	For an arbitrary $\gamma \in [\rho(\vLtil^\star), 1)$, we define
	\begin{equation}
		\tau(\vLtil^\star, \gamma) := \sup_{k\in \mathbb{N}} \curlybrackets{\norm{(\vLtil^\star)^k}/\gamma^k}.
	\end{equation}
\end{definition}
Note that $\tau(\vLtil^\star, \gamma)$ is finite by Gelfand's formula. It is easy to see that $\tau(\vLtil^\star, \gamma)$ monotonically decreases with $\gamma$ and that $\tau(\vLtil^\star, \gamma) \geq 1$. This quantity measures the transient response of a non-switching system with state matrix $\vLtil^\star$ and can be upper bounded by its $\mathcal{H}_\infty$ norm \cite{tu2017non}.

Finally, for the ease of exposition, we define a few constants: 
\begin{equation}
	\begin{split}
		\xi &:=\min \curlybrackets{\norm{\Bbs_{1:\numSys}}_+^{-2} \norm{\vR_{1:\numSys}^\inv}_+^{-1}
			\norm{\Lbs_{1:\numSys}}_+^{-2}, \underline{\sigma}(\Pbs_{1:s}) }, \\
		C_{\epsilon} &:= \norm{\Abs_{1:\numSys}}_+^2 \norm{\Bbs_{1:\numSys}}_+ \norm{\Pbs_{1:\numSys}}_+^2 \norm{\vR_{1:\numSys}^\inv}_+, \\
		C_{\epsilon}^u &:= C_{\epsilon}^\inv \norm{\Bbs_{1:\numSys}}_+^{-2} \norm{\Pbs_{1:\numSys}}_+^{-1} \norm{\vR_{1:\numSys}^\inv}_+^{-1}, \\
		C_{\eta} &:= \norm{\Abs_{1:\numSys}}_+^2 \norm{\Bbs_{1:\numSys}}_+^4 \norm{\Pbs_{1:\numSys}}_+^3 \norm{\vR_{1:\numSys}^\inv}_+^2, \\
		C_{\eta}^u &:= C_{\eta}^\inv,\\
		\Gamma_\star &:= \max \curlybrackets{\norm{\vA^\star_{1:\numSys}}_+, \norm{\vB^\star_{1:\numSys}}_+, \norm{\vP^\star_{1:\numSys}}_+, \norm{\vK^\star_{1:\numSys}}_+}.
	\end{split}
\end{equation}
%\subsection{cDARE Perturbation Theory for MJS}\label{subsec_cDAREPerturb}

In the following, we will show that despite being coupled, cDARE for MJS-LQR satisfies nice properties. To be more precise, we show that if the \NOMsys MJS is accurate enough, i.e., $\epsilon$ and $\eta$ are sufficiently small, we can guarantee that not only the positive definite solution $\vPhat_{1:\numSys}$ to the \NOMcDARE cDARE uniquely exists, but also $\vPhat_{1:\numSys}$ does not not deviate much from $\vP^\star_{1:\numSys}$.
\begin{theorem}\label{prop_riccatiPertb} Let $\epsilon,\eta \geq 0$ be as in~\eqref{eqn:uncertainty bound}. Under Assumptions~\ref{asmp_1}, and as long as 
	$
	\epsilon \leq 
	\min \curlybracketsbig{ \frac{C_{\epsilon}^u \xi \parentheses{1-\gamma}^2}{204 \dimSt \numSys \tau(\tilde{\Lb}^\star, \gamma)^2}, \ \norm{\Bbs_{1:s}}, \  \underline{\sigma}(\vQ_{1:s}) }
	$,
	$
	\eta \leq 
	\frac{C_{\eta}^u \xi \parentheses{1-\gamma}^2}{48 \dimSt \numSys \tau(\tilde{\Lb}^\star, \gamma)^2}
	$,
	% 	\begin{align}
	% 	\epsilon &\leq 
	% 	\min \curlybracketsbig{ \frac{C_{\epsilon}^u \xi \parentheses{1-\gamma}^2}{204 \dimSt \numSys \tau(\tilde{\Lb}^\star, \gamma)^2}, \ \norm{\Bbs_{1:s}}, \  \underline{\sigma}(\vQ_{1:s}) } \label{eq_60},\\	
	% 	\eta 
	% 	&\leq 
	% 	\frac{C_{\eta}^u \xi \parentheses{1-\gamma}^2}{48 \dimSt \numSys \tau(\tilde{\Lb}^\star, \gamma)^2}  \label{eq_61}
	% 	\end{align}
	the \NOMcDARE cDARE in \eqref{eq_CARE_perturbed} is guaranteed to have a unique solution $\vPhat_{1:s}$ in $\curlybrackets{\vX_{1:\numSys}: \vX_i \succeq 0, ~\forall i}$ such that $\vPhat_i \succ 0$ for all $i$ and
	\begin{equation}\label{eq_PPerturbResult}
		\norm{\vPhat_{1:\numSys} - \Pbs_{1:\numSys}} 
		\leq
		\frac{\sqrt{\dimSt \numSys} \tau(\tilde{\Lb}^\star, \gamma)}{1-\gamma}
		\parentheses{6 C_{\epsilon} \epsilon + 2 C_\eta \eta }.
	\end{equation}	
\end{theorem}
In this theorem, the choice of parameter $\gamma$ balances the numerator and the denominator of $\frac{\tau(\tilde{\Lb}^\star, \gamma)}{1-\gamma}$. From the constants, we see we would have milder requirement on $\epsilon$ and $\eta$ and tighter bound on $\norm{\vPhat_{1:\numSys} - \Pbs_{1:\numSys}}$ when (i) $\norm{\vA^\star_{1:\numSys}}, \norm{\vB^\star_{1:\numSys}}$, (ii) $ \norm{\vL^\star_{1:\numSys}}$, $\tau(\tilde{\Lb}^\star, \gamma)$, and (iii) $\norm{\vR^\inv_{1:\numSys}}$ are smaller. These translate to the cases when (i) the true MJS is easier to stabilize; (ii) the closed-loop MJS under the optimal controller is more stable; and (iii) the input dominates more in the cost function. 
The role of $\tau(\tilde{\Lb}^\star, \gamma)$ in this theorem is closely related to the damping property in ARE perturbation analysis \cite{kenney1990sensitivity}. Coefficients for $\epsilon$ and $\eta$ on the RHS of \eqref{eq_PPerturbResult} are also known as condition numbers in algebraic Riccati equation sensitivity literature \cite{sun2002condition}.
The uniqueness result in Theorem \ref{prop_riccatiPertb} guarantees the perturbation bound \eqref{eq_PPerturbResult} indeed applies to the perturbed solution one would obtain in practice.
% The \NOMcntrlr controller $\vKhat_{1:\numSys}$ stabilizes the \NOMsys MJS $(\vAhat_{1:s},\vBhat_{1:s},\vThat)$, hence the upper bounds on $\epsilon$ and $\eta$ also provides a stability margin type result for the optimal LQR controllers \cite{shaked1986guaranteed} for MJS.
% {\color{red}stability margin should probably come after Theorem 7??}
Using similar proof techniques, the exact same theorem can be proved for cases when approximate $\vQhat_{1:\numSys}$ with $\norm{\vQhat_{1:\numSys} {-} \vQ_{1:\numSys}} {\leq} \epsilon$ is used in place of $\vQ_{1:\numSys}$ in the computations, which can be useful when the cost for $\vx_t$ is in the form of $\norm{\vy_t}^2$ where $\vy_t {=} \vC_{\omega(t)} \vx_t$ represents observation, and we only have approximate parameter $\vChat_{1:\numSys}$. In this case, $\vQ_i {=} \vC^\T_i \vC_i$ and $\vQhat_i {=} \vChat^\T_i \vChat_i$.

Finally, using Theorem~\ref{prop_riccatiPertb}, we quantify the mismatch between the performance of the optimal controller $\Kbs_{1:s}$ and the certainty equivalent controller $\vKhat_i$. We then quantify the suboptimality gap $\Jhat - J^\star$ in terms of the controller mismatch and derive an upper bound on this mismatch so that the certainty equivalent controller $\vKhat_i$ stabilizes the MJS in the mean-square sense. This leads to the following suboptimality result.
% \nec{There was a question on how to makes sure if the perturbed Markov chain is ergodic or not (and whether it is needed). Has this been addressed?}

%{\color{blue} Maybe change the above paragraph to: Combining Theorem~\ref{prop_riccatiPertb} with (i) perturbation results on controllers with respect to Riccati equation solutions and (ii) perturbation results on the costs with respect to controllers, we bound the suboptimality $\Jhat - J^\star$ in terms of the uncertainty $\epsilon$ and $\eta$ as follows.}

\begin{theorem}\label{thrm:meta plus perturbation}
	Let $\epsilon,\eta \geq 0$ be as in~\eqref{eqn:uncertainty bound}. Suppose Assumptions~\ref{asmp_1} (a) and (b) hold and additionally the Markov chain $\{\omega(t)\}_{t=0}^\infty$ is ergodic. Then, as long as $\epsilon$ and $\eta$ are such that the upper bounds in Theorem \ref{prop_riccatiPertb} are valid and $C_\epsilon \epsilon + C_\eta \eta \lesssim \frac{(1-\gamma)^2\underline{\sigma}(\Rbs_i)^2}{s\sqrt{s n} \Gamma_\star^6(\underline{\sigma}(\Rbs_i)+\Gamma_\star^3)\tau(\tilde{\Lb}^\star,\gamma)^2}$, the CE controller $\vKhat_{1:s}$ stabilizes the true MJS and we have
	\begin{align}\label{eq_suboptimality}
		\nonumber 
		\Jhat - J^\star& \lesssim \sigma_w^2s^2 n \min\{n,p\}(\norm{\Rbs_{1:s}}+\Gamma_\star^3) \Gamma_\star^6 \\
		&\frac{\tau(\vLtil^\star,\gamma)^3(\underline{\sigma}(\Rbs_{1:s})+\Gamma_\star^3)^2}{(1-\gamma)^3\underline{\sigma}(\Rbs_{1:s})^4}\parentheses{6 C_{\epsilon} \epsilon + 2 C_\eta \eta }^2.
		%\end{aligned}
	\end{align}
	%as long as $\epsilon$ and $\eta$ are sufficiently small such that the upper bounds in Theorem \ref{prop_riccatiPertb} are valid. %, and (ii) the right hand side of \eqref{eq_suboptimality} is smaller than $\sigma_w^2$.
\end{theorem}
Our result states that the suboptimality decays as the square of the size of uncertainty. 
Also note that the \NOMcntrlr controller $\vKhat_{1:\numSys}$ stabilizes both the \NOMsys MJS $(\vAhat_{1:s},\vBhat_{1:s},\vThat)$ and MJS $(\vA^\star_{1:s},\vB^\star_{1:s},\vT^\star)$, hence the upper bounds on $\epsilon$ and $\eta$ also provides a stability margin type result for the optimal LQR controllers \cite{shaked1986guaranteed} for MJS.

To see the significance of our bound, suppose the uncertainty in the system dynamics and the transition matrix are due to the estimation error induced by a system identification procedure that uses $T$ samples. Then, if the estimation error decays as $\order{1/\sqrt{T}}$, Theorem~\ref{thrm:meta plus perturbation} states that the suboptimality decays as $\order{1/T}$ which is similar to what is known in the case of LDS~\cite{mania2019certainty} except that we suffer from an additional $s$ and $n$ dependence in the numerator because of multiple modes in MJS. 
% {\color{blue} 
More concretely, having estimates twice as good will reduce the suboptimality to one fourth, which can be achieved by doubling the data if a system identification algorithm with sample complexity $\order{1/\sqrt{T}}$ is used. If we look at the dependence of suboptimality gap on the system properties, our result states that more modes $s$, larger system order $n$ and larger noise variance $\sigma_w$ adversely affect the gap.

\section{Numerical Experiments}\label{sec:exp}

In this section, we present some numerical results to support our proposed theory. All of the synthesis and performance experiments are run in MATLAB. %and our scripts for generating the examples will be made available.
\begin{figure*}[t]
	\begin{center}
		\includegraphics[scale=0.086]{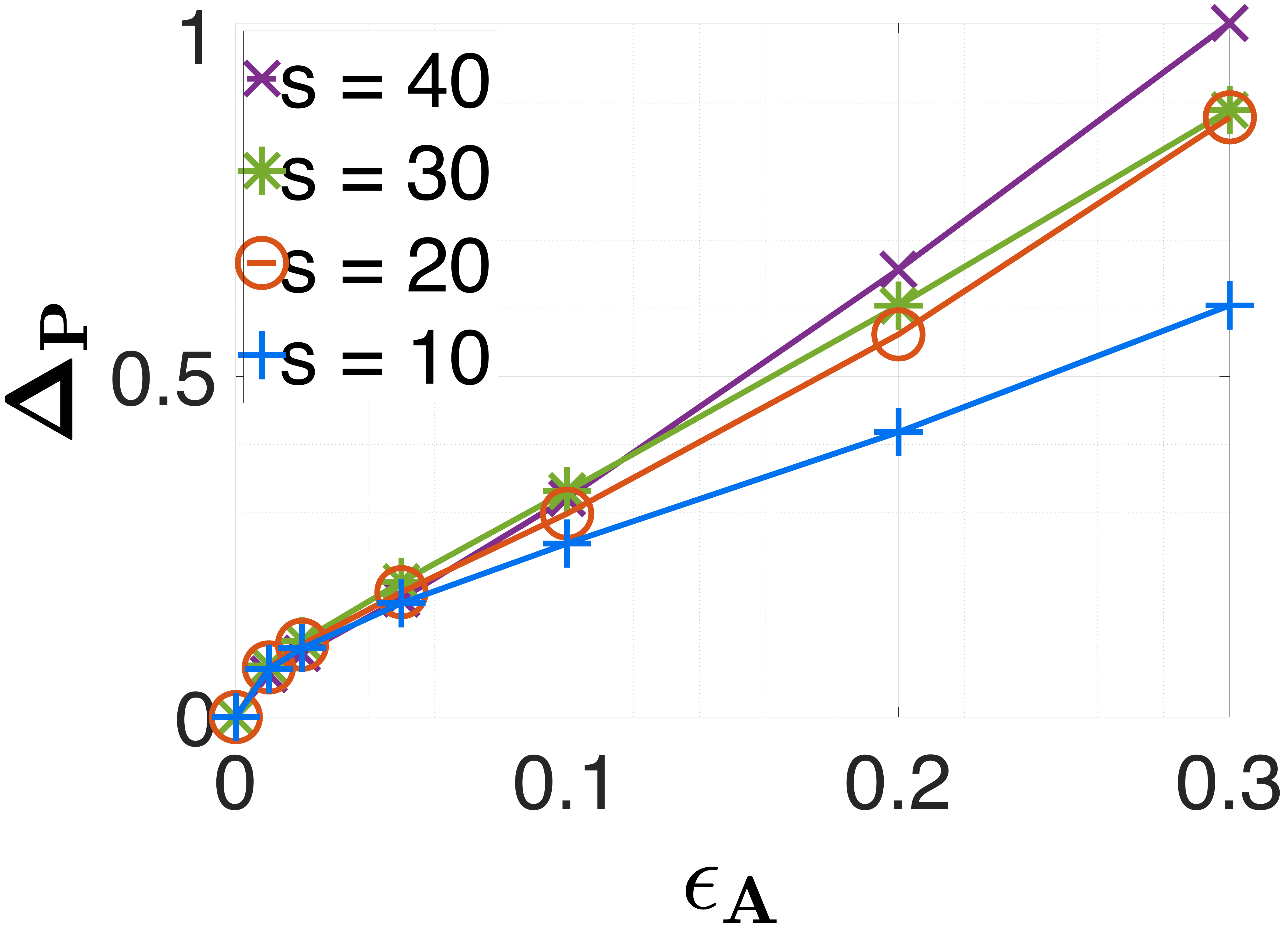}
		\hspace{.6em}
		\includegraphics[scale=0.086]{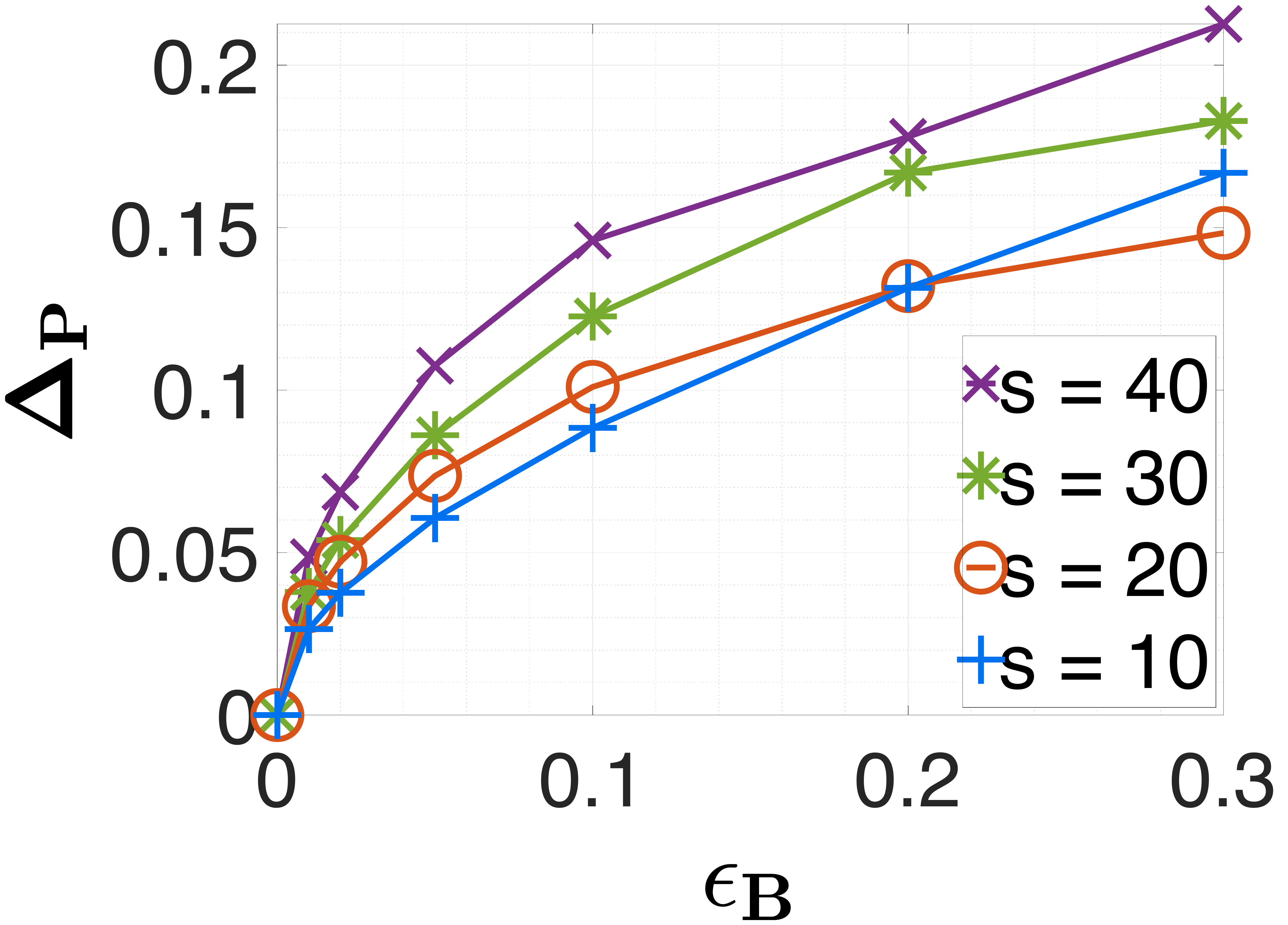}
		\hspace{.6em}
		\includegraphics[scale=0.086]{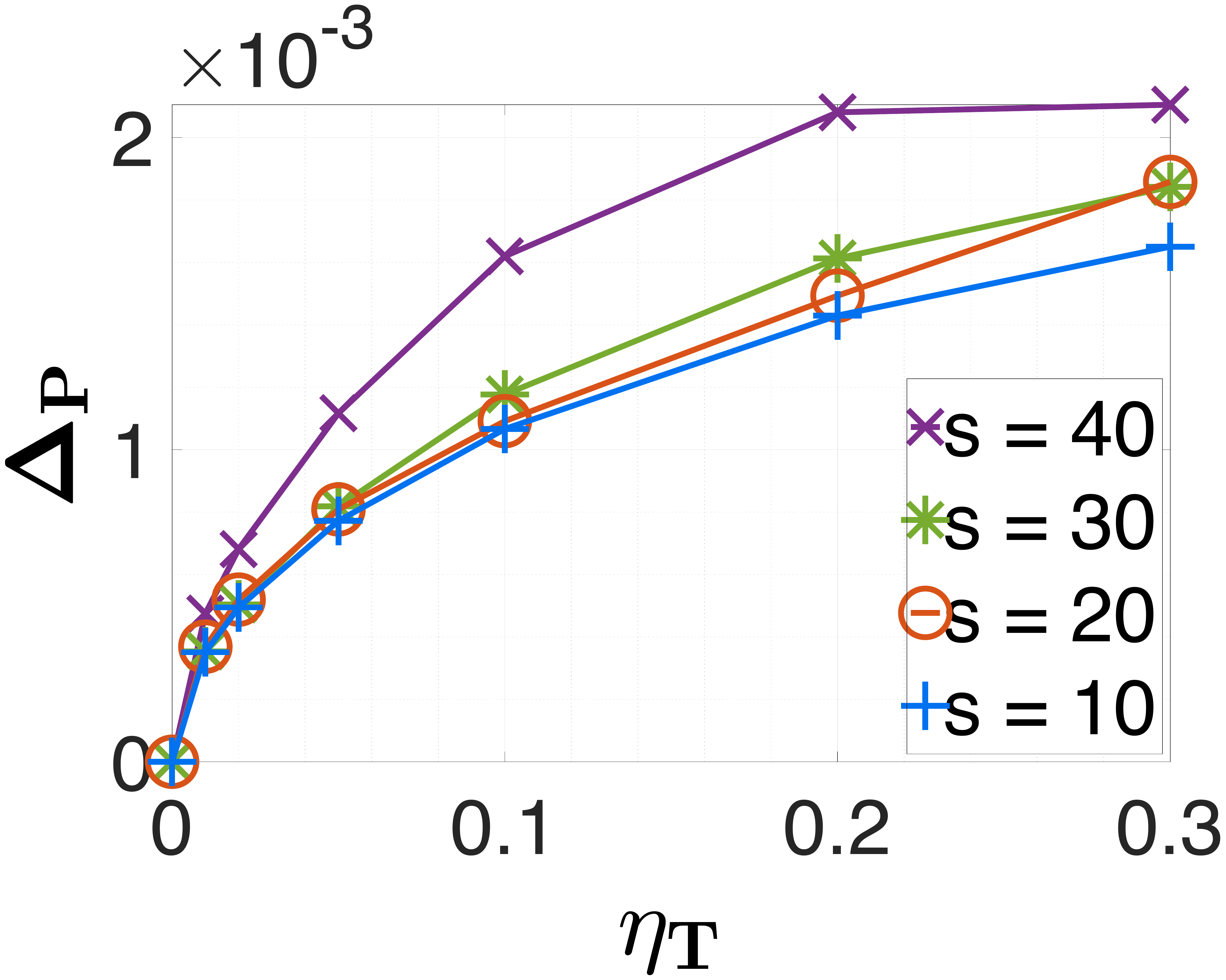} 
		\hspace{.6em}
		\includegraphics[scale=0.086]{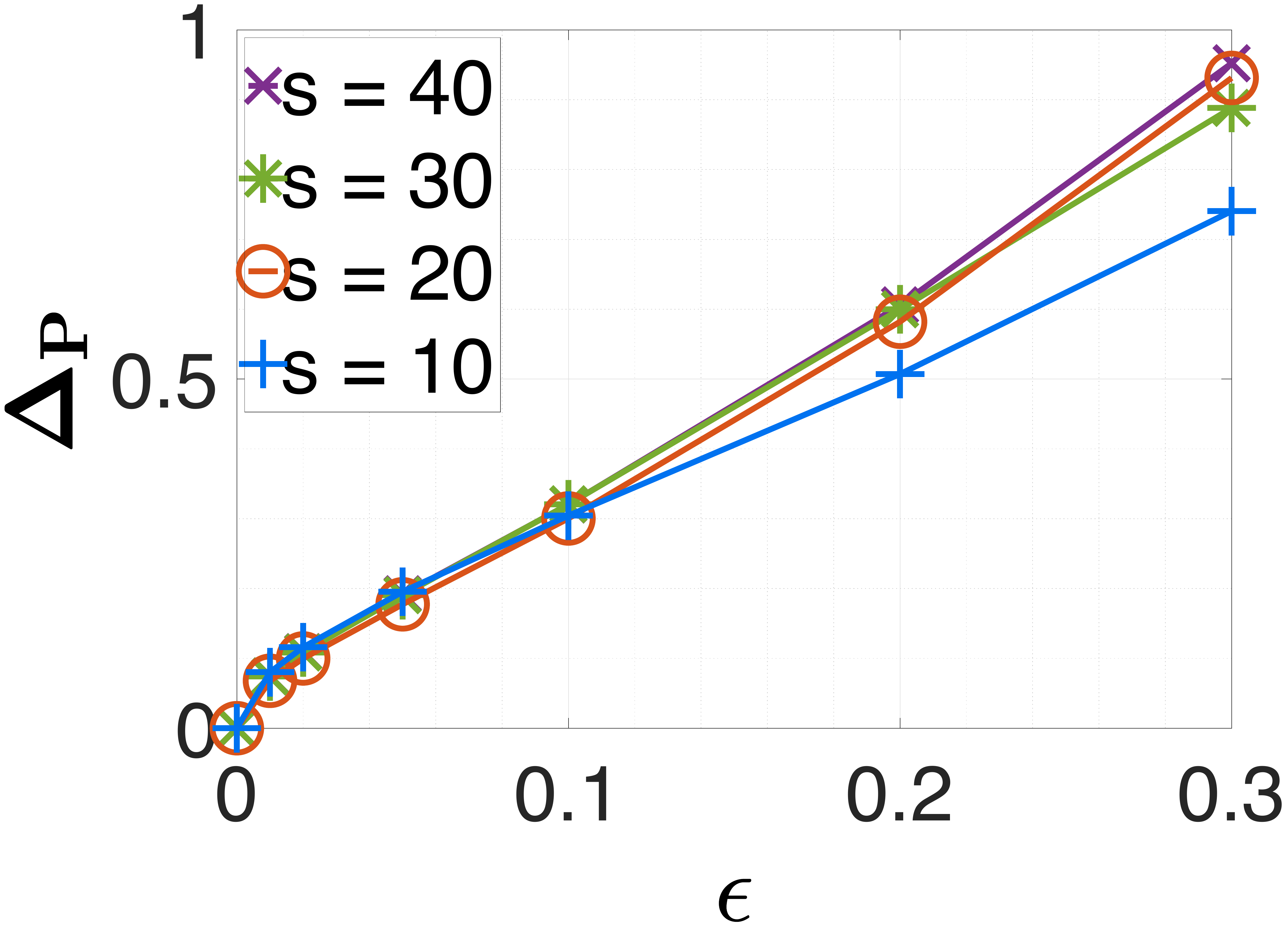} 
		\caption{{The performance of solving coupled Riccati equations via \NOMparam $(\vAhat_{1:s},\vBhat_{1:s},\vThat)$. Left to right: $\epsilon_\mtx{B}=\eta_\mtx{T}=0$, $\epsilon_\mtx{A}=\eta_\mtx{T}=0$, $\epsilon_\mtx{A}=\epsilon_\mtx{B}=0$,  and $\epsilon=\epsilon_\mtx{A}=\epsilon_\mtx{B}=\eta_\mtx{T}$. }}\label{fig:estABvsK}
	\end{center}
	\vspace{-12pt}
\end{figure*}

\begin{figure*}[t]
	\begin{center}
		\vspace{1em}% Space between image A and B
		\includegraphics[scale=0.085]{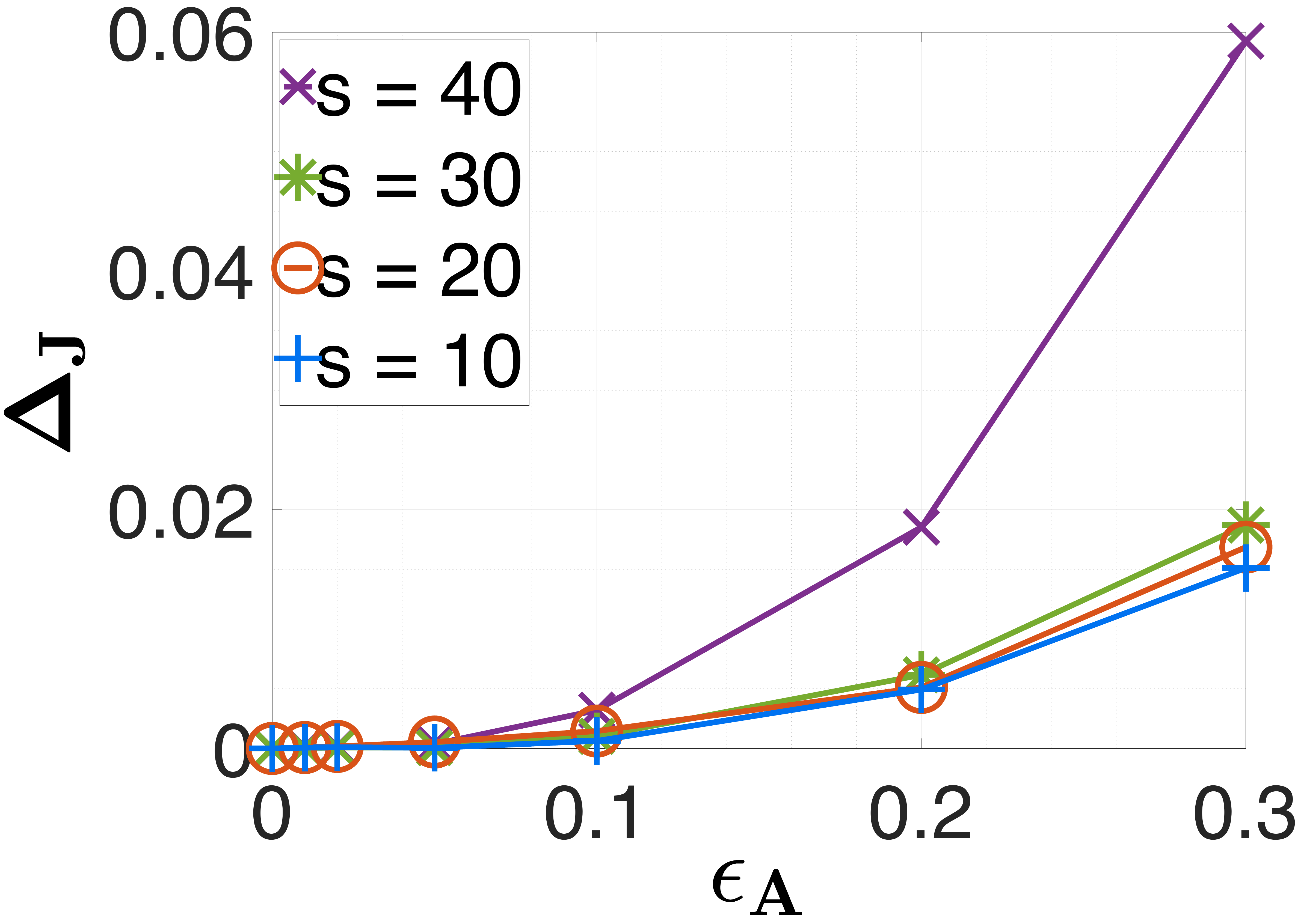}
		\hspace{.3em}
		\includegraphics[scale=0.085]{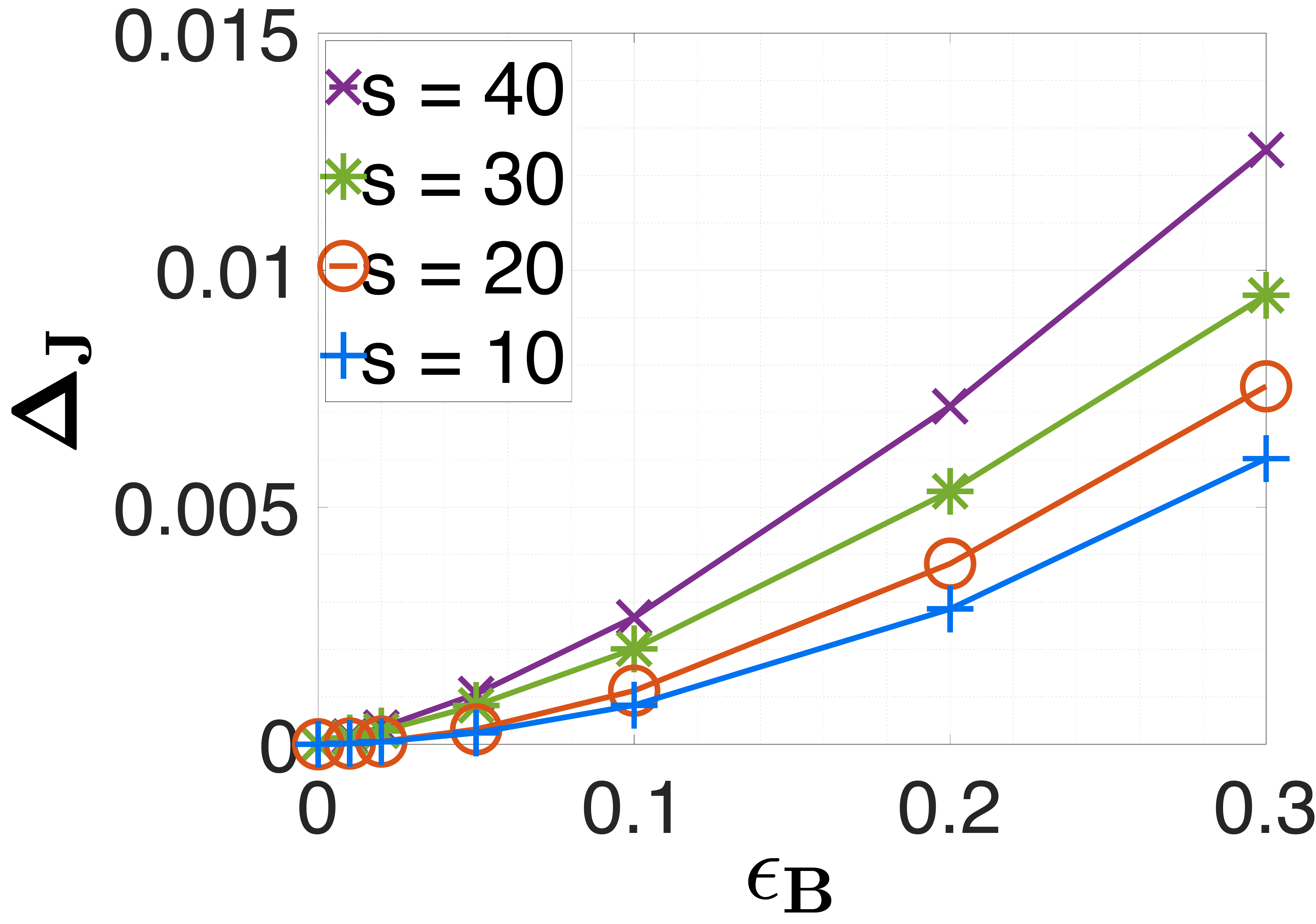}
		\hspace{.3em}
		\includegraphics[scale=0.085]{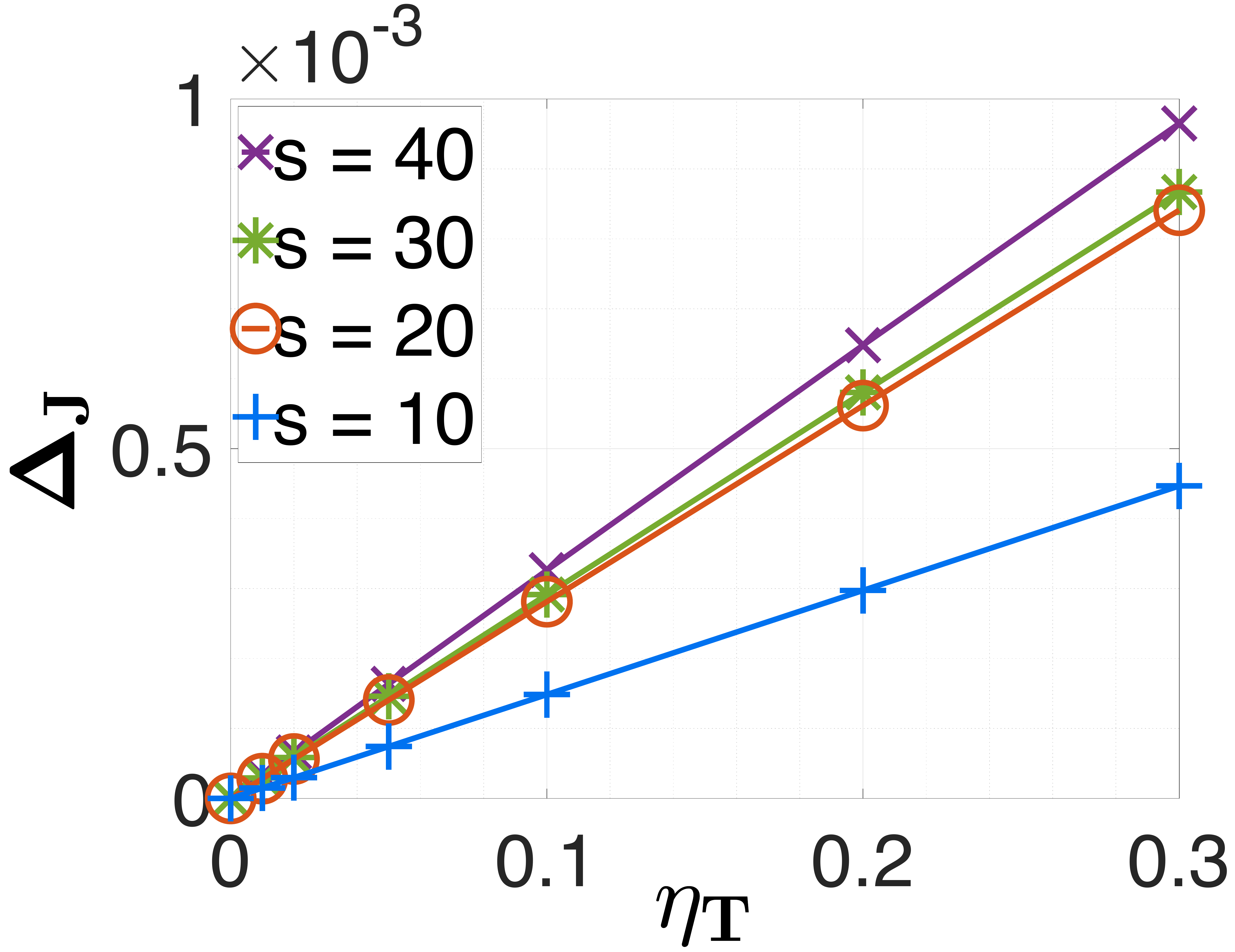} 
		\hspace{.3em}
		\includegraphics[scale=0.085]{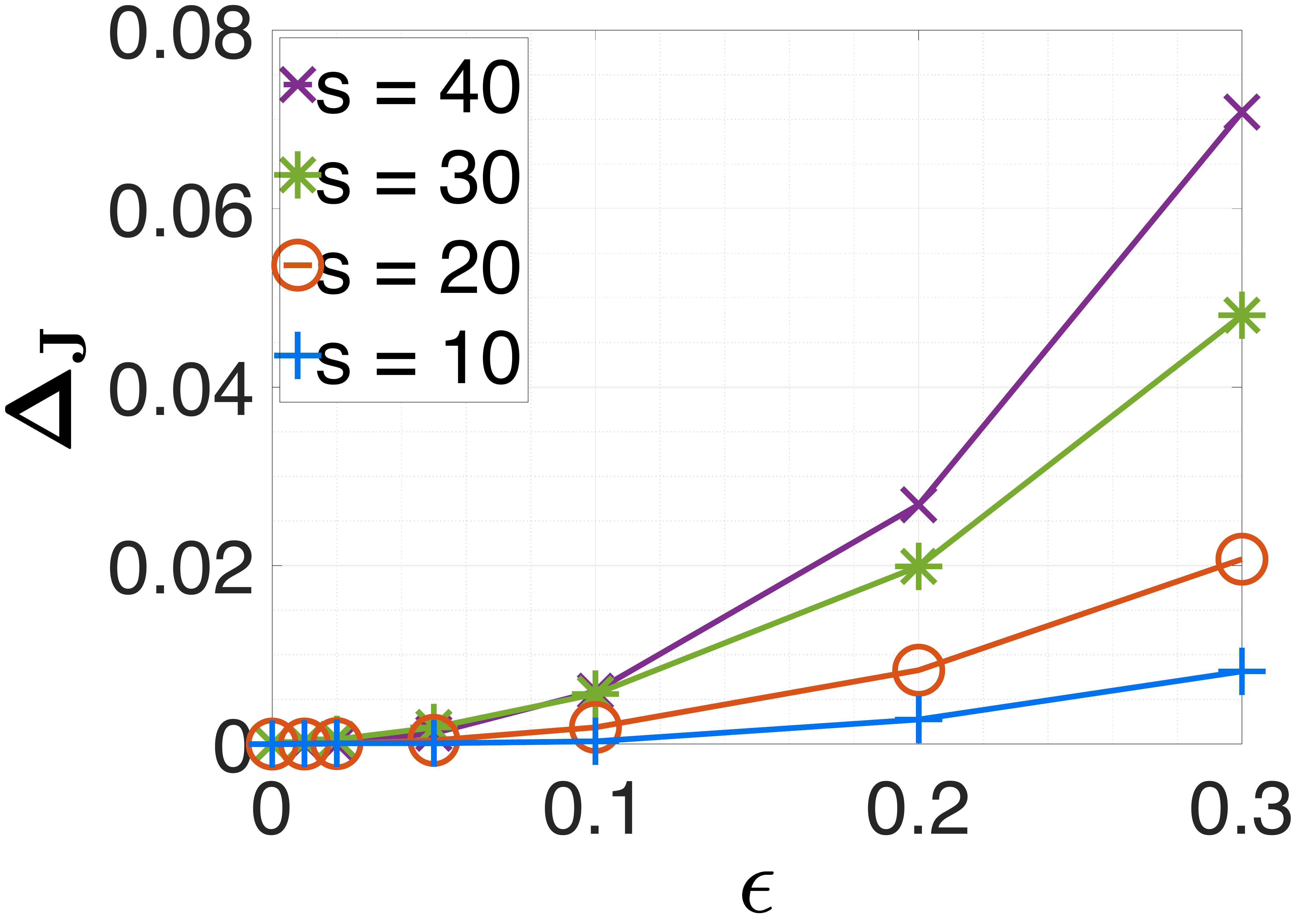} 
		\caption{{The performance of solving MJS--LQR problem via \NOMparam $(\vAhat_{1:s},\vBhat_{1:s},\vThat)$. Left to right: $\epsilon_\mtx{B}=\eta_\mtx{T}=0$, $\epsilon_\mtx{A}=\eta_\mtx{T}=0$, $\epsilon_\mtx{A}=\epsilon_\mtx{B}=0$,  and $\epsilon=\epsilon_\mtx{A}=\epsilon_\mtx{B}=\eta_\mtx{T}$. }}\label{fig:JestABvsK}
	\end{center}
	\vspace{-12pt}
\end{figure*}

Consider a system with $n$ states, $p$ inputs, and the number of modes $s$. The entries of the true system matrices $(\mtx{A}_{1:s}^\star,\mtx{B}_{1:s}^\star)$ were generated randomly from a standard normal distribution.  %\texttt{randn}\footnote{\texttt{randn}$(n,p)$ returns an $n$-by-$p$ matrix of normally distributed random numbers.} function in MATLAB. 
We scaled each $\mtx{A}^\star_i$ to have spectral radius equal to $0.3$ to obtain a mean square stable MJS. For the cost matrices $(\mtx{Q}_{1:s}, \mtx{R}_{1:s})$, and the \NOMparam $(\hat{\mtx{A}}_{1:s}, \hat{\mtx{B}}_{1:s})$, we set  
\begin{align*}
	\mtx{Q}_i&=\underline{\mtx{Q}}_i\underline{\mtx{Q}}_i^\top, ~~~~ \mtx{R}_i=\underline{\mtx{R}}_i\underline{\mtx{R}}_i^\top, \\ 
	\hat{\mtx{A}}_i &=\mtx{A}_{i}^\star + \epsilon_\mtx{A} \underline{\mtx{A}}_i,~~~~
	\hat{\mtx{B}}_i =\mtx{B}_{i}^\star + \epsilon_\mtx{B} \underline{\mtx{B}}_i,
\end{align*}
where $\underline{\mtx{Q}}_i$, $\underline{\mtx{R}}_i$, $ \underline{\mtx{A}}_i$,  and $ \underline{\mtx{B}}_i$ were generated using $\texttt{randn}$ function; and $\epsilon_\mtx{A}$ and $\epsilon_\mtx{B}$ are some fixed scalars. The \NOMparam $\hat{\mtx{T}}$ was sampled from a Dirichlet Process $\texttt{Dir}((s-1)\cdot \mtx{I}_s + 1)$. To generate the true Markov matrix $\mtx{T}^\star$ for MJS, we let $ \mtx{T}^\star = \hat{\mtx{T}} + \mtx{E}$, where the perturbation $\mtx{E} = \eta_{\mtx{T}} ( \texttt{Dir}((s-1)\cdot \mtx{I}_s + 1))-\hat{\mtx{T}})$. We note that when $\eta_\mtx{T}=0$, there is no perturbation at all and $\mtx{T}^\star = \hat{\mtx{T}}$; when $\eta_\mtx{T}=1$, $\mtx{T}^\star$ will preserve no information of $\hat{\mtx{T}}$. We also assume that we had equal probability of starting in any initial mode. 

%Further, the feasibility of Problem~\eqref{LQR optimization} can be checked using the sufficient LMI conditions~\cite{costa2006discrete}. To solve the optimization problem associated with these conditions, we use the YALMIP toolbox\footnote{\url{https://yalmip.github.io}}.  

%Considering the implementation of the iterative algorithm, we suggest to initialize  

%We note that the feasible set of Problem~\eqref{LQR optimization} consists of all $\hat{\mtx{K}}=[\hat{\mtx{K}}_1, \hat{\mtx{K}}_2, \cdots, \hat{\mtx{K}}_s]$ stabilizing the closed-loop dynamics in the mean square sense. Let $\mathbf{\mathcal{K}}$ denote this feasible set. For $\hat{\mtx{K}} \in \bf{\mathcal{K}}$, the cost in \eqref{LQR optimization} is finite and differentiable while for $\hat{\mtx{K}} \notin \bf{\mathcal{K}}$ it can be infinite and non-smooth. 

We next study how the system errors vary with $\epsilon_\mtx{A},\epsilon_\mtx{B}, \eta_\mtx{T} \in \{0.01, 0.02, 0.05, 0.1,0.2,0.3\}$, and the number of modes $s \in \{10,20, 30, 40\}$. We set the number of states and inputs to $n=10$ and $p=5$, respectively. For each choice of $\epsilon_\mtx{A}$, $\epsilon_\mtx{B}$, and $\eta_\mtx{T}$, we run $100$ experiments, and record $(\mtx{P}_{1:s}^\star,\hat{\mtx{P}}_{1:s})$ and the costs for these matrices. 

Let $\Delta_{\mtx{P}}$ denote the maximum of $ \|\hat{\mtx{P}}_i-\mtx{P}_i^\star\|/\|\mtx{P}_i^\star\|$ over the experiments and modes. We also use $ \Delta_J:=(\hat{J}-J^\star)/J^\star$ to denote the relative suboptimality gap for MJS, where $J^\star$ and $\hat{J}$ are the costs incurred by playing the optimal controller and the certainty equivalent controller on the true system, respectively. In Figures~\ref{fig:estABvsK} and \ref{fig:JestABvsK}, we plot $\Delta_{\mtx{P}}$ and $\Delta_{J}$ versus 
(i) $\epsilon_\mtx{A} ( \epsilon_\mtx{B}=\eta_\mtx{T}=0)$, (ii) $\epsilon_\mtx{B} (\epsilon_\mtx{A}=\eta_\mtx{T}=0)$, (iii) $ \eta_\mtx{T} (\epsilon_\mtx{A}=\epsilon_\mtx{B}=0)$, and (iv) $\epsilon=\epsilon_\mtx{A}=\epsilon_\mtx{B}=\eta_\mtx{T}$.

%As expected, these gaps degrade as $\epsilon_\mtx{A}$, $\epsilon_\mtx{B}$, and $\eta_\mtx{T}$ increase. 

Figure~\ref{fig:estABvsK} presents four plots that demonstrate how $\Delta_{\mtx{P}}$ changes as $\epsilon_\mtx{A}$, $\epsilon_\mtx{B}$, $\eta_\mtx{T}$, and $\epsilon$ increase, respectively. Each curve on the plot represents a fixed number of modes $s$. These empirical results are all consistent with \eqref{eq_PPerturbResult}. In particular, Figure~\ref{fig:estABvsK} (right) shows that given the uncertainty in the system matrices and in the Markov transition matrix is bounded by $\epsilon$, the perturbation bound to coupled Riccati equations has the rate $\mathcal{O}(\epsilon)$ which
degrades linearly as $\epsilon$ increase.  Further, it can be easily seen that the gaps indeed increase with the number of modes in the system.

Figure~\ref{fig:JestABvsK} demonstrates the relationship between the relative suboptimality $\Delta_J$ and the five parameters $\epsilon_\mtx{A}$, $\epsilon_\mtx{B}$, $\eta_\mtx{T}$, $\epsilon$ and $s$. As can be seen in Figure~\ref{fig:JestABvsK} (right), given the uncertainty in the system matrices and in the Markov transition matrix is bounded by $\epsilon$, the perturbation bounds to the optimal cost decay quadratically ($\mathcal{O}(\epsilon^2)$) which is consistent with \eqref{eq_suboptimality}.

%\vspace{-8pt}
\section{Conclusions}\label{sec:conc}
In this work, we provide a perturbation analysis for cDARE, which arise in the solution of MJS-LQR, and an end-to-end suboptimality guarantee for certainty equivalence control for MJS-LQR. Our results show the robustness of the optimal policy to perturbations in system dynamics and establish the validity of the certainty equivalent control in a neighborhood of the original system. This work opens up multiple future directions. First, with proper system identification algorithms, we can analyze model-based online/adaptive algorithms where control policy is updated continuously over a single trajectory. Second, a natural extension would be to study MJS with output measurements where states are only partially observed, i.e., the LQG setting. This will require considering the dual coupled Riccati equations for filtering.

% In this work, we provided a diverse set of results for the data-driven quadratic control of Markov Jump Systems. We first proved a key stability result for the coupled Riccati equations associated with the MJS. This showed the robustness of the optimal policy to perturbations in system dynamics and established the validity of the certainty equivalent control. {We then coupled these results with novel guarantees on MJS system identification to establish the first finite-sample suboptimality guarantee for MJS with optimal data dependency of $1/T$.} This work opens up multiple future directions. First, our findings enable the analysis of model-based online/adaptive algorithms where control policy is updated continuously over a single trajectory. What are the associated regret guarantees in terms of problem parameters? Second, a natural extension would be to study MJS with output measurements where states are only partially observed.

%a wealth of future directions\dots
%Despite a rich literature and widespread applications, MJS 

\bibliography{MJS_LQR_LCSS}

\begin{thebibliography}{10}
\providecommand{\url}[1]{#1}
\csname url@rmstyle\endcsname
\providecommand{\newblock}{\relax}
\providecommand{\bibinfo}[2]{#2}
\providecommand\BIBentrySTDinterwordspacing{\spaceskip=0pt\relax}
\providecommand\BIBentryALTinterwordstretchfactor{4}
\providecommand\BIBentryALTinterwordspacing{\spaceskip=\fontdimen2\font plus
\BIBentryALTinterwordstretchfactor\fontdimen3\font minus
  \fontdimen4\font\relax}
\providecommand\BIBforeignlanguage[2]{{%
\expandafter\ifx\csname l@#1\endcsname\relax
\typeout{** WARNING: IEEEtran.bst: No hyphenation pattern has been}%
\typeout{** loaded for the language `#1'. Using the pattern for}%
\typeout{** the default language instead.}%
\else
\language=\csname l@#1\endcsname
\fi
#2}}

\bibitem{campi1998adaptive}
M.~C. Campi and P.~Kumar, ``Adaptive linear quadratic gaussian control: the
  cost-biased approach revisited,'' \emph{SIAM Journal on Control and
  Optimization}, vol.~36, no.~6, pp. 1890--1907, 1998.

\bibitem{abbasi2011regret}
Y.~Abbasi-Yadkori and C.~Szepesv{\'a}ri, ``Regret bounds for the adaptive
  control of linear quadratic systems,'' in \emph{Proceedings of the 24th
  Annual Conference on Learning Theory}.\hskip 1em plus 0.5em minus 0.4em\relax
  JMLR Workshop and Conference Proceedings, 2011, pp. 1--26.

\bibitem{dean2019sample}
S.~Dean, H.~Mania, N.~Matni, B.~Recht, and S.~Tu, ``On the sample complexity of
  the linear quadratic regulator,'' \emph{Foundations of Computational
  Mathematics}, pp. 1--47, 2019.

\bibitem{mania2019certainty}
H.~Mania, S.~Tu, and B.~Recht, ``Certainty equivalence is efficient for linear
  quadratic control,'' in \emph{NeurIPS}, 2019.

\bibitem{simchowitz2020naive}
M.~Simchowitz and D.~Foster, ``Naive exploration is optimal for online lqr,''
  in \emph{International Conference on Machine Learning}.\hskip 1em plus 0.5em
  minus 0.4em\relax PMLR, 2020, pp. 8937--8948.

\bibitem{abeille2020efficient}
M.~Abeille and A.~Lazaric, ``Efficient optimistic exploration in
  linear-quadratic regulators via lagrangian relaxation,'' in
  \emph{International Conference on Machine Learning}.\hskip 1em plus 0.5em
  minus 0.4em\relax PMLR, 2020, pp. 23--31.

\bibitem{lale2020explore}
S.~Lale, K.~Azizzadenesheli, B.~Hassibi, and A.~Anandkumar, ``Explore more and
  improve regret in linear quadratic regulators,'' \emph{arXiv preprint
  arXiv:2007.12291}, 2020.

\bibitem{chizeck1986discrete}
H.~J. Chizeck, A.~S. Willsky, and D.~Castanon, ``Discrete-time markovian-jump
  linear quadratic optimal control,'' \emph{International Journal of Control},
  vol.~43, no.~1, pp. 213--231, 1986.

\bibitem{costa2006discrete}
O.~L.~V. Costa, M.~D. Fragoso, and R.~P. Marques, \emph{Discrete-time Markov
  jump linear systems}.\hskip 1em plus 0.5em minus 0.4em\relax Springer Science
  \& Business Media, 2006.

\bibitem{konstantinov1993perturbation}
M.~M. Konstantinov, P.~H. Petkov, and N.~D. Christov, ``Perturbation analysis
  of the discrete riccati equation,'' \emph{Kybernetika}, vol.~29, no.~1, pp.
  18--29, 1993.

\bibitem{konstantinov2003perturbation}
M.~Konstantinov, D.~W. Gu, V.~Mehrmann, and P.~Petkov, \emph{Perturbation
  theory for matrix equations}.\hskip 1em plus 0.5em minus 0.4em\relax Gulf
  Professional Publishing, 2003.

\bibitem{kenney1990sensitivity}
C.~Kenney and G.~Hewer, ``The sensitivity of the algebraic and differential
  riccati equations,'' \emph{SIAM journal on control and optimization},
  vol.~28, no.~1, pp. 50--69, 1990.

\bibitem{sun1998perturbation}
J.-G. Sun, ``Perturbation theory for algebraic riccati equations,'' \emph{SIAM
  Journal on Matrix Analysis and Applications}, vol.~19, no.~1, pp. 39--65,
  1998.

\bibitem{sun2002condition}
J.-g. Sun, ``Condition numbers of algebraic riccati equations in the frobenius
  norm,'' \emph{Linear algebra and its applications}, vol. 350, no. 1-3, pp.
  237--261, 2002.

\bibitem{zhou2009perturbation}
L.~Zhou, Y.~Lin, Y.~Wei, and S.~Qiao, ``Perturbation analysis and condition
  numbers of symmetric algebraic riccati equations,'' \emph{Automatica},
  vol.~45, no.~4, pp. 1005--1011, 2009.

\bibitem{konstantinov2002perturbation}
M.~Konstantinov, V.~Angelova, P.~Petkov, D.~Gu, and V.~Tsachouridis,
  ``Perturbation analysis of coupled matrix riccati equations,'' \emph{IFAC
  Proceedings Volumes}, vol.~35, no.~1, pp. 307--312, 2002.

\bibitem{konstantinov2003perturbation2}
------, ``Perturbation bounds for coupled matrix riccati equations,''
  \emph{Linear algebra and its applications}, vol. 359, no. 1-3, pp. 197--218,
  2003.

\bibitem{shi1999control}
P.~Shi, E.-K. Boukas, and R.~K. Agarwal, ``Control of markovian jump
  discrete-time systems with norm bounded uncertainty and unknown delay,''
  \emph{IEEE Trans. Automat. Control}, vol.~44, no.~11, pp. 2139--2144, 1999.

\bibitem{tu2017non}
S.~Tu, R.~Boczar, A.~Packard, and B.~Recht, ``Non-asymptotic analysis of robust
  control from coarse-grained identification,'' \emph{arXiv preprint
  arXiv:1707.04791}, 2017.

\bibitem{shaked1986guaranteed}
U.~Shaked, ``Guaranteed stability margins for the discrete-time linear
  quadratic optimal regulator,'' \emph{IEEE Transactions on Automatic Control},
  vol.~31, no.~2, pp. 162--165, 1986.

\bibitem{jansch2020policy}
J.~P. Jansch-Porto, B.~Hu, and G.~Dullerud, ``Policy optimization for markovian
  jump linear quadratic control: Gradient-based methods and global
  convergence,'' \emph{arXiv preprint arXiv:2011.11852}, 2020.

\end{thebibliography}
\bibliographystyle{IEEEtran}

\onecolumn
\appendix

\subsection{Useful matrix facts}
The following results on matrices will be used repeatedly throughout the proof and will be referred using the fact number.
\begin{fact} \label{lemma_4}
	Let $\vM$, $\vN$ be two symmetric and positive semidefinite matrices, then
	\begin{align}
		\norm{\vN (\vI + \vM \vN)^\inv} \leq \norm{\vN}, \label{eq_17} \\
		\norm{(\vI + \vM \vN)^\inv} \leq 1 + \norm{\vN} \norm{\vM}. \label{eq_18}
	\end{align}
\end{fact}
\begin{fact}
	Let $\vM$, $\vN$ be two arbitrary matrices, where $\vM$ and $\vM+\vN$ are invertible, then
	\begin{equation} \label{eq_19_tmp}
		(\vM + \vN)^\inv = \vM^\inv - \vM^\inv \vN (\vM + \vN)^\inv = \vM^\inv - (\vM + \vN)^\inv \vN \vM^\inv.
	\end{equation}
\end{fact}
\begin{fact}
	For two arbitrary matrices $\vM$ and $\vN$ such that $\vI + \vM$ and $\vI + \vN$ are both invertible, we have
	\begin{equation} \label{eq_19}
		(\vI + \vM)^\inv - (\vI + \vN)^\inv = (\vI + \vM)^\inv (\vN - \vM) (\vI + \vN)^\inv. 
	\end{equation}
\end{fact}

Above, \eqref{eq_17} is due to \cite[Lemma~7]{mania2019certainty} (in their supplement). To see \eqref{eq_18}, first note that $(\vI + \vM \vN)^\inv = \vI - \vM \vN (\vI + \vM \vN)^\inv$ by matrix inversion lemma, and then apply \eqref{eq_17}. \eqref{eq_19_tmp} and \eqref{eq_19} also follow from matrix inversion lemma.

\begin{fact}\label{fact:tilde_vek}
	Consider a block-diagonal matrix $\vX \in \dm{\numSys \dimSt}{\numSys \dimSt}$ composed of $\vX_{1:\numSys}$ such that the $i$th diagonal block of $\vX$ is given by $\vX_i \in \dm{\dimSt}{\dimSt}$. Let $\tilde{\vek}(\cdot)$ be the operator that vectorizes all diagonal blocks of $\vX$ into a vector, i.e. $\tilde{\vek} (\vX):= (\vek(\vX_1),  \dots, \vek(\vX_\numSys))$. Let $\tilde{\vek}^\inv$ denote the inverse of $\tilde{\vek}$ such that $\tilde{\vek}^\inv(\tilde{\vek}(\vX)) = \vX$.
	Then,
	\begin{align}
		\norm{\tilde{\vek}} &= \sup_{\Xb = \diag(\Xb_{1:\numSys}), \norm{\Xb}=1} \norm{\tilde{\vek}(\Xb)} \overset{\text{(i)}}{=} \sqrt{n s} \label{eq_5} \\
		\norm{\tilde{\vek}^\inv} &= \sup_{\norm{\vx}=1} \norm{\tilde{\vek}^\inv(\vx)} \overset{\text{(ii)}}{=} 1. \label{eq_6}
	\end{align}
\end{fact}
Fact~\ref{fact:tilde_vek} follows by noting that (i) achieves the supremum when $\Xb_i = \Ib_n$ for all $i$ and (ii) achieves the supremum when $\vx = (1, 0, \dots, 0)$. The following fact is adapted from~\cite{mania2019certainty} and is useful in bounding the spectral radius of a perturbed matrix.

\begin{fact}\label{fact:rho perturbation}
	Let $\Mb$ be an arbitrary matrix in $\R^{n \times n}$ and let $\rho(\Mb) \leq \gamma$. Then for all $k \geq 1$ and real matrices $\Delta$ of appropriate dimensions, we have $\norm{(\Mb + \mtx{\Delta})^k} \leq \tau(\Mb,\gamma)\big(\tau(\Mb,\gamma)\norm{\mtx{\Delta}}+\gamma\big)^k$.
\end{fact}
\subsection{Proof of Theorem \ref{prop_riccatiPertb}}

We first provide a lemma, used in the proof of Theorem~\ref{prop_riccatiPertb},  that establishes that positive definite solutions of coupled algebraic Riccati equations are unique among the set of positive semidefinite matrices when they exist. Note that existence of such solutions guarantee stabilizability in mean square sense. %While the lemma follows from \cite[Lemma A.14]{costa2006discrete}, for completeness, we also provide a self-contained proof.

\begin{lemma}\label{lemma_lqrSolPD} (\cite[Lemma A.14]{costa2006discrete}) Consider cDARE($\vA_{1:\numSys}, \vB_{1:\numSys}, \vT$) for a generic MJS($\vA_{1:\numSys}, \vB_{1:\numSys}, \vT$) and LQR cost matrices $\vQ_{1:\numSys}, \vR_{1:\numSys}$. Assume $\vQ_i, \vR_i \succ 0$ for all $i \in [\numSys]$. Then, if there exists a positive definite solution $\vP_{1:\numSys}$ to cDARE($\vA_{1:\numSys}, \vB_{1:\numSys}, \vT$), then it is the unique solution among $\curlybrackets{\vX_{1:\numSys}: \vX_i \succeq 0, \forall~ i \in [s]}$. 
\end{lemma}

\begin{proof}[Main Proof for Theorem \ref{prop_riccatiPertb}]
	{\,} \\
	The outline of the proof is as follows:
	\begin{enumerate}[label=(\alph*)]
		\item We first construct an operator $\Kcal(\vX_{1:\numSys})$ using the difference between the true cDARE$(\vA^\star_{1:\numSys}, \vB^\star_{1:\numSys}, \vT^\star)$ and the \NOMcDARE cDARE$(\vAhat_{1:s}, \vBhat_{1:s}, \vThat)$, whose fixed point(s) $\vX_{1:\numSys}^\star$ (when exist) will guarantee $\vP_{1:s}^\star + \vX_{1:\numSys}^\star$ is a solution to the \NOMcDARE cDARE$(\vAhat_{1:s}, \vBhat_{1:s}, \vThat)$.
		\item Then we show when $\epsilon$ and $\eta$ are small enough, $\Kcal(\vX_{1:\numSys})$ will be a contraction mapping on a closed set $\Scal_\nu$ whose radius $\nu$ is a function of $\epsilon$ and $\eta$. Thus, there exists a unique fixed point $\vX^\star_{1:\numSys} \in \Scal_\nu$ by contraction mapping theorem, and $\vP^\star_{1:s} + \vX^\star_{1:\numSys}$ is a solution to the \NOMcDARE cDARE$(\vAhat_{1:s}, \vBhat_{1:s}, \vThat)$.
		\item Finally, given there exists a unique solution to the \NOMcDARE cDARE$(\vAhat_{1:s}, \vBhat_{1:s}, \vThat)$ in the neighborhood $\Scal_\nu$ of $\vP^\star_{1:\numSys}$, we will show this solution is the only possible solution among all positive semi-definite matrices. To do this, we first show $\vP^\star_i + \vX^\star_i \succ 0$. By Lemma \ref{lemma_lqrSolPD}, we know the \NOMcDARE cDARE \eqref{eq_CARE_perturbed} has a unique solution, given by $\vPhat_{1:\numSys}:=\vP^\star_{1:s} + \vX^\star_{1:\numSys}$, among $\curlybrackets{\vX_{1:\numSys}: \vX_i \succeq 0, \forall i}$. Furthermore, $\norm{\vPhat_{1:\numSys} - \vP^\star_{1:\numSys}} = \norm{\vX^\star_{1:\numSys}} \leq \nu(\epsilon, \eta)$.
	\end{enumerate}
	\noindent \textbf{Step (a):} Construct operator 
	$\Kcal$.
	
	First we define a few notations for the ease of exposition. For all $i \in [\numSys]$, let $\Sb_i^\star := \vB_i^\star \Rb_i^{-1} {\vB_i^\star}^\T$ and $\vShat_i := \vBhat_i \Rb_i^{-1} \vBhat_i^\T$. 
	Define block diagonal matrices $\vA^\star$, $\vAhat$, $\vB^\star$, $\vBhat$, $\Qb$, $\Rb$, $\vP^\star$, $\vK^\star$, $\vL^\star$, $\vS^\star$, $\vShat$, $\Xb$, $\vPhi^\star(\Xb)$, $\hat{\vPhi}(\Xb)$ such that their $i$th diagonal blocks are given by $\vA^\star_i$, $\vAhat_i$, $\vB_i^\star$, $\vBhat_i$, $\Qb_i$, $\Rb_i$, $\vP^\star_i$, $\vK^\star_i$, $\vL^\star_i$, $\Sb_i^\star$, $\vShat_i$, $\Xb_i$, $\varphi^\star_i(\Xb_{1:s})$, $\hat{\varphi}_i(\Xb_{1:s})$ respectively. Note that $\vX_{1:\numSys}$ represent arguments of matrix functions or unknown variables used in matrix equations, e.g. \eqref{eq_CARE}. We will see many equations that hold for each single block also hold for the diagonally concatenated notations. 
	
	We have
	$
	\vK^\star = - (\Rb + {\vB^\star}^\T \vPhi^\star(\vP^\star) \vB^\star)^\inv {\vB^\star}^\T \vPhi^\star(\vP^\star) \vA^\star
	$
	from \eqref{eq_optfeedback}, then using the matrix inversion lemma, we can get
	\begin{equation} \label{eq_2}
		\vL^\star = \vA^\star + \vB^\star \vK^\star  = \parentheses{\Ib + \vS^\star \vPhi^\star(\vP^\star)}^\inv \vA^\star.
	\end{equation}	
	Furthermore, by diagonally concatenating cDARE \eqref{eq_CARE} and then applying the matrix inversion lemma again, we have
	\begin{equation}\label{eq_lcss_3}
		\begin{split}
			\Xb 
			&= {\vA^\star}^\T \vPhi^\star(\Xb) \parentheses{\Ib + \vS^\star \vPhi^\star(\Xb)}^\inv \vA^\star + \Qb.
		\end{split}		
	\end{equation}
	Then, we define the following Riccati difference function using the difference between LHS and RHS of \eqref{eq_lcss_3}, while keeping the RHS parameter dependent:
	\begin{equation}\label{eq_1}
		F(\Xb; \vA^\star, \vB^\star, \vT^\star) {:=} \Xb - {\vA^\star}^\T \vPhi^\star(\Xb) \parentheses{\Ib + \vS^\star \vPhi^\star(\Xb)}^\inv \vA^\star - \Qb.
	\end{equation}
	Though not explicitly listed, $\vPhi^\star$ and $\vS^\star$ on the RHS of \eqref{eq_1} depend on arguments $\vT^\star$ and $\vB^\star$ respectively. Since $\vP_{1:\numSys}^\star$ is the solution to the true $\text{cDARE}(\vA_{1:s}^\star, \vB_{1:s}^\star, \vT^\star)$, we have $F(\vP^\star; \vA^\star, \vB^\star, \vT^\star)=0$. Similarly, if there exists solution $\vPhat_{1:\numSys}$ to the \NOMcDARE $\text{cDARE}(\vAhat_{1:s}, \vBhat_{1:s}, \vThat)$, then $\vPhat:=\diag(\vPhat_{1:\numSys})$ would satisfy $F(\vPhat; \vAhat, \vBhat, \vThat)=0$.
	% Later we will show that when $\epsilon$ and $\eta$ are small, $\vPhat_{1:\numSys}$ indeed exists. 
	
	Now we consider the function $F(\vP^\star + \Xb; \vA^\star, \vB^\star, \vT^\star)$. When $\vP^\star + \vX \succeq 0$, we have the following.
	\begin{equation}
		\begin{split}
			F(\vP^\star + \Xb; &\vA^\star, \vB^\star, \vT^\star)\\
			\overset{\text{(i)}}{=} & \vP^\star + \Xb - {\vA^\star}^\T \squarebracketsbig{\vPhi^\star(\vP^\star) + \vPhi^\star(\vX)} \squarebracketsbig{\Ib + \vS^\star \vPhi^\star(\vP^\star) + \vS^\star \vPhi^\star(\vX)}^\inv \vA^\star - \Qb \\
			\overset{\text{(ii)}}{=} & \vP^\star + \Xb - {\vA^\star}^\T \squarebracketsbig{\vPhi^\star(\vP^\star) + \vPhi^\star(\vX)} \\
			\quad& \squarebrackets{ (\Ib + \vS^\star \vPhi^\star(\vP^\star))^\inv - \underbrace{\parentheses{\Ib + \vS^\star \vPhi^\star(\vP^\star) + \vS^\star \vPhi^\star(\vX)}^\inv}_{=:\vGamma} \vS^\star \vPhi^\star(\Xb) (\Ib + \vS^\star \vPhi^\star(\vP^\star))^\inv} \vA^\star - \Qb \\
			= & \vP^\star + \Xb - {\vA^\star}^\T \squarebracketsbig{\vPhi^\star(\vP^\star) + \vPhi^\star(\vX)} \squarebracketsbig{ \Ib - \vGamma \vS^\star \vPhi^\star(\Xb)} \squarebracketsbig{\Ib + \vS^\star \vPhi^\star(\vP^\star)}^\inv \vA^\star - \Qb \\
			\overset{\text{(iii)}}{=} & \vP^\star + \Xb - {\vA^\star}^\T \squarebracketsbig{\vPhi^\star(\vP^\star) + \vPhi^\star(\vX)} \squarebracketsbig{ \Ib - \vGamma \vS^\star \vPhi^\star(\Xb)} \vL^\star - \Qb \\
			%= & \Xb - {\vA^\star}^\T \squarebracketsbig{\vPhi^\star(\vP^\star) + \vPhi^\star(\vX)} \squarebracketsbig{ \Ib - \vGamma \vS^\star \vPhi^\star(\Xb)} \vL^\star + {\vA^\star}^\T \vPhi^\star(\vP^\star) \vL^\star + \squarebracketsbig{\vP^\star - {\vA^\star}^\T \vPhi^\star(\Xb) \vL^\star - \Qb}\\
			\overset{\text{(iv)}}{=} & \Xb - {\vA^\star}^\T \squarebracketsbig{\vPhi^\star(\vP^\star) + \vPhi^\star(\vX)} \squarebracketsbig{ \Ib - \vGamma \vS^\star \vPhi^\star(\Xb)} \vL^\star  + {\vA^\star}^\T \vPhi^\star(\vP^\star) \vL^\star \\
			= & \Xb - {\vA^\star}^\T \curlybracketsbig{\squarebracketsbig{\vPhi^\star(\vP^\star) + \vPhi^\star(\vX)} \squarebracketsbig{ \Ib - \vGamma \vS^\star \vPhi^\star(\Xb)} - \vPhi^\star(\vP^\star)}\vL^\star \\
			\overset{\text{(v)}}{=}  & \Xb - {\vL^\star}^\T (\Ib + \vPhi^\star(\vP^\star) \vS^\star)\curlybracketsbig{\squarebracketsbig{\vPhi^\star(\vP^\star) + \vPhi^\star(\vX)} \squarebracketsbig{ \Ib - \vGamma \vS^\star \vPhi^\star(\Xb)} - \vPhi^\star(\vP^\star)}\vL^\star \\
			= & \Xb - {\vL^\star}^\T \underbrace{(\Ib + \vPhi^\star(\vP^\star) \vS^\star) \squarebracketsbig{ - \vPhi^\star(\vP^\star) \vGamma \vS^\star + \Ib - \vPhi^\star(\Xb) \vGamma \vS^\star}}_{=: \vLambda}  \vPhi^\star(\Xb) \vL^\star  \\
		\end{split}
	\end{equation}
	where (i) follows from the definition in \eqref{eq_1};
	(ii) follows from the identity $(\Mb + \Nb)^\inv = \Mb^\inv - (\Mb + \Nb)^\inv \Nb \Mb^\inv$ and the fact $\Ib + \vS^\star \vPhi^\star(\vP^\star) + \vS^\star \vPhi^\star(\vX)$ is invertible due to the assumption that $\vP^\star + \vX \succeq 0$; (iii) and (v) follow from the fact $\parentheses{\Ib + \vS^\star \vPhi^\star(\vP^\star)}^\inv \vA^\star = \vL^\star$ in \eqref{eq_2}; (iv) follows from the fact $F(\vP^\star; \vA^\star, \vB^\star, \vT^\star) = \vP^\star - {\vA^\star}^\T \vPhi^\star(\Xb) \vL^\star - \Qb = 0$ in \eqref{eq_1}.
	% ; (vi) follows from the identity $(\Mb + \Nb)^\inv = \Mb^\inv - (\Mb + \Nb)^\inv \Nb \Mb^\inv$
	% \todo{where is (vi)?} 
	Note that
	\begin{equation}
		\begin{split}
			\vLambda 
			&= -\vPhi^\star(\vP^\star) \vGamma \vS^\star + \Ib - \vPhi^\star(\Xb) \vGamma \vS^\star - \vPhi^\star(\vP^\star) \vS^\star \vPhi^\star(\vP^\star) \vGamma \vS^\star + \vPhi^\star(\vP^\star) \vS^\star - \vPhi^\star(\vP^\star) \vS^\star \vPhi^\star(\Xb) \vGamma \vS^\star \\
			&= - \curlybracketsbig{\vPhi^\star(\vP^\star) + \vPhi^\star(\vP^\star) \vS^\star \vPhi^\star(\vP^\star) - \vPhi^\star(\vP^\star) \squarebracketsbig{ (\Ib + \vS^\star \vPhi^\star(\vP^\star) + \vS^\star \vPhi^\star(\Xb))} + \vPhi^\star(\vP^\star) \vS^\star \vPhi^\star(\Xb)} \vGamma \vS^\star \\
			& \quad + \Ib - \vPhi^\star(\Xb)\vGamma \vS^\star \\
			&= \Ib - \vPhi^\star(\Xb)\vGamma \vS^\star.
		\end{split}	
	\end{equation}
	Therefore
	\begin{equation}
		F(\vP^\star + \Xb; \vA^\star, \vB^\star, \vT^\star) = \Xb - {\vL^\star}^\T \vPhi^\star(\Xb) \vL^\star + {\vL^\star}^\T \vPhi^\star(\Xb) \parentheses{\Ib + \vS^\star \vPhi^\star(\vP^\star) + \vS^\star \vPhi^\star(\Xb)}^\inv \vS^\star \vPhi^\star(\Xb) \vL^\star.
	\end{equation}
	If we define
	\begin{align}
		\Tcal(\Xb) &:= \Xb - {\vL^\star}^\T \vPhi^\star(\Xb) \vL^\star, \\
		\Hcal(\Xb) &:= {\vL^\star}^\T \vPhi^\star(\Xb) \parentheses{\Ib + \vS^\star \vPhi^\star(\vP^\star) + \vS^\star \vPhi^\star(\Xb)}^\inv \vS^\star \vPhi^\star(\Xb) \vL^\star,
	\end{align}
	we have
	\begin{equation}\label{lemma_1}
		F(\vP^\star+\Xb; \vA^\star, \vB^\star, \vT^\star) = \Tcal(\Xb) + \Hcal(\Xb).
	\end{equation}
	
	Let $\vY_i := \vX_i - {\vL^\star_i}^\T \varphi^\star_i(\vX_{1:\numSys}) \vL^\star_i$, and $\vY: = \diag(\vY_{1:\numSys})$, then linear operator $\Tcal$ can be viewed as $\Tcal: \vX \mapsto \vY$. By vectorization, we see for every $i$
	\begin{equation}
		\parentheses{\Ib - [\vT^\star]_{ii} \cdot {\vL_i^\star}^\T \otimes {\vL_i^\star}^\T} \vek(\Xb_i) 
		- \sum_{j\neq i} [\vT^\star]_{ij} {\vL_i^\star}^\T \otimes {\vL_i^\star}^\T \vek(\Xb_j) = \vek(\vY_i).
	\end{equation}
	
	Stacking this equation for all $i$, we have $\parentheses{\Ib - \vLtil^\star} \tilde{\vek}(\vX) = \tilde{\vek}(\vY)$,
	where $\tilde{\vek}(\cdot)$ is defined in Fact~\ref{fact:tilde_vek}. From the stability discussion in Sec \ref{sec:main}, we know  $\rho(\vLtil^\star)<1$, thus $\parentheses{\Ib - \vLtil^\star}$ is invertible, and $\Tcal^\inv$ is given by
	\begin{equation}\label{eq:Tinv_def}\Xb = \Tcal^\inv (\vY) = \tilde{\vek}^\inv \circ \parentheses{\Ib - \vLtil^\star}^\inv \circ \tilde{\vek} (\vY),\end{equation}
	where $\circ$ denotes operator composition. 
	Since $\Tcal^\inv$ is well defined, we can construct the following operator:
	\begin{equation}
		\label{eq:defKcal}
		\Kcal(\Xb) = \Tcal^\inv ( F(\vP^\star+\Xb; \vA^\star, \vB^\star, \vT^\star)
		- F(\vP^\star+\Xb; \vAhat, \vBhat, \vThat) - \Hcal(\Xb) ).
	\end{equation}
	From \eqref{lemma_1}, we see that if there exists a fixed point $\vX^\star$ for $\Kcal$, then $F(\vP^\star + \vX^\star; \vAhat, \vBhat, \vThat)=0$, i.e. $\vP_{1:\numSys}^\star+\Xb_{1:\numSys}$ is a solution to the \NOMcDARE $\text{cDARE}(\vAhat_{1:s}, \vBhat_{1:s}, \vThat)$.	 %\todo{if the goal of this section is to construct $Kcal$, why do we need \ref{eq_7} and \ref{lemma_2}?}
	
	\noindent \textbf{Step (b):} Show that $\Kcal$ is a contraction to conclude existence of a perturbed solution.
	
	We consider the set
	\begin{equation}
		\Scal_\nu:= \curlybrackets{\Xb: \norm{\Xb} \leq \nu, \Xb = \diag(\Xb_{1:\numSys}), \vP^\star+\Xb \succeq 0}.
	\end{equation}
	We will show when $\nu$ is small enough, $\Kcal$ maps $\Scal_\nu$ into itself and is a contraction mapping. Thus, $\Kcal$ is guaranteed to have a fixed point in $\Scal_\nu$. To do this, we first present a lemma that bounds $\Kcal$ when $\epsilon$ and $\nu$ are sufficiently small. Then, we provide a choice of $\nu$ that makes this bound valid.%\todo{say more on why. connections are missing.}
	\begin{lemma}\label{lemma_3}
		Assume $\epsilon \leq \min \curlybrackets{\norm{\vB^\star}, 1}$. Suppose $\vX, \vX_1, \vX_2 \in \Scal_\nu$ with $\nu \leq \min \curlybrackets{1, \norm{\vS^\star}^\inv}$, then
		\begin{align}
			\norm{\Kcal(\vX)} &\leq \frac{\sqrt{\dimSt \numSys} \tau (\vLtil^\star, \gamma)}{1-\gamma} \left( \norm{\vL^\star}^2 \norm{\vS^\star} \nu^2 {+}  3 C_\epsilon \epsilon {+} C_\eta \eta \right) \label{eq_12}\\
			\norm{\Kcal(\vX_1) - \Kcal(\vX_2)} &\leq \frac{\sqrt{\dimSt \numSys} \tau (\vLtil^\star, \gamma)}{1-\gamma} \left( 3 \norm{\vL^\star}^2 \norm{\vS^\star} \nu \right. \left. + \norm{\vB^\star}_+^2 \norm{\vR^\inv}_+ (51 \epsilon/C_\epsilon^u + 2 \eta /C_\eta^u  ) \right) \norm{\vX_1 - \vX_2}. \label{eq_13}
		\end{align}
	\end{lemma}
	Proof is given in Appendix \ref{subsec_lcss_1}.
	% \begin{proof}[Proof (sketch only, complete proof in Appendix \ref{subsec_lcss_1})]\todo{proof sketch inside the appendix is odd.}
	% We define
	% $
	% \Gcal_1(\vX): = F(\vP^\star + \vX; \vA^\star, \vB^\star, \vThat) - F(\vP^\star + \vX; \vAhat, \vBhat, \vThat)$, and $\Gcal_2(\vX): = F(\vP^\star + \vX; \vA^\star, \vB^\star, \vT^\star) - F(\vP^\star + \vX; \vA^\star, \vB^\star, \vThat).
	% $
	% Then, we can see $\Kcal(\vX) = \Tcal^\inv(\Gcal_1(\vX) + \Gcal_2(\vX) - \Hcal(\vX))$ and $\Kcal(\vX_1) - \Kcal(\vX_2) = \Tcal^\inv(\Gcal_1(\vX_1) - \Gcal_1(\vX_2) + \Gcal_2(\vX_2) - \Gcal_2(\vX_2) - \Hcal(\vX_1) + \Hcal(\vX_2))$. We can upper bound $\norm{\Hcal(\vX)}$, $\norm{\Gcal_1(\vX)}$, $\norm{\Gcal_2(\vX)}$, $\norm{\Hcal(\vX_1) - \Hcal(\vX_2)}$, $\norm{\Gcal_1(\vX_1) - \Gcal_1(\vX_2)}$, $\norm{\Gcal_2(\vX_1) - \Gcal_2(\vX_2)}$, for any $\vX, \vX_1, \vX_2 \in \Scal_\nu$, and combine the bound for $\norm{\Tcal^\inv}$ in \eqref{lemma_2}, then we can conclude the proof for Lemma \ref{lemma_3}.
	% \end{proof}
	To apply this lemma, let
	\begin{equation}\label{eq_9}
		\nu = 	\frac{\sqrt{\dimSt \numSys} \tau(\vLtil^\star, \gamma)}{1-\gamma}
		\parenthesesbig{
			6 C_\epsilon \epsilon
			+
			2 C_\eta \eta		
		}.
	\end{equation}
	Applying the upper bounds for $\epsilon$ and $\eta$ in the premises of Theorem \ref{prop_riccatiPertb} to \eqref{eq_9}, we have
	\begin{equation}\label{eq_10}
		\begin{split}
			\nu 
			&\leq \frac{1-\gamma}{\sqrt{\dimSt \numSys} \tau(\vLtil^\star, \gamma)} \parentheses{\frac{1}{34} \norm{\vB^\star}_+^{-2} \norm{\vP^\star}_+^{-1} \norm{\vR^\inv}_+^{-2} + \frac{1}{24}} \min \curlybrackets{\norm{\vB^\star}_+^{-2} \norm{\vR^\inv}_+^{-1} \norm{\vL^\star}_+^{-2},  \underline{\sigma}(\vP^\star)} \\
			&\leq \frac{1-\gamma}{12 \sqrt{\dimSt \numSys} \tau(\vLtil^\star, \gamma)} \min \curlybrackets{\norm{\vB^\star}_+^{-2} \norm{\vR^\inv}_+^{-1} \norm{\vL^\star}_+^{-2},  \underline{\sigma}(\vP^\star)} \\
			&\leq \min \hspace{-2pt} \curlybracketsbig{1, \frac{1}{\norm{\vS^\star}}, \frac{1-\gamma}{12 \sqrt{\dimSt \numSys} \tau(\vLtil^\star, \gamma) \norm{\vL^\star}^{2} \norm{\vS^\star}} , \frac{\underline{\sigma}(\vP^\star)}{12}},
		\end{split}
	\end{equation}
	where the second line follows from the fact $\norm{\cdot}_+ = \norm{\cdot}+1$, and the last line uses $\norm{\vS^\star} \leq \norm{\vB^\star}^2 \norm{\vR^\inv}$ by definition of $\vS^\star$.
	It is easy to see that the upper bound of $\epsilon$ in the premises of Theorem \ref{prop_riccatiPertb} guarantees $\epsilon \leq \min \curlybrackets{\norm{\vB^\star}, 1}$, which when combined with \eqref{eq_10}, makes the bounds in Lemma \ref{lemma_3} applicable and we get:
	%then with the upper bounds for $\nu$ in \eqref{eq_10}, the requirements on $\epsilon$ and $\eta$ in Lemma \ref{lemma_3} are satisfied, which allows us to apply Lemma \ref{lemma_3}.
	\begin{equation}
		\norm{\Kcal(\vX)} \leq \frac{1}{12} \nu + \frac{1}{2} \nu = \frac{7}{12} \nu,
	\end{equation}
	by cancelling off $\epsilon$ and $\eta$ in \eqref{eq_12} using \eqref{eq_9}, and applying the third upper bound for $\nu$ in \eqref{eq_10}.
	Since $\nu \leq \underline{\sigma}(\vP^\star)$ in \eqref{eq_10}, we can see $\vP^\star + \Kcal(\vX) \succ 0$, thus $\Kcal(\vX) \in \Scal_\nu$, i.e. $\Kcal$ maps $\Scal_\nu$ into itself. Furthermore, applying upper bounds for $\epsilon$ and $\eta$ in Theorem \ref{prop_riccatiPertb} and the third upper bound for $\nu$ in \eqref{eq_10} to \eqref{eq_13} gives
	\begin{align*}
		\norm{\Kcal(\vX_1) - \Kcal(\vX_2)} 
		&\leq \parenthesesbig{\frac{1}{4} + \frac{1}{4} \frac{1-\gamma}{\sqrt{\dimSt \numSys} \tau(\vLtil^\star, \gamma)} \norm{\vL^\star}_+^{-2} + \frac{1}{24} \frac{1-\gamma}{\sqrt{\dimSt \numSys} \tau(\vLtil^\star, \gamma)} \norm{\vL^\star}_+^{-2}} \norm{\vX_1 - \vX_2} \\
		&\leq \frac{13}{24} \norm{\vX_1 - \vX_2}.
	\end{align*}
	We have shown $\Kcal(\vX)$ not only maps closed set $\Scal_\nu$ into itself but also is a contraction mapping on $\Scal_\nu$, which means $\Kcal(\vX)$ has a unique fixed point $\vX^\star \in \Scal_\nu$. By definition of $\Kcal(\vX)$ and the identity in \eqref{lemma_1}, we see  $F(\vP^\star+\vX^\star; \vAhat, \vBhat, \vThat) = 0$, i.e. $\vP^\star+\vX^\star$ is a solution to the \NOMcDARE $\text{cDARE}(\vAhat_{1:s}, \vBhat_{1:s}, \vThat)$. 
	
	\noindent \textbf{Step (c):} Show the uniqueness of the \NOMsoltn solution. 
	
	By definition of $\Scal_\nu$ and \eqref{eq_10}, we know $\norm{\vX^\star} \leq \nu < \underline{\sigma}(\vP^\star)$, thus $\vP^\star+\vX^\star \succ 0$. This implies $\vP^\star_i+\vX^\star_i \succ 0$ for all $i$. By Lemma \ref{lemma_lqrSolPD}, we know  $\vPhat_{1:\numSys}:=\vP^\star_{1:\numSys}+\vX^\star_{1:\numSys}$ is the only possible solution to the \NOMcDARE $\text{cDARE}(\vAhat_{1:s}, \vBhat_{1:s}, \vThat)$ among $\curlybrackets{\vX_{1:\numSys}: \vX_i \succeq 0, \forall i}$. 
	
	Finally, note that
	$
	\norm{\vPhat_{1:\numSys} - \vP^\star_{1:\numSys}}
	=
	\norm{\vX^\star_{1:\numSys}} 
	=
	\norm{\vX^\star} 
	\leq
	\nu
	$
	where $\nu$ is defined in \eqref{eq_9}, which concludes the proof for Theorem \ref{prop_riccatiPertb}.
\end{proof}

\subsection{Proof of Theorem~\ref{thrm:meta plus perturbation}}

%\nec{@Yahya: start with a road map of what you are trying to do.}

%\begin{theorem*}[Meta Theorem] \todo{we should not have a meta theorem }
%	Let $\epsilon,\eta \geq 0$ be fixed scalars. Suppose $\norm{\vThat - \vT^\star}_\infty \leq \eta$, $\norm{\vAhat_i - \Abs_i} \leq \epsilon$, $\norm{\vBhat_i - \Bbs_i} \leq \epsilon$ and $\norm{\vPhat_i - \Pbs_i} \leq f(\epsilon,\eta)$ for all $i \in [s]$ and for some function $f$ such that $\max\{\epsilon,\eta\}\leq f(\epsilon,\eta)\leq \Gamma_\star$. Assume that the Markov chain $\{\omega(t)\}_{t=0}^\infty$ is ergodic. Let $\gamma>0$ such that $\rho(\tilde{\Lb}^\star) \leq \gamma < 1$. Then, under Assumption~\ref{asmp_1}, the certainty equivalent controller $\ub_{t} = \vKhat_{\omega(t)}\xb_{t}$ achieves
%	\begin{align*}
%		\Jhat - J^\star &\lesssim \sigma_w^2s\min\{n,p\}(\norm{\Rb_{1:s}}+\Gamma_\star^3) \Gamma_\star^6\frac{\tau(\Ltil,\gamma)}{1-\gamma} \frac{(\underline{\sigma}(\Rb_{1:s})+\Gamma_\star^3)^2}{\underline{\sigma}(\Rb_{1:s})^4}f(\epsilon,\eta)^2,
%	\end{align*}
%	as long as $f(\epsilon,\eta) \lesssim \frac{(1-\gamma)\underline{\sigma}(\Rbs_i)^2}{180s \Gamma_\star^6(\underline{\sigma}(\Rbs_i)+\Gamma_\star^3)\tau(\tilde{\Lb}^\star,\gamma)}$.
%\end{theorem*}
We first outline our high-level proof strategy for Theorem \ref{thrm:meta plus perturbation}:
\begin{itemize}
	\item We first bound the mismatch between the optimal controller $\Kbs_i$ and the certainty equivalent controller $\vKhat_i$ in terms of the upper bounds on the quantities $\norm{\vAhat_i - \Abs_i}$, $\norm{\vBhat_i - \Bbs_i}$, $\norm{\vPhat_i - \Pbs_i}$ and $\norm{\vThat - \vT^\star}_\infty$.
	%\item Next, we will show that in the limit when $T \to \infty$, the average cost~\eqref{eqn:MJS LQR} incurred by a stabilizing controller on the true MJS can be expressed as a cumulative cost~(scaled by the noise variance) incurred by the same controller on a certain noiseless MJS.
	\item Next, assuming the certainty equivalent controller stabilizes the MJS in the mean-square sense, we quantify the suboptimality gap $\Jhat - J^\star$ in terms of the controller mismatch $\norm{\Kbs_i - \vKhat_i}$.
	\item Lastly, using the matrix Fact~\ref{fact:rho perturbation}, we derive an upper bound on the mismatch $\norm{\Kbs_i - \vKhat_i}$ so that the certainty equivalent controller $\vKhat_i$ indeed stabilizes the MJS in the mean-square sense. Using the derived upper bound in the expression of $\Jhat - J^\star$, we get our final result.
\end{itemize}

\begin{lemma}[Controller mismatch]\label{thrm:P to K}
	Let $\epsilon,\eta >0$ be fixed scalars. Suppose  $\norm{\vThat - \vT^\star}_\infty \leq \eta$, $\norm{\vAhat_i - \Abs_i} \leq \epsilon$, $\norm{\vBhat_i - \Bbs_i} \leq \epsilon$ and $\norm{\vPhat_i - \Pbs_i} \leq f(\epsilon,\eta)$ for all $i \in [s]$ and for some function $f$ such that $\max\{\epsilon,\eta\}\leq f(\epsilon,\eta)\leq \Gamma_\star$. Then, under Assumption~\ref{asmp_1}, we have
	\[
	\norm{\Kbs_i - \vKhat_i} \leq 28 \Gamma_\star^3 \frac{(\underline{\sigma}(\Rb_i)+\Gamma_\star^3)}{\underline{\sigma}(\Rb_i)^2} f(\epsilon,\eta) \quad \text{for all}\quad i \in [s]
	\]
\end{lemma}
\begin{proof}
	To begin, recall that given $\Pbs_{1:s}$ and $\vPhat_{1:s}$, the optimal controller and the certainty equivalent controller is given by
	\begin{align*}
		\Kbs_i &= -\big(\Rb_i + \Bbst_i \varphis_i(\Pbs_{1:\numSys}) \Bbs_i\big)^{-1} \Bbst_i \varphis_i(\Pbs_{1:\numSys}) \Abs_i, \\ 
		\text{and} \quad \vKhat_i &= -\big(\Rb_i + \vBhat_i^\T \hat{\varphi}_i(\vPhat_{1:\numSys}) \vBhat_i\big)^{-1} \vBhat_i^\T \hat{\varphi}_i(\vPhat_{1:\numSys}) \vAhat_i,
	\end{align*}
	respectively for all $i \in [s]$. As an auxiliary step, we define $\tilde{\Kb}_i := -\big(\Rb_i + \vBhat_i^\T \varphis_i(\vPhat_{1:\numSys}) \vBhat_i\big)^{-1} \vBhat_i^\T \varphis_i(\vPhat_{1:\numSys}) \vAhat_i$. Then, we have $\norm{\Kbs_i - \vKhat_i} \leq \norm{\Kbs_i - \tilde{\Kb}_i} + \norm{\tilde{\Kb}_i  - \vKhat_i}$. We bound the first term on the right side of this inequality as follows,
	\begin{align}
		\nonumber 
		\norm{\Kbs_i - \tilde{\Kb}_i} &= \norm{\big(\Rb_i + \Bbst_i \varphis_i(\Pbs_{1:\numSys}) \Bbs_i\big)^{-1} \Bbst_i \varphis_i(\Pbs_{1:\numSys}) \Abs_i - \big(\Rb_i + \vBhat_i^\T \varphis_i(\vPhat_{1:\numSys}) \vBhat_i\big)^{-1} \vBhat_i^\T \varphis_i(\vPhat_{1:\numSys}) \vAhat_i}, \\
		&\leq \norm{\big(\Rb_i + \Bbst_i \varphis_i(\Pbs_{1:\numSys}) \Bbs_i\big)^{-1}}\norm{\Bbst_i \varphis_i(\Pbs_{1:\numSys}) \Abs_i - \vBhat_i^\T \varphis_i(\vPhat_{1:\numSys}) \vAhat_i} \nonumber\\
		&+ \norm{\big(\Rb_i + \Bbst_i \varphis_i(\Pbs_{1:\numSys}) \Bbs_i\big)^{-1} - \big(\Rb_i + \vBhat_i^\T \varphis_i(\vPhat_{1:\numSys}) \vBhat_i\big)^{-1}}\norm{\vBhat_i^\T \varphis_i(\vPhat_{1:\numSys}) \vAhat_i}.\label{eqn:split KKtilde}
	\end{align}
	Note that, since we are assuming that the cost matrices $\Qb_i$ and $\Rb_i$ are positive definite for all $i \in [s]$, therefore, we have
	\begin{align}
		\norm{\big(\Rb_i + \Bbst_i \varphis_i(\Pbs_{1:\numSys}) \Bbs_i\big)^{-1}} = \frac{1}{\underline{\sigma}(\Rb_i + \Bbst_i \varphis_i(\Pbs_{1:\numSys}) \Bbs_i)}\leq \frac{1}{\underline{\sigma}(\Rb_i)}. \label{eqn:Lambda_1i}
	\end{align}
	Next, we bound the difference,
	\begin{align}
		&\norm{\Bbst_i \varphis_i(\Pbs_{1:\numSys}) \Abs_i - \vBhat_i^\T \varphis_i(\vPhat_{1:\numSys}) \vAhat_i} \nn \\
		&\quad\quad= \norm{\Bbst_i \varphis_i(\Pbs_{1:\numSys}) \Abs_i - (\Bbs_i+\Delta_{\Bbs_i})^\T (\varphis_i(\Pbs_{1:\numSys})+\varphis_i(\Delta_{\Pbs_{1:\numSys}})) (\Abs_i+\Delta_{\Abs_i})},\nn \\
		&\quad\quad= \norm{\Bbst_i \varphis_i(\Pbs_{1:\numSys}) \Abs_i - (\Bbs_i+\Delta_{\Bbs_i})^\T(\varphis_i(\Pbs_{1:\numSys})\Abs_i+\varphis_i(\Delta_{\Pbs_{1:\numSys}})\Abs_i + \varphi_i(\Pbs_{1:\numSys})\Delta_{\Abs_i}+\varphis_i(\Delta_{\Pbs_{1:\numSys}})\Delta_{\Abs_i})},\nn \\
		&\quad\quad= \|\Bbst_i \varphis_i(\Pbs_{1:\numSys}) \Abs_i - (\Bbst_i\varphis_i(\Pbs_{1:\numSys})\Abs_i +\Bbst_i\varphis_i(\Delta_{\Pbs_{1:\numSys}})\Abs_i + \Bbst_i\varphis_i(\Pbs_{1:\numSys})\Delta_{\Abs_i} +\Bbst_i\varphis_i(\Delta_{\Pbs_{1:\numSys}})\Delta_{\Abs_i} \nn \\
		&\quad\quad+\Delta_{\Bbs_i}^\T\varphis_i(\Pbs_{1:\numSys})\Abs_i
		+\Delta_{\Bbs_i}^\T\varphis_i(\Delta_{\Pbs_{1:\numSys}})\Abs_i + \Delta_{\Bbs_i}^\T\varphis_i(\Pbs_{1:\numSys})\Delta_{\Abs_i}
		+\Delta_{\Bbs_i}^\T\varphis_i(\Delta_{\Pbs_{1:\numSys}})\Delta_{\Abs_i})\|,\nn \\
		&\quad\quad\leq \|\Abs_i\|\|\Bbs_i\|f(\epsilon,\eta)  +\|\Bbs_i\|\|\varphis_i(\Pbs_{1:\numSys})\|\epsilon + \norm{\Bbs_i}f(\epsilon,\eta)\epsilon + \|\Abs_i\|\|\varphis_i(\Pbs_{1:\numSys})\|\epsilon + \|\Abs_i\|f(\epsilon,\eta)\epsilon \nn \\
		&\quad\quad+ \|\varphis_i(\Pbs_{1:\numSys})\|\epsilon^2 + f(\epsilon,\eta)\epsilon^2, \nn \\
		&\quad\quad \leq (\|\Abs_i\|\|\Bbs_i\| + \|\Abs_i\|(\max_{i \in [\numSys]}\|\Pbs_i\|) + \|\Bbs_i\|(\max_{i \in [\numSys]}\|\Pbs_i\|)+\|\Abs_i\|\epsilon + \norm{\Bbs_i}\epsilon  + \max_{i \in [\numSys]}\|\Pbs_i\|\epsilon + \epsilon^2)f(\epsilon,\eta), \nn \\
		&\quad\quad \leq 3 \Gamma_\star^2f(\epsilon,\eta), \label{eqn:Lambda_2i}
	\end{align}
	where the second inequality follows from the assumption that  $ f(\epsilon,\eta) \geq \max\{\epsilon,\eta\}$. To proceed, we use the matrix identity $\mtx{X}^{-1} - \mtx{Y}^{-1} = \mtx{X}^{-1}(\mtx{Y} - \mtx{X})\mtx{Y}^{-1}$, to get the following norm bound,
	\begin{align}
		&\norm{\big(\Rb_i + \Bbst_i \varphis_i(\Pbs_{1:\numSys}) \Bbs_i\big)^{-1} - \big(\Rb_i + \vBhat_i^\T \varphis_i(\vPhat_{1:\numSys}) \vBhat_i\big)^{-1}}, \nn \\
		&\quad\quad = \norm{\big(\Rb_i + \Bbst_i \varphis_i(\Pbs_{1:\numSys}) \Bbs_i\big)^{-1}\big(\vBhat_i^\T \varphis_i(\vPhat_{1:\numSys}) \vBhat_i-\Bbst_i \varphis_i(\Pbs_{1:\numSys}) \Bbs_i\big)\big(\Rb_i + \vBhat_i^\T \varphis_i(\vPhat_{1:\numSys}) \vBhat_i\big)^{-1}}, \nn \\
		&\quad\quad\leq \norm{\big(\Rb_i + \Bbst_i \varphis_i(\Pbs_{1:\numSys}) \Bbs_i\big)^{-1}}\norm{\big(\Rb_i + \vBhat_i^\T \varphis_i(\vPhat_{1:\numSys}) \vBhat_i\big)^{-1}}\norm{\vBhat_i^\T \varphis_i(\vPhat_{1:\numSys}) \vBhat_i-\Bbst_i \varphis_i(\Pbs_{1:\numSys}) \Bbs_i}\nn \\
		&\quad\quad\leq \frac{1}{\underline{\sigma}(\Rb_i)^2}(\|\Bbs_i\|^2 + \|\Bbs_i\|(\max_{i \in [\numSys]}\|\Pbs_i\|) + \|\Bbs_i\|(\max_{i \in [\numSys]}\|\Pbs_i\|)+\|\Bbs_i\|\epsilon + \norm{\Bbs_i}\epsilon  + \max_{i \in [\numSys]}\|\Pbs_i\|\epsilon + \epsilon^2)f(\epsilon,\eta),\nn \\
		&\quad\quad \leq \frac{3\Gamma_\star^2f(\epsilon,\eta)}{\underline{\sigma}(\Rb_i)^2}. \label{eqn:Lambda_3i}
	\end{align}
	
	Using a similar idea, we get 
	\begin{align}
		\norm{\vBhat_i^\T \varphis_i(\vPhat_{1:\numSys}) \vAhat_i}&= \norm{(\Bbs_i+\Delta_{\Bbs_i})^\T (\varphis_i(\Pbs_{1:\numSys})+\varphis_i(\Delta_{\Pbs_{1:\numSys}})) (\Abs_i+\Delta_{\Abs_i})}, \nn \\
		&= \|\Bbst_i\varphis_i(\Pbs_{1:\numSys})\Abs_i +\Bbst_i\varphis_i(\Delta_{\Pbs_{1:\numSys}})\Abs_i + \Bbst_i\varphis_i(\Pbs_{1:\numSys})\Delta_{\Abs_i} +\Bbst_i\varphis_i(\Delta_{\Pbs_{1:\numSys}})\Delta_{\Abs_i} \nn \\
		&+\Delta_{\Bbs_i}^\T\varphis_i(\Pbs_{1:\numSys})\Abs_i
		+\Delta_{\Bbs_i}^\T\varphis_i(\Delta_{\Pbs_{1:\numSys}})\Abs_i + \Delta_{\Bbs_i}^\T\varphis_i(\Pbs_{1:\numSys})\Delta_{\Abs_i}
		+\Delta_{\Bbs_i}^\T\varphis_i(\Delta_{\Pbs_{1:\numSys}})\Delta_{\Abs_i}\|, \nn \\
		& \leq (\|\Abs_i\|\|\Bbs_i\| + \|\Abs_i\|(\max_{i \in [\numSys]}\|\Pbs_i\|) + \|\Bbs_i\|(\max_{i \in [\numSys]}\|\Pbs_i\|)+\|\Abs_i\|\epsilon + \norm{\Bbs_i}\epsilon  + \max_{i \in [\numSys]}\|\Pbs_i\|\epsilon + \epsilon^2) \nn \\
		&\quad\; f(\epsilon,\eta) + \norm{\Abs_i}\norm{\Bbs_i}(\max_{i \in [\numSys]}\norm{\Pbs_i}), \nn \\
		&\leq 3\Gamma_\star^2f(\epsilon,\eta) + \Gamma_\star^3. \label{eqn:Lambda_4i}
	\end{align}
	
	Substituting \eqref{eqn:Lambda_1i}, \eqref{eqn:Lambda_2i}, \eqref{eqn:Lambda_3i}, and \eqref{eqn:Lambda_4i} into \eqref{eqn:split KKtilde}, we obtain
	\begin{align}
		\norm{\Kbs_i - \tilde{\Kb}_i} \leq 12 \Gamma_\star^2\frac{(\underline{\sigma}(\Rb_i)+\Gamma_\star^3)}{\underline{\sigma}(\Rb_i)^2}f(\epsilon,\eta), \quad \text{for all} \quad i \in [s].\label{eqn:KKtilde}
		%&\leq \Lambda_{1,i}\Lambda_{2,i} + \Lambda_{3,i}\Lambda_{4,i}, \nn \\
		%&= (1+\frac{\norm{\Abs_i}\norm{\Bbs_i}(\max_{i \in [\numSys]}\norm{\Pbs_i})}{\underline{\sigma}(\Rbs_i)}) \nn \\
		%&\;\;\;\;\;\frac{1}{\underline{\sigma}(\Rbs_i)}(\|\Abs_i\|\|\Bbs_i\| + \|\Abs_i\|(\max_{i \in [\numSys]}\|\Pbs_i\|) + \|\Bbs_i\|(\max_{i \in [\numSys]}\|\Pbs_i\|)+\|\Abs_i\|\epsilon+ \norm{\Bbs_i}\epsilon  + \max_{i \in [\numSys]}\|\Pbs_i\|\epsilon + \epsilon^2)\epsilon \nn \\
		%&+ \frac{1}{\underline{\sigma}(\Rbs_i)^2}(\|\Abs_i\|\|\Bbs_i\| + \|\Abs_i\|(\max_{i \in [\numSys]}\|\Pbs_i\|) + \|\Bbs_i\|(\max_i\|\Pbs_i\|)+\|\Abs_i\|\epsilon + \norm{\Bbs_i}\epsilon  + \max_{i \in [\numSys]}\|\Pbs_i\|\epsilon + \epsilon^2)^2\epsilon^2, \nn \\
		%& \leq \frac{3(\underline{\sigma}(\Rbs_i) + \Gamma_\star^3)}{\underline{\sigma}(\Rbs_i)^2}(\Gamma_\star+\epsilon)^2\epsilon + \frac{3}{\underline{\sigma}(\Rbs_i)^2}(\Gamma_\star+\epsilon)^4\epsilon^2, \nn \\
		%&\leq \frac{3\epsilon}{\underline{\sigma}(\Rbs_i)^2}(\Gamma_\star+\epsilon)^2(\underline{\sigma}(\Rbs_i) +(\Gamma_\star+\epsilon)^2\epsilon+ \Gamma_\star^3),
	\end{align}
	This gives the final bound on $\norm{\Kbs_i - \tilde{\Kb}_i}$. Using similar proof techniques, we can also bound $\norm{\tilde{\Kb}_i  - \vKhat_i}$ as follows,
	\begin{align}
		\norm{\tilde{\Kb}_i  - \vKhat_i} &= \norm{\big(\Rb_i + \vBhat_i^\T \varphis_i(\vPhat_{1:\numSys}) \vBhat_i\big)^{-1} \vBhat_i^\T \varphis_i(\vPhat_{1:\numSys}) \vAhat_i - \big(\Rb_i + \vBhat_i^\T \varphihat_i(\vPhat_{1:\numSys}) \vBhat_i\big)^{-1} \vBhat_i^\T \varphihat_i(\vPhat_{1:\numSys}) \vAhat_i}, \nn \\
		%&\leq  \norm{\big(\Rbs_i + \vBhat_i^\T \varphi_i(\vPhat_{1:\numSys}) \vBhat_i\big)^{-1} \vBhat_i^\T \varphi_i(\vPhat_{1:\numSys})  - \big(\Rbs_i + \vBhat_i^\T \varphihat_i(\vPhat_{1:\numSys}) \vBhat_i\big)^{-1} \vBhat_i^\T \varphihat_i(\vPhat_{1:\numSys})}\norm{\vAhat_i} ,\nn \\
		&\leq \norm{\big(\Rb_i + \vBhat_i^\T \varphis_i(\vPhat_{1:\numSys}) \vBhat_i\big)^{-1}}\norm{\vBhat_i^\T \varphis_i(\vPhat_{1:\numSys}) \vAhat_i - \vBhat_i^\T \varphihat_i(\vPhat_{1:\numSys}) \vAhat_i} \nn \\
		&+\norm{\big(\Rb_i + \vBhat_i^\T \varphis_i(\vPhat_{1:\numSys}) \vBhat_i\big)^{-1}-\big(\Rb_i + \vBhat_i^\T \varphihat_i(\vPhat_{1:\numSys}) \vBhat_i\big)^{-1}}\norm{\vBhat_i^\T \varphihat_i(\vPhat_{1:\numSys}) \vAhat_i}.\label{eqn:split KtildeKhat}
	\end{align}
	In the following, we will bound each norm in the above expression separately to get a bound on $\norm{\tilde{\Kb}_i - \vKhat_i}$. First, we have
	\begin{align}
		\norm{\big(\Rb_i + \vBhat_i^\T \varphis_i(\vPhat_{1:\numSys}) \vBhat_i\big)^{-1}} = \frac{1}{\underline{\sigma}(\Rb_i + \vBhat_i^\T \varphis_i(\vPhat_{1:\numSys}) \vBhat_i)} \leq \frac{1}{\underline{\sigma}(\Rb_i)}. \label{eqn:Lambda_5i}	
	\end{align}
	Next, we bound the difference,
	\begin{align}
		&\norm{\vBhat_i^\T \varphis_i(\vPhat_{1:\numSys}) \vAhat_i - \vBhat_i^\T \varphihat_i(\vPhat_{1:\numSys}) \vAhat_i}, \nn \\
		&\quad\quad\leq \norm{\vBhat_i} \norm{\varphis_i(\vPhat_{1:\numSys}) - \varphihat_i(\vPhat_{1:\numSys})}\norm{\vAhat_i}, \nn \\
		&\quad\quad \leq \norm{\Abs_i+\Delta_{\Abs_i}}\norm{\Bbs_i+\Delta_{\Bbs_i}}\norm{\sum_{j \in [\numSys]} ([\vT^\star]_{ij}-[\vThat]_{ij}) \vPhat_j} ,\nn \\
		&\quad\quad\leq (\norm{\Abs_i} + \epsilon)(\norm{\Bbs_i} + \epsilon)(\eta\max_{i \in [\numSys]}\norm{\Pbs_i + \Delta_{\Pbs_i}}), \nn \\
		&\quad\quad\leq \eta(\norm{\Abs_i}\norm{\Bbs_i}+ \norm{\Abs_i}\epsilon+\norm{\Bbs_i}\epsilon + \epsilon^2)(\max_{i \in [\numSys]}\norm{\Pbs_i}+f(\epsilon,\eta)), \nn \\
		&\quad\quad\leq \eta\bigg((\|\Abs_i\|\|\Bbs_i\| + \|\Abs_i\|(\max_{i \in [\numSys]}\|\Pbs_i\|) + \|\Bbs_i\|(\max_{i \in [s]}\|\Pbs_i\|)+\|\Abs_i\|\epsilon + \norm{\Bbs_i}\epsilon  + \max_{i \in [\numSys]}\|\Pbs_i\|\epsilon + \epsilon^2)f(\epsilon,\eta) \nn \\
		& \quad\quad+ \norm{\Abs_i}\norm{\Bbs_i}(\max_{i \in [\numSys]}\norm{\Pbs_i})\bigg), \nn \\
		&\quad\quad \leq \eta(3\Gamma_\star^2f(\epsilon,\eta) + \Gamma_\star^3). \label{eqn:Lambda_6i}
	\end{align}
	To proceed, using matrix identity $\mtx{X}^{-1} - \mtx{Y}^{-1} = \mtx{X}^{-1}(\mtx{Y} - \mtx{X})\mtx{Y}^{-1}$, we get the following norm bound,
	\begin{align}
		&\norm{\big(\Rb_i + \vBhat_i^\T \varphis_i(\vPhat_{1:\numSys}) \vBhat_i\big)^{-1}-\big(\Rb_i + \vBhat_i^\T \varphihat_i(\vPhat_{1:\numSys}) \vBhat_i\big)^{-1}}, \nn \\
		&\quad\quad= \norm{\big(\Rb_i + \vBhat_i^\T \varphis_i(\vPhat_{1:\numSys}) \vBhat_i\big)^{-1}\big(\vBhat_i^\T \varphihat_i(\vPhat_{1:\numSys})\vBhat_i - \vBhat_i^\T \varphis_i(\vPhat_{1:\numSys})\vBhat_i\big)\big(\Rb_i + \vBhat_i^\T \varphihat_i(\vPhat_{1:\numSys}) \vBhat_i\big)^{-1}}, \nn \\
		&\quad\quad \leq \norm{\big(\Rb_i + \vBhat_i^\T \varphis_i(\vPhat_{1:\numSys}) \vBhat_i\big)^{-1}}\norm{\big(\Rb_i + \vBhat_i^\T \varphihat_i(\vPhat_{1:\numSys}) \vBhat_i\big)^{-1}}\norm{\vBhat_i^\T \varphihat_i(\vPhat_{1:\numSys})\vBhat_i - \vBhat_i^\T \varphis_i(\vPhat_{1:\numSys})\vBhat_i}, \nn \\
		&\quad\quad\leq \frac{\eta}{\underline{\sigma}(\Rb_i)^2}\bigg((\|\Bbs_i\|^2 + \|\Bbs_i\|(\max_{i \in [\numSys]}\|\Pbs_i\|) + \|\Bbs_i\|(\max_{i \in [\numSys]}\|\Pbs_i\|)+\|\Bbs_i\|\epsilon + \norm{\Bbs_i}\epsilon  + \max_{i \in [\numSys]}\|\Pbs_i\|\epsilon + \epsilon^2)f(\epsilon,\eta) \nn \\
		&\quad\quad + \norm{\Bbs_i}^2(\max_{i \in [\numSys]}\norm{\Pbs_i})\bigg), \nn \\
		&\quad\quad \leq \frac{\eta(3\Gamma_\star^2f(\epsilon,\eta) + \Gamma_\star^3)}{\underline{\sigma}(\Rb_i)^2}. \label{eqn:Lambda_7i}
	\end{align}
	Using similar idea, we get the following norm bound, 
	\begin{align}
		\norm{\vBhat_i^\T \varphihat_i(\vPhat_{1:\numSys}) \vAhat_i} &= \norm{(\Bbs_i+\Delta_{\Bbs_i})^\T (\varphihat_i(\Pbs_{1:\numSys})+\varphihat_i(\Delta_{\Pbs_{1:\numSys}})) (\Abs_i+\Delta_{\Abs_i})}, \nn \\
		&= \|\Bbst_i\varphihat_i(\Pbs_{1:\numSys})\Abs_i +\Bbst_i\varphihat_i(\Delta_{\Pbs_{1:\numSys}})\Abs_i + \Bbst_i\varphihat_i(\Pbs_{1:\numSys})\Delta_{\Abs_i} +\Bbst_i\varphihat_i(\Delta_{\Pbs_{1:\numSys}})\Delta_{\Abs_i} \nn \\
		&+\Delta_{\Bbs_i}^\T\varphihat_i(\Pbs_{1:\numSys})\Abs_i
		+\Delta_{\Bbs_i}^\T\varphihat_i(\Delta_{\Pbs_{1:\numSys}})\Abs_i + \Delta_{\Bbs_i}^\T\varphihat_i(\Pbs_{1:\numSys})\Delta_{\Abs_i}
		+\Delta_{\Bbs_i}^\T\varphihat_i(\Delta_{\Pbs_{1:\numSys}})\Delta_{\Abs_i}\|, \nn \\
		& \leq (\|\Abs_i\|\|\Bbs_i\| + \|\Abs_i\|(\max_{i \in [\numSys]}\|\Pbs_i\|) + \|\Bbs_i\|(\max_{i \in [\numSys]}\|\Pbs_i\|)+\|\Abs_i\|\epsilon + \norm{\Bbs_i}\epsilon  + \max_{i \in [\numSys]}\|\Pbs_i\|\epsilon + \epsilon^2) \nn \\
		& \quad \; f(\epsilon,\eta)+ \norm{\Abs_i}\norm{\Bbs_i}(\max_{i \in [\numSys]}\norm{\Pbs_i}), \nn \\
		& \leq 3\Gamma_\star^2f(\epsilon,\eta) + \Gamma_\star^3. \label{eqn:Lambda_8i} 
	\end{align}
	
	Substituting \eqref{eqn:Lambda_5i}, \eqref{eqn:Lambda_6i}, \eqref{eqn:Lambda_7i} and \eqref{eqn:Lambda_8i} into \eqref{eqn:split KtildeKhat}, we get the following norm bound,
	\begin{align}
		\norm{\tilde{\Kb}_i- \vKhat_i} &\leq 16 \Gamma_\star^3 \frac{(\underline{\sigma}(\Rb_i)+\Gamma_\star^3)}{\underline{\sigma}(\Rb_i)^2}\eta.\label{eqn:KtildeKhat}
		%&\leq \frac{\eta}{\underline{\sigma}(\Rbs_i)}\bigg((\|\Abs_i\|\|\Bbs_i\| + \|\Abs_i\|(\max_{i \in [\numSys]}\|\Pbs_i\|) + \|\Bbs_i\|(\max_i\|\Pbs_i\|)+\|\Abs_i\|\epsilon + \norm{\Bbs_i}\epsilon  + \max_{i \in [\numSys]}\|\Pbs_i\|\epsilon + \epsilon^2)\epsilon \nn \\
		%& + \norm{\Abs_i}\norm{\Bbs_i}(\max_{i \in [\numSys]}\norm{\Pbs_i})\bigg) + \frac{\eta}{\underline{\sigma}(\Rbs_i)^2}\bigg((\|\Abs_i\|\|\Bbs_i\| + \|\Abs_i\|(\max_{i \in [\numSys]}\|\Pbs_i\|) + \|\Bbs_i\|(\max_i\|\Pbs_i\|)+\|\Abs_i\|\epsilon \nn \\
		%&+ \norm{\Bbs_i}\epsilon  + \max_{i \in [\numSys]}\|\Pbs_i\|\epsilon + \epsilon^2)\epsilon+ \norm{\Abs_i}\norm{\Bbs_i}(\max_{i \in [\numSys]}\norm{\Pbs_i})\bigg)^2, \nn \\
		%&\leq \frac{3\eta}{\underline{\sigma}(\Rbs_i)^2}\big((\Gamma_\star+\epsilon)^2\epsilon+\Gamma_\star^3\big)\big(\underline{\sigma}(\Rbs_i)+ (\Gamma_\star+\epsilon)^2\epsilon+\Gamma_\star^3\big)
	\end{align}
	This gives the final bound on $\norm{\tilde{\Kb}_i- \vKhat_i}$. Finally, we combine the two norm bounds \eqref{eqn:KKtilde} and \eqref{eqn:KtildeKhat}, to obtain the statement of the theorem,
	\begin{align}
		\norm{\Kbs_i - \vKhat_i} &\leq \norm{\Kbs_i - \tilde{\Kb}_i} + \norm{\tilde{\Kb}_i  - \vKhat_i}, \nn \\ 
		&\leq 12 \Gamma_\star^2\frac{(\underline{\sigma}(\Rb_i)+\Gamma_\star^3)}{\underline{\sigma}(\Rb_i)^2}f(\epsilon,\eta) + 16 \Gamma_\star^3\frac{(\underline{\sigma}(\Rb_i)+\Gamma_\star^3)}{\underline{\sigma}(\Rb_i)^2}\eta, \nn \\
		%&\leq 16 \Gamma_\star^3\frac{(\underline{\sigma}(\Rb_i)+\Gamma_\star^3)}{\underline{\sigma}(\Rb_i)^2}(f(\epsilon,\eta)+\eta), \nn \\
		&\leq 28 \Gamma_\star^3\frac{(\underline{\sigma}(\Rb_i)+\Gamma_\star^3)}{\underline{\sigma}(\Rb_i)^2}f(\epsilon,\eta),
	\end{align}
	for all $i \in [s]$. This completes the proof.
\end{proof}
\begin{lemma}[Suboptimality gap]\label{thrm:K to J2}
	Let $\epsilon,\eta >0$ be fixed scalars. Suppose $\norm{\vThat - \vT^\star}_\infty \leq \eta$, $\norm{\vAhat_i - \Abs_i} \leq \epsilon$, $\norm{\vBhat_i - \Bbs_i} \leq \epsilon$ and $\norm{\vPhat_i - \Pbs_i} \leq f(\epsilon,\eta)$ for all $i \in [s]$ and for some function $f$ such that $\max\{\epsilon,\eta\}\leq f(\epsilon,\eta)\leq \Gamma_\star$. Assume that the Markov chain $\{\omega(t)\}_{t=0}^\infty$ is ergodic and let $\gamma>0$ such that $\rho(\tilde{\Lb}^\star) \leq \gamma < 1$. Then, under Assumption~\ref{asmp_1}, as long as $f(\epsilon,\eta) \lesssim \frac{(1-\gamma)\underline{\sigma}(\Rbs_i)^2}{180s \Gamma_\star^6(\underline{\sigma}(\Rbs_i)+\Gamma_\star^3)\tau(\tilde{\Lb}^\star,\gamma)}$, the certainty equivalent controller $\ub_{t} = \vKhat_{\omega(t)}\xb_{t}$ achieves
	\[
	\Jhat - J^\star \lesssim \sigma_w^2s\min\{n,p\}(\norm{\Rb_{1:s}}+\Gamma_\star^3) \Gamma_\star^6 \frac{\tau(\vLtil^\star,\gamma)(\underline{\sigma}(\Rb_{1:s})+\Gamma_\star^3)^2}{(1-\gamma)\underline{\sigma}(\Rb_{1:s})^4}f(\epsilon,\eta)^2.
	\]
\end{lemma}
\begin{proof}
	To prove this lemma, we need to quantify the suboptimality gap $\Jhat - J^\star$ in terms of the controller mismatch $\norm{\Kbs_i - \vKhat_i}$ and derive an upper bound on this mismatch so that the certainty equivalent controller $\vKhat_i$ stabilizes the MJS in the mean-square sense. For this purpose, recall that the goal of infinite time horizon LQR problem is to solve the following optimization problem,
	\begin{align}
		\inf_{\{\ub_0,\ub_1,\dots\}} \limsup_{T \to \infty}\expctn \bigg[\frac{1}{T}\sum_{t=0}^{T}\vx_t^\T\Qb_{\omega(t)}\vx_t + \ub_t^\T\Rb_{\omega(t)}\ub_t\bigg] \quad \text{s.t.} \quad \vx_{t+1} = \Abs_{\omega(t)}\vx_t + \Bbs_{\omega(t)}\ub_t + \vw_t,\label{eqn:MJS LQR}
	\end{align}
	where the expectation is taken over the initial state $\vx_0 \sim \Nc(0,\Iden_n)$, Markovian modes $\{\omega(t)\}_{t = 0}^{\infty}$, and the i.i.d. noise $\{\vw_t\}_{t = 0}^{\infty} \distas \Nc(0,\sigma_w^2\Iden_n)$. If the input is given by $\ub_t = \Kb_{\omega(t)}\xb_t$, then, setting $\Lb_{\omega(t)} = \Abs_{\omega(t)} + \Bbs_{\omega(t)}\Kb_{\omega(t)}$, the state updates as follows,
	\begin{align}
		\xb_{t+1} = \Lb_{\omega(t)}\xb_t + \wb_t \implies \xb_t = \begin{cases}
			\xb_0 &\text{if} \quad t=0, \\
			\Lb_{\omega(0)}\xb_0 + \wb_0 &\text{if} \quad t=1, \\
			\prod_{i=0}^{t-1}\Lb_{\omega(i)}\xb_0 + \sum_{\tau=0}^{t-2}\prod_{j=\tau+1}^{t-1}\Lb_{\omega(j)}\wb_\tau + \wb_{t-1} &\text{if}\quad t \geq 2
		\end{cases} \label{eqn:state update} 
	\end{align}
	where $\prod_{i=0}^{t-1}\Lb_{\omega(i)} = \Lb_{\omega(t-1)}\Lb_{\omega(t-2)} \cdots\Lb_{\omega(0)}$. For ease of notation, we set $\Cb_{\omega(t)}:=\Qb_{\omega(t)} + \Kb_{\omega(t)}^\T\Rb_{\omega(t)}\Kb_{\omega(t)}$. Then, using~\eqref{eqn:state update} and the independence of $\xb_0, \{\wb_t\}_{t=0}^\infty$ and $\{\omega(t)\}_{t=0}^\infty$, the MJS cost function~\eqref{eqn:MJS LQR} can be simplified as follows,
	\begin{align}
		J(\Abs_{1:s},\Bbs_{1:s},\Kb_{1:s}) &= \limsup_{T \to \infty}\expctn \bigg[\frac{1}{T}\sum_{t=0}^{T}\xb_t^\T(\Qb_{\omega(t)} + \Kb_{\omega(t)}^\T\Rb_{\omega(t)}\Kb_{\omega(t)})\xb_t\bigg], \nn \\
		&=\limsup_{T \to \infty} \frac{1}{T}\bigg[\E\big[\xb_0^\T\Cb_{\omega(0)}\xb_0\big]+\E\big[\xb_0^\T\Lb_{\omega(0)}^\T\Cb_{\omega(1)}\Lb_{\omega(0)}\xb_0\big] \nn \\
		&+\E\big[\wb_0^\T\Cb_{\omega(1)}\wb_0\big]+\sum_{t=2}^{T}\E\big[\xb_0^\T\big(\prod_{i=0}^{t-1}\Lb_{\omega(i)}\big)^\T\Cb_{\omega(t)}\big(\prod_{i=0}^{t-1}\Lb_{\omega(i)}\big)\xb_0\big] \nn \\
		&+\sum_{t=2}^{T}\sum_{\tau=0}^{t-2}\E\big[\wb_\tau^\T\big(\prod_{j=\tau+1}^{t-1}\Lb_{\omega(j)}\big)^\T\Cb_{\omega(t)}\big(\prod_{j=\tau+1}^{t-1}\Lb_{\omega(j)}\big)\wb_\tau\big]+\sum_{t=2}^{T}\E\big[\wb_{t-1}^\T\Cb_{\omega(t)}\wb_{t-1}\big]\bigg]. \nn
	\end{align}
	To proceed, we solve the expectations with respect to $\vx_0 \sim \Nc(0,\Iden_n)$ and $\{\vw_t\}_{t = 0}^{\infty} \distas \Nc(0,\sigma_w^2\Iden_n)$ while keeping the expectation with respect to the Markovian modes $\{\omega(t)\}_{t=0}^\infty$ and use the linearity of trace operator to obtain the following, 
	\begin{align}
		J(\Abs_{1:s},\Bbs_{1:s},\Kb_{1:s}) 
		&=\limsup_{T \to \infty} \frac{1}{T}\E\bigg[\tr\bigg(\Cb_{\omega(0)}+\Lb_{\omega(0)}^\T\Cb_{\omega(1)}\Lb_{\omega(0)}+\sum_{t=2}^{T}\big(\prod_{i=0}^{t-1}\Lb_{\omega(i)}\big)^\T\Cb_{\omega(t)}\big(\prod_{i=0}^{t-1}\Lb_{\omega(i)}\big)\bigg)\bigg] \nn \\
		&+ \sigma_w^2\limsup_{T \to \infty} \frac{1}{T}\E\bigg[\tr\bigg(\Cb_{\omega(1)}+\sum_{t=2}^{T}\sum_{\tau=0}^{t-2}\big(\prod_{j=\tau+1}^{t-1}\Lb_{\omega(j)}\big)^\T\Cb_{\omega(t)}\big(\prod_{j=\tau+1}^{t-1}\Lb_{\omega(j)}\big)+\sum_{t=2}^{T}\Cb_{\omega(t)}\bigg)\bigg], \nn \\
		&=\limsup_{T \to \infty} \frac{1}{T}\E\bigg[\tr\bigg(\Cb_{\omega(0)}+\sum_{t=1}^{T}\big(\prod_{i=0}^{t-1}\Lb_{\omega(i)}\big)^\T\Cb_{\omega(t)}\big(\prod_{i=0}^{t-1}\Lb_{\omega(i)}\big)\bigg)\bigg] \nn \\
		&+ \sigma_w^2 \limsup_{T \to \infty} \frac{1}{T}\E\bigg[\sum_{\tau=1}^{T}\tr\bigg(\Cb_{\omega(\tau)}+\sum_{t=\tau+1}^{T}\big(\prod_{i=\tau}^{t-1}\Lb_{\omega(i)}\big)^\T\Cb_{\omega(t)}\big(\prod_{i=\tau}^{t-1}\Lb_{\omega(i)}\big)\bigg)\bigg], \nn \\ 
		&=\limsup_{T \to \infty} \frac{1}{T}\E\bigg[\xb_0^\T\Cb_{\omega(0)}\xb_0+\xb_0^\T\sum_{t=1}^{T}\big(\prod_{i=0}^{t-1}\Lb_{\omega(i)}\big)^\T\Cb_{\omega(t)}\big(\prod_{i=0}^{t-1}\Lb_{\omega(i)}\big)\xb_0\bigg] \nn \\
		&+ \sigma_w^2 \limsup_{T \to \infty} \frac{1}{T}\E\bigg[\sum_{\tau=1}^{T}\xb_0^\T\Cb_{\omega(\tau)}\xb_0+\xb_0^\T\sum_{t=\tau+1}^{T}\big(\prod_{i=\tau}^{t-1}\Lb_{\omega(i)}\big)^\T\Cb_{\omega(t)}\big(\prod_{i=\tau}^{t-1}\Lb_{\omega(i)}\big)\xb_0\bigg], \label{eqn:limit_split}
		%&= \limsup_{T \to \infty}\frac{1}{T} \Jbar(\Abs_{1:s},\Bbs_{1:s},\Kb_{1:s}) + \sigma_w^2 \Jtil(\Abs_{1:s},\Bbs_{1:s},\Kb_{1:s})
	\end{align}
	where, the expectation is with respect to $\xb_{0} \sim \Nc(0,\Iden_n)$ and the Markovian modes $\{\omega(t)\}_{t=0}^\infty$. Let $\Mcal_{\text{init}}$ denotes the original Markov chain and $\Mcal_{\text{final}}$ denotes a new Markov chain with the same transition matrix $\Tbs$ but with the initial distribution as the stationary distribution of $\Mcal_{\text{init}}$. Observe that, for the Markov chain $\Mcal_{\text{final}}$, we have $\omega(0) \sim \omega(1) \sim \cdots \sim \vpi_{\Tbs}$. Using these, we define the following two cost functions for noiseless Markov jump systems,
	%Let $\Jbar(\Abs_{1:s},\Bbs_{1:s},\Kb_{1:s})$  and $\Jtil(\Abs_{1:s},\Bbs_{1:s},\Kb_{1:s})$ are the costs achieved by $\Kb_{1:\numSys}$ when the state updates are noiseless and the initial distribution of the Markov chain is $\Dcal_0$ and $\Dcal_s$ respectively. Specifically, we have
	\begin{align}
		\Jbar_{\Mcal_{\text{init}}}(\Abs_{1:s},\Bbs_{1:s},\Kb_{1:s}):&=	\E_{\xbb_0\sim\Nc(0,\Iden_n),\{\omega_t\}_{t=0}^\infty \sim \Mcal_{\text{init}}}\bigg[\sum_{t=0}^{\infty}\xbb_t^\T(\Qb_{\omega(t)} + \Kb_{\omega(t)}^\T\Rb_{\omega(t)}\Kb_{\omega(t)})\xbb_t\bigg],  \\
		%\text{s.t.} \quad \vx_{t+1} &= (\vA_{\omega(t)} + \Kb_{\omega(t)}^\T\Bbs_{\omega(t)}\Kb_{\omega(t)})\vx_t, \nn \\
		\Jbar_{\Mcal_{\text{final}}}(\Abs_{1:s},\Bbs_{1:s},\Kb_{1:s}):&=	\E_{\xbb_0\sim\Nc(0,\Iden_n),\{\omega_t\}_{t=0}^\infty \sim \Mcal_{\text{final}}}\bigg[\sum_{t=0}^{\infty}\xbb_t^\T(\Qb_{\omega(t)} + \Kb_{\omega(t)}^\T\Rb_{\omega(t)}\Kb_{\omega(t)})\xbb_t\bigg],
		%\text{s.t.} \quad \vx_{t+1} &= (\vA_{\omega(t)} + \Kb_{\omega(t)}^\T\Bbs_{\omega(t)}\Kb_{\omega(t)})\vx_t,
	\end{align}
	subject to $\xbb_{t+1} = (\Abs_{\omega(t)} + \Bbs_{\omega(t)}\Kb_{\omega(t)})\xbb_t$. To proceed, using the assumption that the original Markov chain is ergodic, in the limit when $T \to \infty$, equation~\eqref{eqn:limit_split} becomes
	\begin{align}
		J(\Abs_{1:s},\Bbs_{1:s},\Kb_{1:s})  = \limsup_{T \to \infty}\frac{1}{T}\Jbar_{\Mcal_{\text{init}}}(\Abs_{1:s},\Bbs_{1:s},\Kb_{1:s}) + \sigma_w^2 \Jbar_{\Mcal_{\text{final}}}(\Abs_{1:s},\Bbs_{1:s},\Kb_{1:s}).
	\end{align}
	Recall that, by assumption we have $\rho(\tilde{\Lb}^\star)\leq \gamma <1$ for some $\gamma >0$. From Proposition~\ref{prop MSS}, this is equivalent to saying that the optimal controller $\Kbs_{1:s}$ stabilizes the noiseless MJS in the mean-square sense. If additionally, the certainty equivalent controller $\vKhat_{1:s}$ stabilizes the noiseless MJS in the mean-square sense, the problem of finding the suboptimality gap for the LQR problem~\eqref{eqn:MJS LQR} reduces to finding the suboptimaly gap for a corresponding noiseless problem as follows,
	\begin{align}
		\Jhat - J^\star = \sigma_w^2\big(\Jbar_{\Mcal_{\text{final}}}(\Abs_{1:s},\Bbs_{1:s},\vKhat_{1:s})-\Jbar_{\Mcal_{\text{final}}}(\Abs_{1:s},\Bbs_{1:s},\Kbs_{1:s})\big). \label{eqn:sub-opt reduction} 
	\end{align}
	Towards the end of this proof, we will show that when $\Kbs_{1:s}$ stabilizes the noiseless MJS in the mean-square sense and the Riccati perturbations $f(\epsilon,\eta)$ are sufficiently small, then $\vKhat_{1:s}$ also stabilizes the noiseless MJS in the mean-square sense. Before that, we introduce a few more concepts and definitions that will be used in the remaining proof.
	\begin{definition*}
		Given a noiseless closed loop MJS, $\xbb_{t+1} = (\Abs_{\omega(t)} + \Bbs_{\omega(t)}\Kb_{\omega(t)})\xbb_{t}$, we define the following quantities:
		
		\noindent {\bfseries (a)} We denote by $\bSi[\xbb_t]=\sum_{i \in [\numSys]}\bSi_i[\xbb_t]$ the covariance matrix of the state $\xbb_t$, where $\bSi_i[\xbb_t] := \E[\xbb_t\xbb_t^\T\indicator_{\omega(t)=i}].$
		
		\noindent {\bfseries (b)} When $\Kb_{1:\numSys}$ stabilizes the noiseless MJS in the mean-square sense, we know $\sum_{t=0}^\infty\bSi[\xbb_t]$ exists. We denote this limit as $\bSi^{\Kb}$ and we have $\bSi^{\Kb} = \sum_{t=0}^{\infty}\Tc^t(\bSi[\xbb_0])$, where $\Tc(\Vb_{1:s}) := (\Tc_1(\Vb_{1:s}),\dots,\Tc_s(\Vb_{1:s}))$ and $\Tc_j(\Vb_{1:s}) = \sum_{i=1}^s [\vT^\star]_{ij}(\Abs_i + \Bbs_i \Kb_i)^\T\Vb_i(\Abs_i + \Bbs_i\Kb_i)$. 
		%\begin{align}
		%	\bSi^{\Kb} = \sum_{t=0}^{\infty}\Tc^t(\bSi[\xbb_0]), \quad \text{where} \quad &\Tc(\Vb_{1:s}) := (\Tc_1(\Vb_{1:s}),\dots,\Tc_s(\Vb_{1:s})) ,\nn \\
		%	 \text{and}\quad &\Tc_j(\Vb_{1:s}) = \sum_{i=1}^s [\vT^\star]_{ij}(\Abs_i + \Bbs_i \Kb_i)^\T\Vb_i(\Abs_i + \Bbs_i\Kb_i). \nn
		%\implies \bSi_i[\xb_{t+1}] &= \Tc_i(\bSi_{1:s}[\xb_t]).
		%\end{align}
		
		{\bfseries (c)} We denote by $\{\Pb_i^{\Kbs}\}_{i=1}^{\numSys}$ the unique positive definite solution of the following coupled Lyapunov equations for MJS,
		\begin{align}\label{eq_Lyapunov}
			\Xb_i^{\Kbs} = \Qb_i + \Kb_i^{\star \T}\Rb_i\Kbs_i +  (\Abs_i + \Bbs_i\Kbs_i)^\T\varphis_i(\Xb_{1:s}^{\Kbs})(\Abs_i + \Bbs_i\Kbs_i), \quad \text{for all} \quad i \in [\numSys].
		\end{align}
		
		{\bfseries (d)} The optimal control problem for noiseless MJS is the following infinite time horizon linear quadratic regulator problem
		%\begin{equation}
		\begin{align}
			\inf_{\{\ub_0,\ub_1,\dots\}}&\expctn \bigg[\sum_{t=0}^{\infty}\xbb_t^\T\Qb_{\omega(t)}\xbb_t + \ub_t^\T\Rb_{\omega(t)}\ub_t\bigg] \quad \text{s.t.} \quad \xbb_{t+1} = \Abs_{\omega(t)}\xbb_t + \Bbs_{\omega(t)}\ub_t,
		\end{align}
		where the expectation is over the initial state $\xbb_{0} \sim \Nc(0,\Iden_n)$ and the Markovian modes $\{\omega(t)\}_{t=0}^\infty$.
		
		{\bfseries (e)} We define $\Jbar(\Abs_{1:s},\Bbs_{1:s},\Kb_{1:s}) :=\expctn \bigg[\sum_{t=0}^{\infty}\xbb_t^\T\big(\Qb_{\omega(t)} + \Kb_{\omega(t)}^\T\Rb_{\omega(t)}\Kb_{\omega(t)}\big)\xbb_t\bigg] \;\; \text{s.t.} \;\; \xbb_{t+1} = \big(\Abs_{\omega(t)} + \Bbs_{\omega(t)}\Kb_{\omega(t)}\big)\xbb_t$.
	\end{definition*} 
	We are now ready to state a lemma which bounds the suboptimality gap of a noiseless LQR problem in terms of the mismatch between the optimal controller $\Kbs_{1:s}$ and the certainty equivalent controller $\vKhat_{1:s}$.
	\begin{lemma*}[Lemma 3 of \cite{jansch2020policy}]
		Suppose, $\Kbs_{1:s}$ and $\vKhat_{1:s}$ stabilize the noiseless system $\xbb_{t+1} = \Abs_{\omega(t)}\xbb_t + \Bbs_{\omega(t)}\ub_t$ in the mean-square sense. Then, the costs incurred by the optimal controller $\Kbs_{1:s}$ and certainty equivalent controller $\vKhat_{1:s}$ on the true system satisfy  
		\begin{align}
			\Jbar(\Abs_{1:s},\Bbs_{1:s},\vKhat_{1:s})-\Jbar(\Abs_{1:s},\Bbs_{1:s},\Kbs_{1:s}) &= \sum_{i=1}^s\tr\big(\bSi_i^{\vKhat}(\Kbs_i-\vKhat_i)^\T(\Rb_i+\Bbst_i\varphis_i(\Pb_{1:\numSys}^{\Kbs})\Bbs_i)(\Kbs_i - \vKhat_i)\big). \nn 
		\end{align}
	\end{lemma*}
	Recall that, by assumption $\Kbs_{1:s}$ stabilizes the noiseless MJS in the mean-square sense. If we assume, the certainty equivalent controller $\vKhat_{1:s}$also stabilizes the noiseless MJS in the mean-square sense, we can use the above Lemma in \eqref{eqn:sub-opt reduction} to get the following suboptimality bound for the LQR problem~\eqref{eqn:MJS LQR},
	\begin{align}
		\Jhat - J^\star & \leq \sigma_w^2\sum_{i=1}^s\norm{\bSi_i^{\vKhat}}\norm{\Rb_i+\Bbst_i\varphi_i(\Pb_{1:\numSys}^{\Kbs})\Bbs_i}\tf{\Kbs_i-\vKhat_i}^2 \nn \\
		& \leq \sigma_w^2\sum_{i=1}^s
		\min\{n,p\}\norm{\bSi_i^{\vKhat}}(\norm{\Rb_i}+\norm{\Bbs_i}^2(\max_{i\in[s]}\norm{\Pb_i^{\Kbs}}))\norm{\Kbs_i-\vKhat_i}^2, \nn \\
		& \leq 800\sigma_w^2\min\{n,p\}s\norm{\bSi^{\vKhat}}(\norm{\Rb_{1:s}}+\Gamma_\star^3) \Gamma_\star^6 \frac{(\underline{\sigma}(\Rb_{1:s})+\Gamma_\star^3)^2}{\underline{\sigma}(\Rb_{1:s})^4}
		f(\epsilon,\eta)^2. \label{eqn:subopt last}
	\end{align}
	What remains is to show that $\vKhat_{1:\numSys}$ stabilizes the true system $\xbb_{t+1} = \big(\Abs_{\omega(t)} + \Bbs_{\omega(t)}\vKhat_{\omega(t)}\big)\xbb_t$ in the mean square sense. For this purpose, we define $\vLhat_i := \Abs_i+\Bbs_i\vKhat_i = \Lbs_i + \Bbs_i\mtx{\Delta}_{\vKhat_i}$, where $\mtx{\Delta}_{\vKhat_i} = \vKhat_i-\Kbs_i$. Then, we have $\vLhat_i^\T \otimes\vLhat_i^\T = \Lb^{\star \T}_i \otimes \Lb^{\star \T}_i + \Lb^{\star \T}_i \otimes \mtx{\Delta}_{\vKhat_i}^\T\Bb^{\star \T}_i + \mtx{\Delta}_{\vKhat_i}^\T\Bb^{\star \T}_i \otimes \Lb^{\star \T}_i +  \mtx{\Delta}_{\vKhat_i}^\T\Bb^{\star \T}_i \otimes \mtx{\Delta}_{\vKhat_i}^\T\Bb^{\star \T}_i$. Using these, we define the following matrix,
	\begin{equation}
		\hat{\Lb} :=  
		\begin{bmatrix}
			[\vT^\star]_{11} \vLhat_1^\T \otimes \vLhat_1^\T & [\vT^\star]_{12} \vLhat_1^\T \otimes \vLhat_1^\T & \cdots & [\vT^\star]_{1\numSys} \vLhat_1^\T \otimes \vLhat_1^\T \\
			[\vT^\star]_{21} \vLhat_2^\T \otimes \vLhat_2^\T & [\vT^\star]_{22} \vLhat_2^\T \otimes \vLhat_2^\T & \cdots & [\vT^\star]_{2\numSys} \vLhat_2^\T \otimes \vLhat_2^\T \\
			\vdots & \vdots & \vdots & \vdots \\
			[\vT^\star]_{\numSys 1} \vLhat_\numSys^\T \otimes \vLhat_\numSys^\T & [\vT^\star]_{\numSys 2} \vLhat_\numSys^\T \otimes \vLhat_\numSys^\T & \cdots & [\vT^\star]_{\numSys\numSys} \vLhat_\numSys^\T \otimes \vLhat_\numSys^\T \\
		\end{bmatrix}.
	\end{equation}
	To show that $\vKhat_{1:\numSys}$ stabilizes the system $\xbb_{t+1} = \big(\Abs_{\omega(t)} + \Bbs_{\omega(t)}\vKhat_{\omega(t)}\big)\xbb_t$ in the mean square sense it suffices to show that $\rho(\vLhat)<1$ due to Proposition~\ref{prop_riccatiPertb}. For this purpose, we use the Fact~\ref{fact:rho perturbation} and the obervation that, for any block matrix $\Mb$ with blocks $\Mb_{ij}$, $\norm{\Mb}^2 \leq \sum_{i,j}\norm{\Mb_{ij}}^2$,  to obtain
	\begin{align}
		\norm{\vLhat^k} &= \norm{(\tilde{\Lb}^\star + \mtx{\Delta})^k} \leq \tau(\tilde{\Lb}^\star,\gamma)\big(\tau(\tilde{\Lb}^\star,\gamma)\norm{\mtx{\Delta}}+\gamma
		\big)^k, \nn \\
		\text{where} \quad \norm{\mtx{\Delta}} &= \sum_{i=1}^s\big( 2\norm{\Abs_i+ \Bbs_i\Kbs_i}\norm{\Bbs_i}\norm{\Kbs_i - \vKhat_i}+ \norm{\Bbs_i}^2\norm{\Kbs_i - \vKhat_i}^2\big), \nn \\
		&\leq s\big(2\Gamma_\star^3\max_{i\in[s]}\norm{\Kbs_i - \vKhat_i}+ \Gamma_\star^2\max_{i\in[s]}\norm{\Kbs_i - \vKhat_i}^2\big), \nn \\
		&\leqsym{a} 3s\Gamma_\star^3\max_{i\in[s]}\norm{\Kbs_i - \vKhat_i}, \nn \\
		&\leqsym{b} 90s\Gamma_\star^6\frac{(\underline{\sigma}(\Rb_i)+\Gamma_\star^3)}{\underline{\sigma}(\Rb_i)^2}f(\epsilon,\eta), \nn \\
		&\leqsym{c} \frac{1-\gamma}{2\tau(\tilde{\Lb}^\star,\gamma)},
	\end{align}
	where, we get (b) from the derived bound on $\norm{\Kbs_i - \vKhat_i}$, while (a) and (c) from the assumption that $f(\epsilon,\eta) \leq \frac{(1-\gamma)\underline{\sigma}(\Rbs_i)^2}{180s \Gamma_\star^6(\underline{\sigma}(\Rbs_i)+\Gamma_\star^3)\tau(\tilde{\Lb}^\star,\gamma)}.$
	This implies $\norm{\vLhat^k} \leq \tau(\tilde{\Lb}^\star,\gamma)\big(\frac{1+\gamma}{2}\big)^k$, which implies $\rho(\vLhat) \leq (1+\gamma)/2 <1$. To summarize, we showed that when the optimal controller $\Kbs_{1:s}$ stabilizes the noiseless system $\xbb_{t+1} = \Abs_{\omega(t)}\xbb_t + \Bbs_{\omega(t)}\ub_t$ in the mean-square sense and $f(\epsilon,\eta)$ satisfies the above inequality then the certainty equivalent controller $\vKhat_{1:s}$ also stabilizes the true system in the mean-square sense. Lastly, we observe that $\norm{\bSi^{\vKhat}} \leq \sum_{t=0}^{\infty} \norm{\E[\xbb_t\xbb_t^\T]} \leq \sum_{t=0}^{\infty} \E[\tn{\xbb_t}^2] \leq \sum_{t=0}^{\infty} \tau(\tilde{\Lb}^\star,\gamma) \big(\frac{1+\gamma}{2}\big)^t \tn{\xbb_{0}}^2 \leq \frac{2\tau(\tilde{\Lb}^\star,\gamma)}{1-\gamma}$ for some $\tau(\tilde{\Lb}^\star,\gamma) \geq 1$ and $\gamma < 1$. Combining this with \eqref{eqn:subopt last}, we get the advertised suboptimality bound. This completes the proof.
\end{proof}

\subsection{Proof of Lemma \ref{lemma_3}} \label{subsec_lcss_1}

%{\color{red} Bound $\norm{\Tcal^\inv}$.} 
Since $\Kcal$ is defined using $\Tcal^\inv$, given in \eqref{eq:Tinv_def}, in order to bound $\Kcal$, we start by bounding $\norm{\Tcal^\inv}$.   
Let us first bound one of the factors in $\Tcal^\inv$:
\begin{equation}
	\norm{\parentheses{\Ib - \vLtil^\star}^\inv} = \normbig{\sum_{k=0}^{\infty} (\vLtil^\star)^k} \leq \sum_{k=0}^{\infty} \norm{(\vLtil^\star)^k}
	\overset{\text{(i)}}{\leq} \sum_{k=0}^{\infty} \tau(\vLtil^\star, \gamma) \gamma^k \overset{\text{(ii)}}{\leq} \frac{\tau(\vLtil^\star, \gamma)}{1- \gamma} \label{eq_7}
\end{equation}
where (i) follows from Definition \ref{def_tau}, and (ii) holds since $\rho(\vLtil^\star) \leq \gamma < 1$. Therefore, using Fact~\ref{fact:tilde_vek}, we have
\begin{equation} \label{lemma_2}
	\norm{\Tcal^\inv} \leq \norm{\tilde{\vek}^\inv} \norm{\parentheses{\Ib - \vLtil^\star}^\inv} \norm{\tilde{\vek}} \leq \sqrt{\dimSt \numSys} \frac{\tau(\vLtil^\star, \gamma)}{1 - \gamma}.
\end{equation} 

% \begin{proof}[Proof (sketch only, complete proof in Appendix \ref{subsec_lcss_1})]\todo{proof sketch inside the appendix is odd.}
% We define
% $
% \Gcal_1(\vX): = F(\vP^\star + \vX; \vA^\star, \vB^\star, \vThat) - F(\vP^\star + \vX; \vAhat, \vBhat, \vThat)$, and $\Gcal_2(\vX): = F(\vP^\star + \vX; \vA^\star, \vB^\star, \vT^\star) - F(\vP^\star + \vX; \vA^\star, \vB^\star, \vThat).
% $
% Then, we can see $\Kcal(\vX) = \Tcal^\inv(\Gcal_1(\vX) + \Gcal_2(\vX) - \Hcal(\vX))$ and $\Kcal(\vX_1) - \Kcal(\vX_2) = \Tcal^\inv(\Gcal_1(\vX_1) - \Gcal_1(\vX_2) + \Gcal_2(\vX_2) - \Gcal_2(\vX_2) - \Hcal(\vX_1) + \Hcal(\vX_2))$. We can upper bound $\norm{\Hcal(\vX)}$, $\norm{\Gcal_1(\vX)}$, $\norm{\Gcal_2(\vX)}$, $\norm{\Hcal(\vX_1) - \Hcal(\vX_2)}$, $\norm{\Gcal_1(\vX_1) - \Gcal_1(\vX_2)}$, $\norm{\Gcal_2(\vX_1) - \Gcal_2(\vX_2)}$, for any $\vX, \vX_1, \vX_2 \in \Scal_\nu$, and combine the bound for $\norm{\Tcal^\inv}$ in \eqref{lemma_2}, then we can conclude the proof for Lemma \ref{lemma_3}.
% \end{proof}

Next, to simplify the notations, for $\vX, \vX_1, \vX_2 \in \Scal_\nu$, we let $\vP^\star_\vX := \vP^\star + \vX$, $\vP^\star_{\vX_1}: = \vP^\star + \vX_1$, $\vP^\star_{\vX_2}: = \vP^\star + \vX_2$, $\vDelta_\vA := \vAhat - \vA^\star$, $\vDelta_\vB := \vBhat - \vB^\star$, and $\vDelta_\vS := \vShat - \vS^\star$. We now derive some basic relations that will be used frequently later.

Recall by definition $\vX = \diag (\vX_{1:\numSys})$ and $\vPhi^\star(\vX) = \diag(\varphi^\star_{1:\numSys}(\vX_{1:\numSys}))$, so
\\
\begin{equation}\label{eq_20}
	\norm{\vPhi^\star(\vX)} 
	=
	\max_{i \in [\numSys]} \norm{\varphi^\star_i (\vX_{1:\numSys})}
	=
	\max_{i \in [\numSys]} \norm{\sum_{j \in [\numSys]} [\vT^\star]_{ij} \vX_j}
	\leq
	\norm{\vX}.
\end{equation}
Similarly, we have $\norm{\hat{\vPhi}(\vX)} \leq \norm{\vX}$. Furthermore,
\begin{equation} \label{eq_27}
	\begin{split}
		\norm{\vPhi^\star(\vX) - \hat{\vPhi}(\vX)}
		&= \max_{i \in [\numSys]} \norm{\varphi^\star_i(\vX_{1:\numSys}) - \hat{\varphi}_i(\vX_{1:\numSys})} \\
		&= \max_{i \in [\numSys]} \norm{\sum_{j \in [\numSys]} ([\vT^\star]_{ij} - [\vThat]_{ij}) \vX_j} \\
		&\leq \eta \norm{\vX}.
	\end{split}	
\end{equation} 
For any $\vX \in \Scal_\nu$, we have
\begin{equation}\label{eq_21}
	\norm{\vP^\star_\vX} \leq \norm{\vP^\star} + \nu \leq \norm{\vP^\star}_+,
\end{equation}
where we used assumption $\nu \leq 1$ in the statement of Lemma \ref{lemma_3}. Combining \eqref{eq_20} and \eqref{eq_21}, we have
\begin{equation}\label{eq_22}
	\max \curlybracketsbig{\norm{\vPhi^\star(\vP^\star_\vX)}, \norm{\hat{\vPhi}(\vP^\star_\vX)}} \leq \norm{\vP^\star}_+.
\end{equation}
Consider $\vDelta_\vS$, since $\vDelta_\vS = \vShat - \vS^\star = \vBhat \vR^\inv \vBhat^\T - \vB^\star \vR^\inv {\vB^\star}^\T = \vDelta_\vB \vR^\inv {\vB^\star}^\T + \vB^\star \vR^\inv \vDelta_\vB^\T + \vDelta_\vB \vR^\inv \vDelta_\vB^\T$ and assumption $ \norm{\vDelta_\vB} \leq 
\epsilon \leq \norm{\vB^\star}$ in the statement of Lemma \ref{lemma_3}, we have
\begin{equation}\label{eq_23}
	\norm{\vS^\star} \leq \norm{\vB^\star}^2 \norm{\vR^\inv}, \quad 
	\norm{\vDelta_\vS} \leq 3 \norm{\vB^\star} \norm{\vR^\inv} \epsilon, \quad
	\norm{\vShat} \leq 4 \norm{\vB^\star}^2 \norm{\vR^\inv}.
\end{equation}
Following Lemma \ref{lemma_4}, \eqref{eq_22}, and \eqref{eq_23}, we have
\begin{align}
	\max \curlybracketsbig{\norm{(\vI + \vS^\star \vPhi^\star(\vP^\star_\vX))^\inv}, \norm{(\vI + \vS^\star \hat{\vPhi}(\vP^\star_\vX))^\inv}} \leq 1+ \norm{\vS^\star} \norm{\vP^\star}_+ \leq \norm{\vB^\star}_+^2 \norm{\vR^\inv}_+ \norm{\vP^\star}_+, \label{eq_24} \\
	\max \curlybracketsbig{\norm{(\vI + \vShat \vPhi^\star(\vP^\star_\vX))^\inv}, \norm{(\vI + \vShat \hat{\vPhi}(\vP^\star_\vX))^\inv}} \leq 1+ \norm{\vShat} \norm{\vP^\star}_+ \leq 4 \norm{\vB^\star}_+^2 \norm{\vR^\inv}_+ \norm{\vP^\star}_+. \label{eq_25}
\end{align}

Finally, recall the definition of $F(\vX; \vA^\star, \vB^\star, \vT^\star)$ in \eqref{eq_1}, and consider the following notation:
\begin{align}
	\Gcal_1(\vX): = F(\vP^\star_\vX; \vA^\star, \vB^\star, \vThat) - F(\vP^\star_\vX; \vAhat, \vBhat, \vThat), \\
	\Gcal_2(\vX): = F(\vP^\star_\vX; \vA^\star, \vB^\star, \vT^\star) - F(\vP^\star_\vX; \vA^\star, \vB^\star, \vThat).
\end{align}

Now we are ready to start the main proof for Lemma \ref{lemma_3}. We will do this in two steps: (a) Proof of \eqref{eq_12}. (b) Proof of \eqref{eq_13}.

\noindent \textbf{Step (a):} Proof of the bound \eqref{eq_12} in Lemma \ref{lemma_3}.

We can see from the definition of $\Kcal(\vX)$ in \eqref{eq:defKcal}:
\begin{align}
	\Kcal(\vX) &= \Tcal^\inv(\Gcal_1(\vX) + \Gcal_2(\vX) - \Hcal(\vX)). \label{eq_26}
\end{align}
We will upper bound $\norm{\Hcal(\vX)}$, $\norm{\Gcal_1(\vX)}$, and $\norm{\Gcal_2(\vX)}$ for any $\vX \in \Scal_\nu$, then combining these with the bound for $\norm{\Tcal^\inv}$ in \eqref{lemma_2}, we can conclude step (a).

Since $\Hcal(\vX) = {\vL^\star}^\T \vPhi^\star(\vX) (\vI + \vS^\star \vPhi^\star(\vP^\star) + \vS^\star \vPhi^\star(\vX))^\inv \vS^\star \vPhi^\star(\vX) \vL^\star$ in \eqref{lemma_1}, we have
\begin{equation}\label{eq_28}
	\norm{\Hcal(\vX)} \leq \norm{\vL^\star}^2 \norm{\vS^\star} \norm{\vX}^2 \leq \norm{\vL^\star}^2 \norm{\vS^\star} \nu^2,
\end{equation}
where \eqref{eq_17} and \eqref{eq_20} are used. Now consider $\Gcal_1(\vX)$
\begin{equation}
	\begin{split}
		\Gcal_1(\vX)
		&= F(\vP^\star_\vX; \vA^\star, \vB^\star, \vThat) - F(\vP^\star_\vX; \vAhat, \vBhat, \vThat) \\
		&\overset{\text{(i)}}{=} -{\vA^\star}^\T \hat{\vPhi}(\vP^\star_\vX)(\vI + \vS^\star \hat{\vPhi}(\vP^\star_\vX))^\inv \vA^\star + (\vA^\star + \vDelta_\vA)^\T \hat{\vPhi}(\vP^\star_\vX)(\vI + \vShat \hat{\vPhi}(\vP^\star_\vX))^\inv (\vA^\star + \vDelta_\vA) \\
		&\overset{\text{(ii)}}{=} -{\vA^\star}^\T \hat{\vPhi}(\vP^\star_\vX)(\vI + \vS^\star \hat{\vPhi}(\vP^\star_\vX))^\inv \vDelta_\vS \hat{\vPhi}(\vP^\star_\vX) (\vI + \vShat \hat{\vPhi}(\vP^\star_\vX))^\inv \vA^\star \\
		& \ \ \ \ \ + \vDelta_\vA^\T \hat{\vPhi}(\vP^\star_\vX) (\vI + \vShat \hat{\vPhi}(\vP^\star_\vX))^\inv \vA^\star + {\vA^\star}^\T \hat{\vPhi}(\vP^\star_\vX) (\vI + \vShat \hat{\vPhi}(\vP^\star_\vX))^\inv \vDelta_\vA \\
		& \ \ \ \ \ + \vDelta_\vA^\T \hat{\vPhi}(\vP^\star_\vX) (\vI + \vShat \hat{\vPhi}(\vP^\star_\vX))^\inv \vDelta_\vA,\\
	\end{split}
\end{equation}
where (i) follows from the definition of $F$ in \eqref{eq_1}, and (ii) uses \eqref{eq_19}. Then,
\begin{equation}\label{eq_29}
	\begin{split}
		\norm{\Gcal_1(\vX)} 
		&\leq \norm{\vA^\star}^2 \norm{\hat{\vPhi}(\vP^\star_\vX)}^2 \norm{\vDelta_\vS} + 2\norm{\hat{\vPhi}(\vP^\star_\vX)} \norm{\vA^\star} \epsilon + \norm{\hat{\vPhi}(\vP^\star_\vX)} \epsilon^2\\
		&\leq 3 \norm{\vA^\star}^2 \norm{\vP^\star}_+^2 \norm{\vB^\star} \norm{\vR^\inv} \epsilon + 2 \norm{\vP^\star}_+ \norm{\vA^\star} \epsilon + \norm{\vB^\star} \norm{\vP^\star}_+ \epsilon\\
		&\leq 3 \norm{\vA^\star}_+^2 \norm{\vB^\star}_+ \norm{\vP^\star}_+^2 \norm{\vR^\inv}_+ \epsilon,
	\end{split}
\end{equation}
where \eqref{eq_17}, \eqref{eq_22}, \eqref{eq_23}, and assumption $\epsilon \leq \norm{\vB^\star}$ in the statement of Lemma \ref{lemma_3} are used. Now consider $\Gcal_2(\vX)$.
\begin{equation}
	\begin{split}
		\Gcal_2(\vX) 
		&= F(\vP^\star_\vX; \vA^\star, \vB^\star, \vT^\star) - F(\vP^\star_\vX; \vA^\star, \vB^\star, \vThat) \\
		&= - {\vA^\star}^\T \vPhi^\star(\vP^\star_\vX) (\vI + \vS^\star \vPhi^\star(\vP^\star_\vX))^\inv \vA^\star + {\vA^\star}^\T \hat{\vPhi}(\vP^\star_\vX) (\vI + \vS^\star \hat{\vPhi}(\vP^\star_\vX))^\inv \vA^\star \\
		&= - {\vA^\star}^\T \vPhi^\star(\vP^\star_\vX) \squarebracketsbig{(\vI + \vS^\star \vPhi^\star(\vP^\star_\vX))^\inv -  (\vI + \vS^\star \hat{\vPhi}(\vP^\star_\vX))^\inv} \vA^\star+ {\vA^\star}^\T (\hat{\vPhi}(\vP^\star_\vX) - \vPhi^\star(\vP^\star_\vX)) (\vI + \vS^\star \hat{\vPhi}(\vP^\star_\vX))^\inv \vA^\star \\
		&= - {\vA^\star}^\T \vPhi^\star(\vP^\star_\vX) (\vI + \vS^\star \vPhi^\star(\vP^\star_\vX))^\inv \vS^\star (\hat{\vPhi}(\vP^\star_\vX) - \vPhi^\star(\vP^\star_\vX)) (\vI + \vS^\star \hat{\vPhi}(\vP^\star_\vX))^\inv \vA^\star \\ 
		& \ \ \ \ \ + {\vA^\star}^\T (\hat{\vPhi}(\vP^\star_\vX) - \vPhi^\star(\vP^\star_\vX)) (\vI + \vS^\star \hat{\vPhi}(\vP^\star_\vX))^\inv \vA^\star \\
		&= - {\vA^\star}^\T \squarebracketsbig{\vPhi^\star(\vP^\star_\vX) (\vI + \vS^\star \vPhi^\star(\vP^\star_\vX))^\inv \vS^\star - \vI} (\hat{\vPhi}(\vP^\star_\vX) - \vPhi^\star(\vP^\star_\vX)) (\vI + \vS^\star \hat{\vPhi}(\vP^\star_\vX))^\inv \vA^\star,
	\end{split}
\end{equation}
where \eqref{eq_19} is used. With some algebra, we can bound the norm of $\Gcal_2(\vX)$:
\begin{equation}\label{eq_30}
	\begin{split}
		\norm{\Gcal_2(\vX)}
		&\leq \norm{\vA^\star}^2 (\norm{\vPhi^\star(\vP^\star_\vX)} \norm{\vS^\star} + 1) \norm{\hat{\vPhi}(\vP^\star_\vX) - \vPhi^\star(\vP^\star_\vX)} \norm{(\vI + \vS^\star \hat{\vPhi}(\vP^\star_\vX))^\inv}\\
		&\leq \norm{\vA^\star}^2 (\norm{\vP^\star}_+ \norm{\vB^\star}^2 \norm{\vR^\inv} + 1) \eta \norm{\vP^\star}_+ \norm{\vB^\star}_+^2 \norm{\vR^\inv}_+ \norm{\vP^\star}_+\\\
		&\leq \norm{\vA^\star}_+^2 \norm{\vB^\star}_+^4 \norm{\vP^\star}_+^3 \norm{\vR^\inv}_+^2 \eta,
	\end{split}
\end{equation}
where \eqref{eq_17}, \eqref{eq_22}, \eqref{eq_23}, \eqref{eq_27}, \eqref{eq_21}, and \eqref{eq_24} are used.

Using upper bounds for $\norm{\Tcal^\inv}$, $\norm{\Hcal(\vX)}$, $\norm{\Gcal_1(\vX)}$, and $\norm{\Gcal_2(\vX)}$ in \eqref{lemma_2}, \eqref{eq_28}, \eqref{eq_29}, and \eqref{eq_30}, then the relation in \eqref{eq_26} gives
\begin{equation}
	\norm{\Kcal(\vX)} \leq \frac{\sqrt{\dimSt \numSys} \tau (\vLtil^\star, \gamma)}{1-\gamma} \parenthesesbig{ \norm{\vL^\star}^2 \norm{\vS^\star} \nu^2 + 3 \norm{\vA^\star}_+^2 \norm{\vB^\star}_+ \norm{\vP^\star}_+^2 \norm{\vR^\inv}_+ \epsilon + \norm{\vA^\star}_+^2 \norm{\vB^\star}_+^4 \norm{\vP^\star}_+^3 \norm{\vR^\inv}_+^2 \eta },
\end{equation}
which shows \eqref{eq_12} in Lemma \ref{lemma_3}.

\noindent \textbf{Step (b):} Proof of the bound \eqref{eq_13} in Lemma \ref{lemma_3}.
Following as in step (a) and invoking the linearity of $\Tcal^\inv$, we have
\begin{align}
	\Kcal(\vX_1) - \Kcal(\vX_2) &= \Tcal^\inv(\Gcal_1(\vX_1) - \Gcal_1(\vX_2) + \Gcal_2(\vX_2) - \Gcal_2(\vX_2) - \Hcal(\vX_1) + \Hcal(\vX_2)) \label{eq_31}.
\end{align} 
We will upper bound $\norm{\Hcal(\vX_1) - \Hcal(\vX_2)}$, $\norm{\Gcal_1(\vX_1) - \Gcal_1(\vX_2)}$, and $\norm{\Gcal_2(\vX_1) - \Gcal_2(\vX_2)}$ for any $\vX_1, \vX_2 \in \Scal_\nu$, then combining these with the bound for $\norm{\Tcal^\inv}$ in \eqref{lemma_2}, we can prove step (b).

First we consider $\Hcal(\vX_1) - \Hcal(\vX_2)$.
\begin{equation}
	\begin{split}
		&\Hcal(\vX_1) - \Hcal(\vX_2)\\
		=& {\vL^\star}^\T \vPhi^\star(\vX_1) (\vI + \vS^\star \vPhi^\star(\vP^\star_{\vX_1}))^\inv \vS^\star \vPhi^\star(\vX_1) \vL^\star - {\vL^\star}^\T \vPhi^\star(\vX_2) (\vI + \vS^\star \vPhi^\star(\vP^\star_{\vX_2}))^\inv \vS^\star \vPhi^\star(\vX_2) \vL^\star \\
		=& {\vL^\star}^\T \vPhi^\star(\vX_1) \squarebracketsbig{(\vI + \vS^\star \vPhi^\star(\vP^\star_{\vX_1}))^\inv - (\vI + \vS^\star \vPhi^\star(\vP^\star_{\vX_2}))^\inv} \vS^\star \vPhi^\star(\vX_1) \vL^\star \\
		& \ \ \ \ \ - {\vL^\star}^\T \vPhi^\star(\vX_2 - \vX_1) (\vI + \vS^\star \vPhi^\star(\vP^\star_{\vX_2}))^\inv \vS^\star \vPhi^\star(\vX_2) \vL^\star - {\vL^\star}^\T \vPhi^\star(\vX_1) (\vI + \vS^\star \vPhi^\star(\vP^\star_{\vX_2}))^\inv \vS^\star \vPhi^\star(\vX_2 - \vX_1) \vL^\star \\
		\overset{\eqref{eq_19}}{=}& {\vL^\star}^\T \vPhi^\star(\vX_1) (\vI + \vS^\star \vPhi^\star(\vP^\star_{\vX_1}))^\inv \vS^\star \vPhi^\star(\vX_2 - \vX_1) (\vI + \vS^\star \vPhi^\star(\vP^\star_{\vX_2}))^\inv \vS^\star \vPhi^\star(\vX_1) \vL^\star \\
		& \ \ \ \ \ - {\vL^\star}^\T \vPhi^\star(\vX_2 - \vX_1) (\vI + \vS^\star \vPhi^\star(\vP^\star_{\vX_2}))^\inv \vS^\star \vPhi^\star(\vX_2) \vL^\star - {\vL^\star}^\T \vPhi^\star(\vX_1) (\vI + \vS^\star \vPhi^\star(\vP^\star_{\vX_2}))^\inv \vS^\star \vPhi^\star(\vX_2 - \vX_1) \vL^\star.
	\end{split}
\end{equation}
Then,
\begin{equation}\label{eq_32}
	\begin{split}
		&\norm{\Hcal(\vX_1) - \Hcal(\vX_2)} \\
		\leq & \norm{\vL^\star}^2 \norm{\vX_1} \norm{\vS^\star} \norm{\vX_2 - \vX_1} \norm{\vS^\star} \norm{\vX_1} + \norm{\vL^\star}^2 \norm{\vX_2 - \vX_1} \norm{\vS^\star} \norm{\vX_2} + \norm{\vL^\star}^2 \norm{\vX_1} \norm{\vS^\star} \norm{\vX_2-\vX_1} \\
		\leq& \norm{\vL^\star}^2 \norm{\vS^\star}^2 \nu^2 \norm{\vX_2 - \vX_1} + 2 \norm{\vL^\star}^2 \norm{\vS^\star} \nu \norm{\vX_2 - \vX_1} \\
		\leq& 3 \norm{\vL^\star}^2 \norm{\vS^\star} \nu \norm{\vX_2 - \vX_1}.
	\end{split}
\end{equation}
The first inequality follows from $\norm{\vN (\vI + \vM \vN)^\inv} \leq \norm{\vN}$ in \eqref{eq_17}, and $\norm{\vPhi^\star(\vX)} \leq \norm{\vX}$ in \eqref{eq_20}. The last inequality can be obtained by recalling the assumption $\nu \leq \norm{\vS^\star}^\inv$ in the statement of Lemma \ref{lemma_3}. 

For $\Gcal_1(\vX_1) - \Gcal_1(\vX_2)$, we have
\begin{equation}\label{eq_35}
	\begin{split}
		\Gcal_1(\vX_1) &- \Gcal_1(\vX_2) \\
		=& F(\vP^\star_{\vX_1}; \vA^\star, \vB^\star, \vThat) - F(\vP^\star_{\vX_1}; \vAhat, \vBhat, \vThat) - F(\vP^\star_{\vX_2}; \vA^\star, \vB^\star, \vThat) + F(\vP^\star_{\vX_2}; \vAhat, \vBhat, \vThat)\\
		=& - {\vA^\star}^\T \hat{\vPhi}(\vP^\star_{\vX_1}) (\vI + \vS^\star \hat{\vPhi}(\vP^\star_{\vX_1}))^\inv \vA^\star + (\vA^\star + \vDelta_\vA)^\T \hat{\vPhi}(\vP^\star_{\vX_1}) (\vI + \vShat \hat{\vPhi}(\vP^\star_{\vX_1}))^\inv (\vA^\star + \vDelta_\vA) \\
		\quad & + {\vA^\star}^\T \hat{\vPhi}(\vP^\star_{\vX_2}) (\vI + \vS^\star \hat{\vPhi}(\vP^\star_{\vX_2}))^\inv \vA^\star - (\vA^\star + \vDelta_\vA)^\T \hat{\vPhi}(\vP^\star_{\vX_2}) (\vI + \vShat \hat{\vPhi}(\vP^\star_{\vX_2}))^\inv (\vA^\star + \vDelta_\vA) \\
		\overset{\eqref{eq_19}}{=}& - {\vA^\star}^\T \underbrace{\hat{\vPhi}(\vP^\star_{\vX_1}) (\vI + \vS^\star \hat{\vPhi}(\vP^\star_{\vX_1}))^\inv}_{:=\vM_1} \vDelta_\vS \underbrace{\hat{\vPhi}(\vP^\star_{\vX_1}) (\vI + \vShat \hat{\vPhi}(\vP^\star_{\vX_1}))^\inv}_{:=\vMhat_1} \vA^\star \\
		\quad& + {\vA^\star}^\T \underbrace{\hat{\vPhi}(\vP^\star_{\vX_2}) (\vI + \vS^\star \hat{\vPhi}(\vP^\star_{\vX_2}))^\inv}_{:=\vM_2} \vDelta_\vS \underbrace{\hat{\vPhi}(\vP^\star_{\vX_2}) (\vI + \vShat \hat{\vPhi}(\vP^\star_{\vX_2}))^\inv}_{:=\vMhat_2} \vA^\star \\
		\quad& + \vDelta_\vA^\T \hat{\vPhi}(\vP^\star_{\vX_1}) (\vI + \vShat \hat{\vPhi}(\vP^\star_{\vX_1}))^\inv \vAhat - \vDelta_\vA^\T \hat{\vPhi}(\vP^\star_{\vX_2}) (\vI + \vShat \hat{\vPhi}(\vP^\star_{\vX_2}))^\inv \vAhat\\
		\quad& + {\vA^\star}^\T \hat{\vPhi}(\vP^\star_{\vX_1}) (\vI + \vShat \hat{\vPhi}(\vP^\star_{\vX_1}))^\inv \vDelta_\vA - {\vA^\star}^\T \hat{\vPhi}(\vP^\star_{\vX_2}) (\vI + \vShat \hat{\vPhi}(\vP^\star_{\vX_2}))^\inv \vDelta_\vA\\
		=& -{\vA^\star}^\T ((\vM_1 - \vM_2) \vDelta_\vS \vMhat_1 + \vM_2 \vDelta_\vS (\vMhat_1-\vMhat_2) ) \vA^\star + \vDelta_\vA^\T (\vMhat_1 - \vMhat_2) \vAhat + {\vA^\star}^\T (\vMhat_1 - \vMhat_2) \vDelta_\vA.
	\end{split}
\end{equation}
For $\vM_1 - \vM_2$, we have
\begin{equation}
	\begin{split}
		\vM_1 - \vM_2
		&= \hat{\vPhi}(\vP^\star_{\vX_1}) \squarebracketsbig{(\vI + \vS^\star \hat{\vPhi}(\vP^\star_{\vX_1}))^\inv - (\vI + \vS^\star \hat{\vPhi}(\vP^\star_{\vX_2}))^\inv} 
		+ \hat{\vPhi}(\vX_1 - \vX_2) (\vI + \vS^\star \hat{\vPhi}(\vP^\star_{\vX_2}))^\inv \\
		&\overset{\eqref{eq_19}}{=} \hat{\vPhi}(\vP^\star_{\vX_1}) (\vI + \vS^\star \hat{\vPhi}(\vP^\star_{\vX_1}))^\inv \vS^\star \hat{\vPhi}(\vX_2 - \vX_1) (\vI + \vS^\star \hat{\vPhi}(\vP^\star_{\vX_2}))^\inv
		+ \hat{\vPhi}(\vX_1 - \vX_2) (\vI + \vS^\star \hat{\vPhi}(\vP^\star_{\vX_2}))^\inv \\
		&=\squarebracketsbig{\hat{\vPhi}(\vP^\star_{\vX_1}) (\vI + \vS^\star \hat{\vPhi}(\vP^\star_{\vX_1}))^\inv \vS^\star - \vI} \hat{\vPhi}(\vX_2 - \vX_1) (\vI + \vS^\star \hat{\vPhi}(\vP^\star_{\vX_2}))^\inv.
	\end{split}
\end{equation}
Then,
\begin{equation}\label{eq_36}
	\begin{split}
		\norm{\vM_1 - \vM_2}
		\leq& (\norm{\vP^\star}_+ \norm{\vB^\star}^2 \norm{\vR^\inv} + 1) \norm{\vX_2 - \vX_1} \norm{\vB^\star}_+^2 \norm{\vR^\inv}_+ \norm{\vP^\star}_+ \\
		\leq& \norm{\vB^\star}_+^4 \norm{\vP^\star}_+^2 \norm{\vR^\inv}_+^2 \norm{\vX_2 - \vX_1}.
	\end{split}
\end{equation}
The first line is obtained by invoking: (i) $\norm{\vN(\vI + \vM \vN)^\inv} \leq \norm{\vN}$ in \eqref{eq_17}, (ii) $\norm{\hat{\vPhi}(\vX)} \leq \norm{\vX}$ in \eqref{eq_20}, (iii) $\norm{\hat{\vPhi}(\vP_\vX^\star)} \leq \norm{\vP^\star}_+$ in \eqref{eq_22}, (iv) $\norm{\vS^\star} \leq \norm{\vB^\star}^2 \norm{\vR^\inv}$ in \eqref{eq_23}, and (v) $\norm{(\vI + \vS^\star \hat{\vPhi}(\vP_\vX^\star))^\inv} \leq \norm{\vB^\star}_+^2 \norm{\vR^\inv}_+ \norm{\vP^\star}_+.$ in \eqref{eq_24}. Now using  \eqref{eq_17}, \eqref{eq_22}, \eqref{eq_23}, \eqref{eq_20}, and \eqref{eq_25} similarly, we have
\begin{equation}\label{eq_37}
	\begin{split}
		\norm{\vMhat_1 - \vMhat_2}
		\leq& \normbig{\squarebracketsbig{\hat{\vPhi}(\vP^\star_{\vX_1}) (\vI + \vShat \hat{\vPhi}(\vP^\star_{\vX_1}))^\inv \vShat - \vI} \hat{\vPhi}(\vX_2 - \vX_1) (\vI + \vShat \hat{\vPhi}(\vP^\star_{\vX_2}))^\inv} \\
		\leq& (\norm{\vP^\star}_+ 4 \norm{\vB^\star}^2 \norm{\vR^\inv} + 1) \cdot \norm{\vX_2 - \vX_1} \cdot 4 \norm{\vB^\star}_+^2 \norm{\vR^\inv}_+ \norm{\vP^\star}_+ \\
		\leq& 16 \norm{\vB^\star}_+^4 \norm{\vP^\star}_+^2 \norm{\vR^\inv}_+^2 \norm{\vX_2 - \vX_1}.
	\end{split}
\end{equation}
Through \eqref{eq_17} and \eqref{eq_22}, we can have
\begin{equation}\label{eq_38}
	\norm{\vMhat_1} \leq \norm{\vP^\star}_+, \quad \norm{\vM_2} \leq \norm{\vP^\star}_+.
\end{equation}
Plugging \eqref{eq_36}, \eqref{eq_37}, and \eqref{eq_38} into \eqref{eq_35} gives
\begin{equation}\label{eq_33}
	\begin{split}
		\norm{\Gcal_1(\vX_1) - \Gcal_1(\vX_2)}
		\leq& 51 \norm{\vA^\star}^2 \norm{\vB^\star}_+^4 \norm{\vB^\star} \norm{\vP^\star}_+^3 \norm{\vR^\inv}_+^2 \norm{\vR^\inv} \norm{\vX_2-\vX_1} \epsilon \\
		\quad& +16 (\norm{\vA^\star} + \epsilon) \norm{\vB^\star}_+^4 \norm{\vP^\star}_+^2 \norm{\vR^\inv}_+^2  \norm{\vX_2-\vX_1} \epsilon \\
		\quad& +16 \norm{\vA^\star} \norm{\vB^\star}_+^4 \norm{\vP^\star}_+^2 \norm{\vR^\inv}_+^2  \norm{\vX_2-\vX_1} \epsilon \\
		\leq & 51 \norm{\vA^\star}_+^2 \norm{\vB^\star}_+^5 \norm{\vP^\star}_+^3 \norm{\vR^\inv}_+^3 \norm{\vX_2-\vX_1} \epsilon,
	\end{split}
\end{equation}
where we additionally use the fact $\norm{\vDelta_\vS} \leq 3 \norm{\vB^\star} \norm{\vR^\inv} \epsilon$ in \eqref{eq_23} and assumption $\epsilon\leq 1$ in the statement of Lemma \ref{lemma_3}. %This proves \eqref{eq_33}.

For $\Gcal_2(\vX_1) - \Gcal_2(\vX_2)$, following the same strategy, we can obtain the following bound:
\begin{equation}\label{eq_34}
	\norm{\Gcal_2(\vX_1) - \Gcal_2(\vX_2)}
	\leq 2 \norm{\vA^\star}_+^2 \norm{\vB^\star}_+^6 \norm{\vP^\star}_+^3 \norm{\vR^\inv}_+^3 \norm{\vX_2 - \vX_1} \eta.
\end{equation}
%which proves \eqref{eq_34} and concludes the proof of Lemma \ref{lemma_5}.

% For brevity, the upper bounds for $\norm{\Gcal_1(\vX_1) - \Gcal_1(\vX_2)}$ and $\norm{\Gcal_2(\vX_1) - \Gcal_2(\vX_2)}$ are enclosed in the following lemma.
% \begin{lemma}[proof in Appendix \ref{subsec_4}] \label{lemma_5}
% 	Under the same assumptions as Lemma \ref{lemma_3}, we have
% 	\begin{align}
% 		\norm{\Gcal_1(\vX_1) - \Gcal_1(\vX_2)} &\leq 51 \norm{\vA^\star}_+^2 \norm{\vB^\star}_+^5 \norm{\vP^\star}_+^3 \norm{\vR^\inv}_+^3 \epsilon \norm{\vX_2 - \vX_1} \label{eq_33}.\\
% 		\norm{\Gcal_2(\vX_1) - \Gcal_2(\vX_2)} &\leq 2 \norm{\vA^\star}_+^2 \norm{\vB^\star}_+^6 \norm{\vP^\star}_+^3 \norm{\vR^\inv}_+^3  \eta \norm{\vX_2 - \vX_1} \label{eq_34}.
% 	\end{align}
% \end{lemma}

Using upper bounds for $\norm{\Tcal^\inv}$, $\norm{\Hcal(\vX_1) - \Hcal(\vX_2)}$, $\norm{\Gcal_1(\vX_1) - \Gcal_1(\vX_2)}$, and $\norm{\Gcal_2(\vX_1) - \Gcal_2(\vX_2)}$ in \eqref{lemma_2}, \eqref{eq_32}, \eqref{eq_33}, and \eqref{eq_34}, then relation in \eqref{eq_31} gives
\begin{multline}
	\norm{\Kcal(\vX_1) - \Kcal(\vX_2)} \leq \\
	\frac{\sqrt{ns} \tau (\vLtil^\star, \gamma)}{1-\gamma} \parenthesesbig{ 3 \norm{\vL^\star}^2 \norm{\vS^\star} \nu + 51 \norm{\vA^\star}_+^2 \norm{\vB^\star}_+^5 \norm{\vP^\star}_+^3 \norm{\vR^\inv}^3_+ \epsilon + 2 \norm{\vA^\star}_+^2 \norm{\vB^\star}_+^6 \norm{\vP^\star}_+^3 \norm{\vR^\inv}^3_+ \eta} \norm{\vX_1 - \vX_2}.
\end{multline}
which shows \eqref{eq_13} in Lemma \ref{lemma_3} and concludes the proof.

\onecolumn
% \input{sec/proofs_lemmaforLQR_zhe}
%\subsection{Proof of Theorem \ref{thrm:meta}}
%\input{sec/proofs_metaTheorem}
\end{document}